\documentclass[pdflatex,sn-mathphys-num]{sn-jnl}

\usepackage{amssymb,mathtools}
\def\rd{\mathrm{d}}
\allowdisplaybreaks

\usepackage{graphicx}%
\usepackage{multirow}%
\usepackage{amsmath,amssymb,amsfonts}%
\usepackage{amsthm}%
\usepackage{mathrsfs}%
\usepackage[title]{appendix}%
\usepackage{xcolor}%
\usepackage{textcomp}%
\usepackage{manyfoot}%
\usepackage{booktabs}%
\usepackage{algorithm}%
\usepackage{algorithmicx}%
\usepackage{algpseudocode}%
\usepackage{listings}%

\theoremstyle{thmstyleone}%
\newtheorem{theorem}{Theorem}[section]
\newtheorem{lemma}{Lemma}[section]
\newtheorem{proposition}{Proposition}[section]
\newtheorem{corollary}{Corollary}[section]

\theoremstyle{thmstyletwo}%

\newtheorem{remark}{Remark}[section]

\theoremstyle{thmstylethree}%
\newtheorem{definition}{Definition}[section]
\newtheorem{assumption}{Assumption}[section]

\begin{document}

\title[Path-dependent HJ equations]{Minimax solutions of path-dependent Hamilton--Jacobi equations under weakened assumptions with application to differential games}

\author*[1,2]{\fnm{Mikhail} \sur{Gomoyunov} \email{m.i.gomoyunov@gmail.com}}

\affil[1]{\orgname{Krasovskii Institute of Mathematics and Mechanics, Ural Branch of Russian Academy of Sciences}, \country{Russia}}

\affil[2]{\orgname{Ural Federal University}, \country{Russia}}

\abstract{We study minimax (generalized) solutions of a Cauchy problem for a (first-order) path-dependent Hamilton--Jacobi equation with co-invariant derivatives under a right-end boundary condition.
    Under assumptions on the Hamiltonian that are more general than those previously considered in the literature and allow, in particular, a measurable dependence on the first (time) variable, we establish existence, uniqueness, stability, and consistency results for minimax solutions.
    As an application, we consider a zero-sum differential game for a time-delay system and prove that this game has a value under assumptions more general than the known ones but rather natural being consistent with the Carath\'{e}odory conditions.}



\maketitle

\section{Introduction and overview of the main results}

\subsection{Path-dependent Hamilton--Jacobi equation}

    Let numbers $n \in \mathbb{N}$, $T > 0$, $h \geq 0$ be fixed.
    Let $C([- h, T], \mathbb{R}^n)$ be the Banach space of all continuous functions $x \colon [- h, T] \to \mathbb{R}^n$ with the standard norm $\|x(\cdot)\|_\infty \coloneq \max_{\tau \in [- h, T]} \|x(\tau)\|$, where $\|\cdot\|$ denotes the Euclidean norm in $\mathbb{R}^n$.
    Given $(t, x(\cdot)) \in [0, T] \times C([- h, T], \mathbb{R}^n)$, we define a function $x(\cdot \wedge t) \in C([- h, T], \mathbb{R}^n)$ by $x(\tau \wedge t) \coloneq x(\tau)$ if $\tau \in [- h, t]$ and by $x(\tau \wedge t) \coloneq x(t)$ if $\tau \in (t, T]$.
    We assume that the segment $[0, T]$ is equipped with the Lebesgue measure $\mu$ and let $\mathcal{L}_1([0, T], \mathbb{R})$ be the linear space of all measurable functions $c \colon [0, T] \to \mathbb{R}$ such that $\|c(\cdot)\|_1 \coloneq \int_{0}^{T} |c(\tau)| \, \rd \tau < + \infty$.

    We consider a {\it Cauchy problem} for a {\it path-dependent Hamilton--Jacobi equation}
    \begin{equation} \label{HJ}
        \partial_t \varphi(t, x(\cdot))
        + H \bigl( t, x(\cdot), \nabla \varphi(t, x(\cdot)) \bigr)
        = 0,
    \end{equation}
    where $(t, x(\cdot)) \in [0, T) \times C([- h, T], \mathbb{R}^n)$, under a right-end {\it boundary condition}
    \begin{equation}\label{boundary_condition}
        \varphi(T, x(\cdot))
        = \sigma(x(\cdot)),
    \end{equation}
    where $x(\cdot) \in C([- h, T], \mathbb{R}^n)$.
    In this problem, $\varphi \colon [0, T] \times C([- h, T], \mathbb{R}^n) \to \mathbb{R}$ is an {\it unknown functional}, $\partial_t \varphi(t, x(\cdot)) \in \mathbb{R}$ and $\nabla \varphi(t, x(\cdot)) \in \mathbb{R}^n$ are so-called {\it co-invariant derivatives} of this functional at a point $(t, x(\cdot))$ (see Definition \ref{definition_ci-derivatives} in Section \ref{section_minimax_solution}), the {\it Hamiltonian} $H \colon [0, T] \times C([- h, T], \mathbb{R}^n) \times \mathbb{R}^n \to \mathbb{R}$ as well as the {\it boundary functional} $\sigma \colon C([- h, T], \mathbb{R}^n) \to \mathbb{R}$ are given.

    The goal of the paper is to develop a theory of generalized (in the minimax sense) solutions of the Cauchy problem \eqref{HJ}, \eqref{boundary_condition} under the following standing assumption on $H$ and $\sigma$.

    \begin{assumption} \label{assumption_H_sigma}
        The following conditions hold.

        (A.1)
            For any fixed $x(\cdot) \in C([- h, T], \mathbb{R}^n)$, $s \in \mathbb{R}^n$, the function $t \mapsto H(t, x(\cdot), s)$, $[0, T] \to \mathbb{R}$, is measurable.

        (A.2)
            There exists a set $E^{\rm (A.2)} \subset [0, T]$ with $\mu(E^{\rm (A.2)}) = T$ and such that the mapping below is continuous for every fixed $t \in E^{\rm (A.2)}$:
            \begin{equation} \label{A.2_mapping}
                (x(\cdot), s) \mapsto H(t, x(\cdot), s),
                \quad C([- h, T], \mathbb{R}^n) \times \mathbb{R}^n \to \mathbb{R}.
            \end{equation}

        (A.3)
            For every compact set $D \subset C([- h, T], \mathbb{R}^n)$, there exists a non-negative function $\lambda_H(\cdot) \coloneq \lambda_H(\cdot; D) \in \mathcal{L}_1([0, T], \mathbb{R})$ satisfying the following property:
            for any $x_1(\cdot)$, $x_2(\cdot) \in D$, $s \in \mathbb{R}^n$, there exists a set $E^{\rm (A.3)} \coloneq E^{\rm (A.3)}(D, x_1(\cdot), x_2(\cdot), s) \subset [0, T]$ with $\mu(E^{\rm (A.3)}) = T$ and such that, for all $t \in E^{\rm (A.3)}$,
            \begin{equation} \label{A.3_inequality}
                |H(t, x_1(\cdot), s) - H(t, x_2(\cdot), s)|
                \leq \lambda_H(t) (1 + \|s\|) \|x_1(\cdot \wedge t) - x_2(\cdot \wedge t)\|_\infty.
            \end{equation}

        (A.4)
            There exist a non-negative function $c_H(\cdot) \in \mathcal{L}_1([0, T], \mathbb{R})$ and a set $E^{\rm (A.4)} \subset [0, T]$ with $\mu(E^{\rm (A.4)}) = T$ such that
            \begin{equation} \label{A.4_inequality}
                |H(t, x(\cdot), s_1) - H(t, x(\cdot), s_2)|
                \leq c_H(t) (1 + \|x(\cdot \wedge t)\|_\infty) \|s_1 - s_2\|
            \end{equation}
            for all $t \in E^{\rm (A.4)}$, $x(\cdot) \in C([- h, T], \mathbb{R}^n)$, $s_1$, $s_2 \in \mathbb{R}^n$.

        (A.5)
            For every compact set $D \subset C([- h, T], \mathbb{R}^n)$, there exists a non-negative function $m_H(\cdot) \coloneq m_H(\cdot; D) \in \mathcal{L}_1([0, T], \mathbb{R})$ satisfying the following property:
            for every $x(\cdot) \in D$, there exists a set $E^{\rm (A.5)} \coloneq E^{\rm (A.5)}(D, x(\cdot)) \subset [0, T]$ with $\mu(E^{\rm (A.5)}) = T$ and such that, for all $t \in E^{\rm (A.5)}$,
            \begin{equation} \label{A.5_inequality}
                |H(t, x(\cdot), 0)|
                \leq m_H(t).
            \end{equation}

        (A.6)
            The functional $\sigma$ is continuous.
    \end{assumption}

    Note that (see Corollary \ref{corollary_H_non-anticipative} in Section \ref{section_minimax_solution}) conditions (A.1)--(A.5) imply the following property of the Hamiltonian $H$:
    for every compact set $D \subset C([- h, T], \mathbb{R}^n)$, there exists a set $E_H \coloneq E_H(D) \subset [0, T]$ with $\mu(E_H) = T$ and such that, for all $t \in E_H$, $x(\cdot) \in D$, $s \in \mathbb{R}^n$,
    \begin{equation} \label{H_non-anticipative}
        H(t, x(\cdot), s)
        = H(t, x(\cdot \wedge t), s).
    \end{equation}
    This {\it property of non-anticipation} plays an important role and justifies the requirement that a solution $\varphi$ of the Cauchy problem \eqref{HJ}, \eqref{boundary_condition} should be {\it non-anticipative}, which means that, for all $(t, x(\cdot)) \in [0, T] \times C([- h, T], \mathbb{R}^n)$,
    \begin{equation*}
        \varphi(t, x(\cdot))
        = \varphi(t, x(\cdot \wedge t)).
    \end{equation*}

    Recall that Cauchy problems of type \eqref{HJ}, \eqref{boundary_condition} arise, for example, as Bellman and Isaacs equations in (deterministic) optimal control problems and differential games for functional differential equations of retarded type.
    Similarly to the classical (i.e., non-path-dependent) case, the Cauchy problem \eqref{HJ}, \eqref{boundary_condition} rarely has a solution interpreted in the classical sense (even under far more restrictive assumptions than Assumption \ref{assumption_H_sigma}), and, thus, appropriate generalized solutions should be considered.

    With regard to the Cauchy problem \eqref{HJ}, \eqref{boundary_condition}, the {\it minimax approach} to the notion of a generalized solution (see, e.g., \cite{Subbotin_1984,Subbotin_1991_Eng,Subbotin_1995}) was historically the first to be developed.
    The reader is referred to \cite{Lukoyanov_2011_Eng,Gomoyunov_Lukoyanov_Plaksin_2021,Gomoyunov_Lukoyanov_RMS_2024} for a discussion of the corresponding results obtained under various groups of assumptions on $H$ and $\sigma$ and also to a recent paper \cite{Bandini_Keller_2025}.
    Later, a number of papers appeared devoted to the development of the {\it viscosity solution technique} \cite{Crandall_Lions_1983,Crandall_Evans_Lions_1984} (see also, e.g., \cite{Crandall_Ishii_Lions_1992}).
    We limit ourselves here to referring to \cite{Lukoyanov_2007_IMM_Eng,
    Ekren_Keller_Touzi_Zhang_2014,
    Ren_Touzi_Zhang_2014,
    Ekren_Touzi_Zhang_2016_1,
    Bayraktar_Keller_2018,
    Zhou_2020_1,
    Cosso_Gozzi_Rosestolato_Russo_2021,
    Plaksin_2021_SIAM,
    Cosso_Russo_2022,
    Kaise_2022,
    Zhou_2022,
    Gomoyunov_Plaksin_2023_JFA}
    (see also a discussion in \cite[Section 3.8]{Gomoyunov_Lukoyanov_RMS_2024}).
    It should be emphasized that it is the theory of minimax solutions that has been constructed in its most complete form and under the most general assumptions by now.
    To some extent, this is due to the fact that the minimax approach partially avoids the difficulties caused by the fact that the Hamilton--Jacobi equation \eqref{HJ} is considered over the {\it infinite dimensional} path space $C([- h, T], \mathbb{R}^n)$, in which the property of local compactness is absent and the norm has rather poor differentiability properties.
    Since, in this paper, we make a next step and further weaken assumptions on $H$ and $\sigma$ to Assumption \ref{assumption_H_sigma}, it seems natural to take this approach as a basis.
    The possibility of developing the viscosity approach under Assumption \ref{assumption_H_sigma} is an interesting question, but it is beyond the scope of this paper.

\subsection{Discussion of the standing assumption}

    Let us discuss Assumption \ref{assumption_H_sigma} in a more detailed way.

    In the case where a Hamiltonian $H$ is {\it continuous}, the most general assumption under which the theory of minimax solutions was developed is the following (see, e.g., \cite{Gomoyunov_Lukoyanov_Plaksin_2021} and also \cite[Section 3.1]{Gomoyunov_Lukoyanov_RMS_2024}).

    \begin{assumption} \label{assumption_H_sigma_strong}
        The following conditions hold.

        (B.1)
            The mapping $H \colon [0, T] \times C([- h, T], \mathbb{R}^n) \times \mathbb{R}^n \to \mathbb{R}$ is continuous.

        (B.2)
            For every compact set $D \subset C([- h, T], \mathbb{R}^n)$, there exists a number $\lambda_H \coloneq \lambda_H(D) > 0$ such that, for all $t \in [0, T]$, $x_1(\cdot)$, $x_2(\cdot) \in D$, $s \in \mathbb{R}^n$,
            \begin{equation*}
                |H(t, x_1(\cdot), s) - H(t, x_2(\cdot), s)|
                \leq \lambda_H  (1 + \|s\|) \|x_1(\cdot \wedge t) - x_2(\cdot \wedge t)\|_\infty.
            \end{equation*}

        (B.3)
            There exists a number $c_H > 0$ such that
            \begin{equation*}
                |H(t, x(\cdot), s_1) - H(t, x(\cdot), s_2)|
                \leq c_H (1 + \|x(\cdot \wedge t)\|_\infty) \|s_1 - s_2\|
            \end{equation*}
            for all $t \in [0, T]$, $x(\cdot) \in C([- h, T], \mathbb{R}^n)$, $s_1$, $s_2 \in \mathbb{R}^n$.

        (B.4)
            The functional $\sigma \colon C([- h, T], \mathbb{R}^n) \to \mathbb{R}$ is continuous.
    \end{assumption}

    It is clear that Assumption \ref{assumption_H_sigma_strong} implies Assumption \ref{assumption_H_sigma}.

    The case of a {\it time-measurable Hamiltonian} $H$ (i.e., measurable with respect to the first variable $t$ and continuous with respect to other variables $x(\cdot)$ and $s$) was considered in the recent paper \cite{Bandini_Keller_2025}, which served as a main motivation for the present study.
    In that paper, a path-dependent Hamilton--Jacobi equation of a more general type was considered, in which the Hamiltonian $H$ may depend on $\varphi(t, x(\cdot))$ also.
    Nevertheless, the specification of Assumptions 4.1 and 4.2 from \cite{Bandini_Keller_2025} in relation to the Cauchy problem \eqref{HJ}, \eqref{boundary_condition} shows that, in essence, the difference is the following: instead of integrable functions $\lambda_H(\cdot)$, $c_H(\cdot)$, and $m_H(\cdot)$ in, respectively, conditions (A.3)--(A.5), numbers $\lambda_H$, $c_H$, and $m_H$ are involved, which is clearly more restrictive (in connection with condition (A.3), see also Remark \ref{remark_condition_A.3} in Section \ref{section_comparison}).
    In addition, note that the (local) boundedness condition (A.5) is clearly weaker than the corresponding condition (vi) of Assumption 4.2 from \cite{Bandini_Keller_2025}, which has the form of a (global) sublinear growth condition (in the spirit of condition (A.4)).
    It is worth emphasizing that such a weakening of the assumptions seems rather natural and important since it is consistent with the {\it Carath\'{e}odory conditions} in the theory of ordinary and functional differential equations and allows us to cover more general classes of optimal control problems and differential games in applications.
    On the other hand, this weakening requires some efforts and is therefore not purely formal.

    Thus, to the best of our knowledge, there are no results in the literature devoted to the development of the theory of generalized solutions of the Cauchy problem \eqref{HJ}, \eqref{boundary_condition} under rather general and natural Assumption \ref{assumption_H_sigma}.
    The present paper is intended to fill this gap.

    In conclusion of the discussion, it is also necessary to note that, in the theory of minimax solutions of (non-path-dependent) Hamilton--Jacobi equations with first-order partial derivatives, a group of conditions closest to Assumption \ref{assumption_H_sigma} was considered in \cite[Chapter 12]{Tran_Mikio_Nguyen_2000} (in this connection, see also, e.g., \cite{Vinter_Wolenski_1990}).
    However, there was an additional assumption of {\it positive homogeneity} of a Hamiltonian $H$ with respect to the third variable $s$.
    Among the related papers on the theory of viscosity solutions of such equations, we mention \cite{Ishii_1985,Barron_Jensen_1987,Lions_Perthame_1987} (see also \cite{Briani_Rampazzo_2005}).

\subsection{Definition of a minimax solution and consistency}

    In the paper, we propose notions of {\it upper}, {\it lower}, and {\it minimax solutions} of the Cauchy problem \eqref{HJ}, \eqref{boundary_condition} under Assumption \ref{assumption_H_sigma}.
    These definitions, involving the function $c_H(\cdot)$ from condition (A.4), are in natural agreement with the definitions given earlier in the case of Assumption \ref{assumption_H_sigma_strong} (see, e.g., \cite{Gomoyunov_Lukoyanov_Plaksin_2021} and \cite[Section 3.3]{Gomoyunov_Lukoyanov_RMS_2024}).
    From a substantive point of view, the definition of an upper (respectively, lower) solution is based on a property of weak invariance of the epigraph (respectively, hypograph) of the solution with respect to so-called characteristic differential inclusions.

    The first question we address is the {\it consistency property} of minimax solutions.
    Given that the Hamiltonian $H$ is only measurable with respect to the first variable $t$, we follow \cite[Chapter 11]{Tran_Mikio_Nguyen_2000} and propose a notion of a {\it semiclassical solution} of the Cauchy problem \eqref{HJ}, \eqref{boundary_condition}.
    To the best of our knowledge, such a notion has not been previously introduced in the literature for the considered case of path-dependent Hamilton--Jacobi equations.
    We prove that a semiclassical solution (if it exists) must be a minimax solution (see Theorem \ref{theorem_semiclassical}).
    On the other hand, we establish that, if a minimax solution $\varphi$ is co-invariantly differentiable at a point $(t, x(\cdot))$ belonging to a certain special set, then it satisfies the Hamilton--Jacobi equation \eqref{HJ} at this point (see Theorems \ref{theorem_consistency_compacts} and \ref{theorem_consistency}; a similar fact in the non-path-dependent case is \cite[Theorem 12.5, ii)]{Tran_Mikio_Nguyen_2000}).
    Note that the arguments use an auxiliary result on the {\it Lebesgue points} (see, e.g., \cite[p. 255]{Natanson_1_1964} and \cite[p. 562]{Bogachev_2007_1}) of the Hamiltonian $H$ with respect to the first variable $t$ (see Lemma \ref{lemma_Lebesgue_point}), which, in particular, allows us to justify equality \eqref{H_non-anticipative}.

\subsection{Existence, uniqueness, and stability results for minimax solutions}

    The next result that we prove is a {\it comparison theorem} for upper and lower solutions (see Theorem \ref{theorem_comparison_principle} and also Theorem \ref{theorem_comparison_principle_2}).
    In general, the proof follows the scheme of the proof of \cite[Theorem 8.1]{Subbotin_1995} and is close to the proof of \cite[Lemma 1]{Gomoyunov_Lukoyanov_RMS_2024}, where the case of Assumption \ref{assumption_H_sigma_strong} was considered.
    We only need to modify the {\it Lyapunov--Krasovskii functional} accordingly and use appropriate properties of the solution sets of functional differential inclusions of retarded type.
    The comparison theorem immediately implies the {\it uniqueness} of a minimax solution (see Theorem \ref{theorem_uniqueness}).
    Moreover, based on the comparison theorem, we show that the notions of upper, lower, and minimax solutions are {\it independent} of the choice of the function $c_H(\cdot)$ from condition (A.4) of Assumption \ref{assumption_H_sigma} (see Theorem \ref{thm:cH_independence} and Corollary \ref{cor:cH_independence}).

    In addition, we note that the proof of the comparison theorem differs from that of \cite[Corollary 5.4]{Bandini_Keller_2025}, where the approach proposed in \cite{Bayraktar_Keller_2018} was developed.
    The proof also differs from that of \cite[Lemma 12.12]{Tran_Mikio_Nguyen_2000}, where, due to the additional assumption of positive homogeneity of $H$, a different notion of a minimax solution was considered, and the corresponding constructions from earlier works \cite{Subbotin_1984,Subbotin_1991_Eng} were developed.

    Further, we establish a {\it stability result} for the minimax solution stating that the minimax solution depends continuously on variations of the Hamiltonian $H$ and the boundary functional $\sigma$ (see Theorem \ref{theorem_continuous_dependence}).
    The proof relies on the comparison theorem and again goes back to the proofs of the corresponding results in the classical case (see, e.g., \cite[Section 4.4]{Subbotin_1991_Eng}) and in the path-dependent case under Assumption \ref{assumption_H_sigma_strong} (see, e.g., \cite[Theorem 9.1]{Lukoyanov_2011_Eng} and also \cite[Theorem 6.1]{Lukoyanov_2001_DE_Eng}).
    A distinctive feature of the presented stability result is that the closeness of the Hamiltonians with respect to the first variable $t$ is understood not in the uniform norm, but in the {\it integral norm} $\|\cdot\|_1$ (in this connection, see also, e.g., \cite[Proposition 7.1]{Ishii_1985}).
    This feature is important in the context of Assumption \ref{assumption_H_sigma}, since it allows us to {\it approximate} the original time-measurable Hamiltonian $H$ by continuous Hamiltonians.

    Finally, we prove an {\it existence theorem} for the minimax solution.
    To this end, we perform the {\it Steklov transformation} of the Hamiltonian $H$ with respect to the first variable $t$ (see, e.g., \cite[p. 212]{Natanson_2_1960}) and obtain a sequence of Hamiltonians $H_k$, $k \in \mathbb{N}$, each of which satisfies conditions (B.1)--(B.3) from Assumption \ref{assumption_H_sigma_strong}.
    Further, for every $k \in \mathbb{N}$, we take a minimax solution $\varphi_k$ of the Cauchy problem \eqref{HJ}, \eqref{boundary_condition} with the Hamiltonian $H_k$ (existence and uniqueness of $\varphi_k$ are established in \cite[Theorem 1]{Gomoyunov_Lukoyanov_Plaksin_2021}).
    Then, after verifying that the sequence $H_k$, $k \in \mathbb{N}$, converges to $H$ in the appropriate sense, we apply the stability theorem and derive the existence of a minimax solution $\varphi$ of the original Cauchy problem \eqref{HJ}, \eqref{boundary_condition}.

    Summarizing, in this paper, we introduce a notion of a minimax solution of the Cauchy problem \eqref{HJ}, \eqref{boundary_condition} under Assumption \ref{assumption_H_sigma} and prove that this solution is {\it well-posed} and {\it consistent} with a notion of a solution in the classical sense.

\subsection{Application to differential games}

    The main results obtained in the paper are then applied to prove the existence of a value of a zero-sum differential game for a time-delay system under weakened assumptions.

    More precisely, we consider a {\it zero-sum differential game} described by {\it initial data} $(t, x(\cdot)) \in [0, T] \times C([- h, T], \mathbb{R}^n)$, the {\it dynamic equation}
    \begin{equation} \label{system}
        \dot{y}(\tau)
        = f(\tau, y(\cdot), u(\tau), v(\tau)),
    \end{equation}
    where $\tau \in [t, T]$, the {\it initial condition} $y(\tau) = x(\tau)$ for all $\tau \in [- h, t]$, and the {\it cost functional}
    \begin{equation} \label{cost_functional}
        J(t, x(\cdot), u(\cdot), v(\cdot))
        \coloneq \sigma(y(\cdot)) - \int_{t}^{T} \chi(\tau, y(\cdot), u(\tau), v(\tau)) \, \rd \tau.
    \end{equation}
    Here, $\tau$ is {\it time},
    $y(\tau) \in \mathbb{R}^n$ is the current {\it state} of the system,
    $\dot{y}(\tau) \coloneq \rd y(\tau) / \rd \tau$,
    $y(\cdot) \in C([- h, T], \mathbb{R}^n)$ is the {\it system motion},
    $u(\tau) \in P$ and $v(\tau) \in Q$ are the current {\it control actions} of the {\it first} and {\it second players}, respectively,
    $P \subset \mathbb{R}^{n_P}$ and $Q \subset \mathbb{R}^{n_Q}$ are compact sets,
    $n_P$, $n_Q \in \mathbb{N}$.

    The goal of the first player is to {\it minimize} the value of the cost functional \eqref{cost_functional} via $u(\cdot) \in \mathcal{U}[t, T]$, while the goal of the second player is to {\it maximize} this value via $v(\cdot) \in \mathcal{V}[t, T]$.
    For (rather standard) definitions of the sets of players' {\it admissible controls} $\mathcal{U}[t, T]$ and $\mathcal{V}[t, T]$, the {\it motions} $y(\cdot) \coloneq y(\cdot; t, x(\cdot), u(\cdot), v(\cdot))$ of system \eqref{system}, the {\it lower} $\rho^-(t, x(\cdot))$ and {\it upper} $\rho^+(t, x(\cdot))$ {\it game values} in the classes of {\it non-anticipative strategies}, and the {\it game value} $\rho(t, x(\cdot))$, see Section \ref{section_differential_game}.

    We study the differential game \eqref{system}, \eqref{cost_functional} under the following standing assumption on the mapping $(f, \chi) \colon [0, T] \times C([- h, T], \mathbb{R}^n) \times P \times Q \to \mathbb{R}^n \times \mathbb{R}$ and the functional $\sigma \colon C([- h, T], \mathbb{R}^n) \to \mathbb{R}$.

    \begin{assumption} \label{assunption_DG}
        The following conditions hold.

        (C.1)
            For any fixed $y(\cdot) \in C([- h, T], \mathbb{R}^n)$, $u \in P$, and $v \in Q$, the mapping $\tau \mapsto (f(\tau, y(\cdot), u, v), \chi(\tau, y(\cdot), u, v))$, $[0, T] \to \mathbb{R}^n \times \mathbb{R}$, is measurable.

        (C.2)
            There exists a set $E^{\rm (C.2)} \subset [0, T]$ with $\mu(E^{\rm (C.2)}) = T$ and such that the mapping below is continuous for every fixed $\tau \in E^{\rm (C.2)}$:
            \begin{equation} \label{C.2_f_chi}
                \begin{aligned}
                    (y(\cdot), u, v) & \mapsto (f(\tau, y(\cdot), u, v), \chi(\tau, y(\cdot), u, v)), \\
                    C([- h, T], \mathbb{R}^n) \times P \times Q & \to \mathbb{R}^n \times \mathbb{R}.
                \end{aligned}
            \end{equation}

        (C.3)
            For every compact set $D \subset C([- h, T], \mathbb{R}^n)$, there exist a non-negative function $\lambda_{f, \chi}(\cdot) \coloneq \lambda_{f, \chi}(\cdot; D) \in \mathcal{L}_1([0, T], \mathbb{R})$ and a set $E^{\rm (C.3)} \coloneq E^{\rm (C.3)}(D) \subset [0, T]$ with $\mu(E^{\rm (C.3)}) = T$ such that, for all $\tau \in E^{\rm (C.3)}$, $y_1(\cdot)$, $y_2(\cdot) \in D$, $u \in P$, $v \in Q$,
            \begin{equation} \label{C.3_f_chi}
                \begin{aligned}
                    & \|f(\tau, y_1(\cdot), u, v) - f(\tau, y_2(\cdot), u, v)\|
                    + |\chi(\tau, y_1(\cdot), u, v) - \chi(\tau, y_2(\cdot), u, v)| \\
                    & \leq \lambda_{f, \chi}(\tau) \|y_1(\cdot \wedge \tau) - y_2(\cdot \wedge \tau)\|_\infty.
                \end{aligned}
            \end{equation}

        (C.4)
            There exist a non-negative function $c_f(\cdot) \in \mathcal{L}_1([0, T], \mathbb{R})$ and a set $E^{\rm (C.4)} \subset [0, T]$ with $\mu(E^{\rm (C.4)}) = T$ such that
            \begin{equation} \label{C.4_f}
                \|f(\tau, y(\cdot), u, v)\|
                \leq c_f(\tau) (1 + \|y(\cdot \wedge \tau)\|_\infty)
            \end{equation}
            for all $\tau \in E^{\rm (C.4)}$, $y(\cdot) \in C([- h, T], \mathbb{R}^n)$, $u \in P$, $v \in Q$.

        (C.5)
            For every compact set $D \subset C([- h, T], \mathbb{R}^n)$, there exist a non-negative function $m_\chi(\cdot) \coloneq m_\chi(\cdot; D) \in \mathcal{L}_1([0, T], \mathbb{R})$ and a set $E^{\rm (C.5)} \coloneq E^{\rm (C.5)}(D) \subset [0, T]$ with $\mu(E^{\rm (C.5)}) = T$ such that, for all $\tau \in E^{\rm (C.5)}$, $y(\cdot) \in D$, $u \in P$, $v \in Q$,
            \begin{equation} \label{C.5_chi}
                |\chi(\tau, y(\cdot), u, v)|
                \leq m_\chi(\tau).
            \end{equation}

        (C.6)
            The functional $\sigma$ is continuous.
    \end{assumption}

    Note that condition (C.3) implies the following property of the mapping $(f, \chi)$:
    for every compact set $D \subset C([- h, T], \mathbb{R}^n)$, there exists a set $E_{f, \chi} \coloneq E_{f, \chi}(D) \subset [0, T]$ with $\mu(E_{f, \chi}) = T$ and such that, for all $\tau \in E_{f, \chi}$, $y(\cdot) \in D$, $u \in P$, $v \in Q$,
    \begin{equation*}
        f(\tau, y(\cdot), u, v)
        = f(\tau, y(\cdot \wedge \tau), u, v),
        \quad \chi(\tau, y(\cdot), u, v)
        = \chi(\tau, y(\cdot \wedge \tau), u, v).
    \end{equation*}
    In particular, this {\it property of non-anticipation} allows us to assert that the dynamic equation \eqref{system} is a {\it functional differential equation of retarded type} and call the dynamical system under consideration a {\it time-delay system}.

    By the mapping $(f, \chi)$, define the {\it lower} $H^- \colon [0, T] \times C([- h, T], \mathbb{R}^n) \times \mathbb{R}^n \to \mathbb{R}$ and {\it upper} $H^+ \colon [0, T] \times C([- h, T], \mathbb{R}^n) \times \mathbb{R}^n \to \mathbb{R}$ {\it Hamiltonians} by
    \begin{equation} \label{H_-_H_+_1}
        \begin{aligned}
            H^-(\tau, y(\cdot), s)
            & \coloneq \max_{v \in Q} \min_{u \in P} \bigl( \langle s, f(\tau, y(\cdot), u, v) \rangle
            - \chi(\tau, y(\cdot), u, v) \bigr), \\
            H^+(\tau, y(\cdot), s)
            & \coloneq \min_{u \in P} \max_{v \in Q} \bigl( \langle s, f(\tau, y(\cdot), u, v) \rangle
            - \chi(\tau, y(\cdot), u, v) \bigr)
        \end{aligned}
    \end{equation}
    if $\tau \in E^{\rm (C.2)}$ and by
    \begin{equation} \label{H_-_H_+_2}
        H^-(\tau, y(\cdot), s)
        \coloneq 0,
        \quad H^+(\tau, y(\cdot), s) \coloneq 0
    \end{equation}
    otherwise.
    Here, $y(\cdot) \in C([- h, T], \mathbb{R}^n)$, $s \in \mathbb{R}^n$, and $E^{\rm (C.2)}$ is the set from condition (C.2).
    Observe that $H^-$ and $H^+$ satisfy conditions (A.1)--(A.5).

    The central result that we prove in this part of the paper is that the {\it lower value functional} $\rho^- \colon [0, T] \times C([- h, T], \mathbb{R}^n) \to \mathbb{R}$ (respectively, {\it upper value functional} $\rho^+ \colon [0, T] \times C([- h, T], \mathbb{R}^n) \to \mathbb{R}$) coincides with the minimax solution $\varphi^-$ (respectively, minimax solution $\varphi^+$) of the Cauchy problem \eqref{HJ}, \eqref{boundary_condition} with $H = H^-$ (respectively, with $H = H^+$) and the boundary functional $\sigma$ from the cost functional \eqref{cost_functional} (see Theorem \ref{theorem_DG_lower_upper}).
    In particular, this fact confirms the {\it meaningfulness} of the notion of a minimax solution considered in the paper.

    Recall that, according to \cite[Theorem 7.1]{Bayraktar_Gomoyunov_Keller_2025} (in this connection, see also, e.g., \cite{Lukoyanov_2010_IMM_Eng_1,Lukoyanov_2009_IMM}), the stated result is valid under the following assumption.

    \begin{assumption} \label{assunption_DG_strong}
        The following conditions hold.

        (D.1)
            The mapping $(f, \chi) \colon [0, T] \times C([- h, T], \mathbb{R}^n) \times P \times Q \to \mathbb{R}^n \times \mathbb{R}$ is continuous.

        (D.2)
            For every compact set $D \subset C([- h, T], \mathbb{R}^n)$, there exists a number $\lambda_{f, \chi} \coloneq \lambda_{f, \chi}(D) > 0$ such that, for all $\tau \in [0, T]$, $y_1(\cdot)$, $y_2(\cdot) \in D$, $u \in P$, $v \in Q$,
            \begin{equation*}
                \begin{aligned}
                    & \|f(\tau, y_1(\cdot), u, v) - f(\tau, y_2(\cdot), u, v)\|
                    + |\chi(\tau, y_1(\cdot), u, v) - \chi(\tau, y_2(\cdot), u, v)| \\
                    & \leq \lambda_{f, \chi} \|y_1(\cdot \wedge \tau) - y_2(\cdot \wedge \tau)\|_\infty.
                \end{aligned}
            \end{equation*}

        (D.3)
            There exists a number $c_f > 0$ such that
            \begin{equation*}
                \|f(\tau, y(\cdot), u, v)\|
                \leq c_f (1 + \|y(\cdot \wedge \tau)\|_\infty)
            \end{equation*}
            for all $\tau \in [0, T]$, $y(\cdot) \in C([- h, T], \mathbb{R}^n)$, $u \in P$, and $v \in Q$.

        (D.4)
            The functional $\sigma \colon C([- h, T], \mathbb{R}^n) \to \mathbb{R}$ is continuous.
    \end{assumption}

    It is clear that Assumption \ref{assunption_DG_strong} implies Assumption \ref{assunption_DG}.
    In addition, conditions (C.1)--(C.5) seem rather natural since they are consistent with the Carath\'{e}odory conditions.

    In order to prove the result, we construct an {\it approximating sequence} of mappings $(f_k, \chi_k)$, $k \in \mathbb{N}$, each of which satisfies conditions (D.1)--(D.3), and compute the limits of the sequences of lower and upper value functionals of the approximating differential games \eqref{system}, \eqref{cost_functional} with $(f, \chi) = (f_k, \chi_k)$, $k \in \mathbb{N}$, as well as the limits of the sequences of the minimax solutions of the Cauchy problems \eqref{HJ}, \eqref{boundary_condition} with the Hamiltonians $H = H^-_k$ and $H = H^+_k$, $k \in \mathbb{N}$, corresponding to these games.
    In the latter case, we use the established stability result for minimax solutions.
    The construction of the appropriate approximation $(f_k, \chi_k)$ of the mapping $(f, \chi)$ is based, in particular, on the {\it Scorza Dragoni theorem} (see, e.g., \cite[Theorem 2]{Kucia_1991}) and the {\it McShane--Whitney extension theorem} (see, e.g., \cite[Theorem 4.1.1]{Cobzas_Miculescu_Nicolae_2019}).

    As a corollary, and owing to the uniqueness result for minimax solutions, we then derive that the differential game \eqref{system}, \eqref{cost_functional} {\it has a value} under Assumption \ref{assunption_DG} and the following assumption.

    \begin{assumption} \label{assumption_saddle_point}
        For every compact set $D \subset C([- h, T], \mathbb{R}^n)$ and every $y(\cdot) \in D$, there exists a set $E \coloneq E(D, y(\cdot)) \subset [0, T]$ with $\mu(E) = T$ and such that
        \begin{equation*}
            H^-(\tau, y(\cdot), s)
            = H^+(\tau, y(\cdot), s)
        \end{equation*}
        for all $\tau \in E$, $s \in \mathbb{R}^n$.
    \end{assumption}

    Moreover, we obtain that the {\it value functional} $\rho \colon [0, T] \times C([- h, T], \mathbb{R}^n) \to \mathbb{R}$ of the differential game \eqref{system}, \eqref{cost_functional} coincides with the minimax solution $\varphi$ of the Cauchy problem \eqref{HJ}, \eqref{boundary_condition} with $H = H^-$ (or, equivalently, with $H = H^+)$ (see Theorem \ref{theorem_game_value}).

\subsection{Organization of the paper}

    In Section \ref{section_minimax_solution}, we introduce the notions of co-invariant differentiability and co-invariant derivatives and provide definitions of upper, lower, and minimax solutions of the Cauchy problem \eqref{HJ}, \eqref{boundary_condition}.
    Furthermore, we present a result showing that, under Assumption \ref{assumption_H_sigma}, conditions (A.3) and (A.5) can be somewhat strengthened (see Corollary \ref{corollary_assumptions_strong}) and justify equality \eqref{H_non-anticipative}.
    In Section \ref{section_consistency}, we propose the notion of a semiclassical solution of the Cauchy problem \eqref{HJ}, \eqref{boundary_condition} and investigate the consistency properties of minimax solutions.
    In Sections \ref{section_comparison}--\ref{section_existence}, we formulate and prove the comparison, uniqueness, stability, and existence theorems for minimax solutions.
    In Section \ref{section_differential_game}, we give the results concerning the differential game \eqref{system}, \eqref{cost_functional}.

\section{Definition of a minimax solution}
\label{section_minimax_solution}

    We begin by introducing the notion of co-invariant differentiability ($ci$-differentiability, for short) and co-invariant derivatives ($ci$-derivatives, for short), which are involved in the Hamilton--Jacobi equation under consideration (see \eqref{HJ}).

    For every point $(t, x(\cdot)) \in [0, T] \times C([-h, T], \mathbb{R}^n)$, denote by $AC(t, x(\cdot))$ the set of all functions $y(\cdot) \in C([- h, T], \mathbb{R}^n)$ such that $y(\cdot \wedge t) = x(\cdot \wedge t)$ and the restriction $y|_{[t, T]}(\cdot)$ of the function $y(\cdot)$ to the interval $[t, T]$ is absolutely continuous.

    \begin{definition} \label{definition_ci-derivatives}
        A functional $\varphi \colon [0, T] \times C([- h, T], \mathbb{R}^n) \to \mathbb{R}$ is called {\it $ci$-differen\-tiable} at a point $(t, x(\cdot)) \in [0, T) \times C([-h, T], \mathbb{R}^n)$ if there exist $\partial_t \varphi (t, x(\cdot)) \in \mathbb{R}$ and $\nabla \varphi (t, x(\cdot)) \in \mathbb{R}^n$ with the following property:
        for every function $y(\cdot) \in AC(t, x(\cdot))$, there exists a function $o \colon (0, + \infty) \to \mathbb{R}$ such that
        \begin{equation} \label{ci-differentiability}
            \begin{aligned}
                & \varphi(\tau, y(\cdot)) - \varphi(t, x(\cdot)) \\
                & = \partial_t \varphi(t, x(\cdot)) (\tau - t)
                + \langle \nabla \varphi(t, x(\cdot)), y(\tau) - x(t) \rangle
                + o(\tau - t)
            \end{aligned}
        \end{equation}
        for all $\tau \in (t, T]$ and
        \begin{equation} \label{o}
            \lim_{\tau \to t^+}
            \frac{o(\tau - t)}
            {\tau - t + \int_{t}^{\tau} \|\dot{y}(\xi)\| \, \rd \xi}
            = 0.
        \end{equation}
        In this case, $\partial_t \varphi (t, x(\cdot))$ and $\nabla \varphi (t, x(\cdot))$ are called {\it $ci$-derivatives} of $\varphi$ at $(t, x(\cdot))$.
    \end{definition}

    Note that the $ci$-derivatives $\partial_t \varphi (t, x(\cdot))$ and $\nabla \varphi (t, x(\cdot))$ are determined uniquely.

    For a discussion of the notion of $ci$-differentiability, see, e.g., \cite[Section 3]{Gomoyunov_Lukoyanov_RMS_2024}.

    \begin{remark} \label{remark_ci-derivatives}
        In the literature, a common definition of $ci$-differentiability of a functional $\varphi$ at a point $(t, x(\cdot))$ uses a narrower class of ``right extensions'' $y(\cdot)$ of the point $(t, x(\cdot))$.
        Specifically, it is required that equality \eqref{ci-differentiability} holds only for functions $y(\cdot) \in AC(t, x(\cdot))$ such that $y|_{[t, T]}(\cdot)$ is Lipschitz continuous.
        In this case, relation \eqref{o} is naturally replaced by $\lim_{\tau \to t^+} o(\tau - t) / (\tau - t) = 0$.
        However, in the present paper, we cannot restrict ourselves to the class of Lipschitz continuous ``right extensions'', which leads to a slight modification of the definition.
    \end{remark}

    To give definitions of upper, lower, and minimax solutions of the Cauchy problem \eqref{HJ}, \eqref{boundary_condition}, we need to carry out auxiliary constructions.

    Given a point $(t, x(\cdot)) \in [0, T] \times C([- h, T], \mathbb{R}^n)$ and a non-negative function $c(\cdot) \in \mathcal{L}_1([0, T], \mathbb{R})$, consider the set
    \begin{equation} \label{Y}
        \begin{aligned}
            Y(t, x(\cdot); c(\cdot))
            & \coloneq \bigl\{ y(\cdot) \in AC(t, x(\cdot)) \colon \\
            & \quad \|\dot{y}(\tau)\|
            \leq  c(\tau) (1 + \|y(\cdot \wedge \tau)\|_\infty)
            \text{ for a.e. } \tau \in [t, T] \bigr\}.
        \end{aligned}
    \end{equation}
    Note that $Y(t, x(\cdot); c(\cdot)) \neq \varnothing$ since $x(\cdot \wedge t) \in Y(t, x(\cdot); c(\cdot))$.
    In addition, we have the following result, which is used several times in the paper.
    \begin{proposition} \label{proposition_1}
         Let points $(t, x(\cdot))$, $(t_k, x_k(\cdot)) \in [0, T] \times C([- h, T], \mathbb{R}^n)$, $k \in \mathbb{N}$, and non-neg\-ative functions $c(\cdot)$, $c_k(\cdot) \in \mathcal{L}_1([0, T], \mathbb{R})$, $k \in \mathbb{N}$ be such that
         \begin{equation} \label{proposition_1_convergence}
            \lim_{k \to \infty} \bigl( |t_k - t|
            + \|x_k(\cdot) - x(\cdot)\|_\infty
            + \|c_k(\cdot) - c(\cdot)\|_1 \bigr)
            = 0.
         \end{equation}
         Then, for every sequence $y_k(\cdot) \in Y(t_k, x_k(\cdot); c_k(\cdot))$, $k \in \mathbb{N}$, there exist a subsequence $y_{k_i}(\cdot)$, $i \in \mathbb{N}$, and a function $y(\cdot) \in Y(t, x(\cdot); c(\cdot))$ such that $\|y_{k_i}(\cdot) - y(\cdot)\|_\infty \to 0$ as $i \to \infty$.
    \end{proposition}
    \begin{proof}
        For every $k \in \mathbb{N}$, we have
        \begin{equation*}
            \|y_k(\tau)\|
            \leq \|x_k(t_k)\| + \int_{t_k}^{\tau} \|\dot{y}_k(\xi)\| \, \rd \xi
            \leq \|x_k(t_k)\| + \int_{t_k}^{\tau} c_k(\xi) (1 + \|y_k(\cdot \wedge \xi)\|_\infty) \, \rd \xi
        \end{equation*}
        for all $\tau \in [t_k, T]$, which yields
        \begin{equation*}
            1 + \|y_k(\cdot \wedge \tau)\|_\infty
            \leq 1 + \|x_k(\cdot \wedge t_k)\|_\infty + \int_{t_k}^{\tau} c_k(\xi) (1 + \|y_k(\cdot \wedge \xi)\|_\infty) \, \rd \xi
        \end{equation*}
        for all $\tau \in [t_k, T]$.
        Therefore, applying the Gronwall inequality, we derive
        \begin{equation*}
            1 + \|y_k(\cdot \wedge \tau)\|_\infty
            \leq (1 + \|x_k(\cdot \wedge t_k)\|_\infty) \exp \biggl( \int_{t_k}^{\tau} c_k(\xi) \, \rd \xi \biggr)
        \end{equation*}
        for all $\tau \in [t_k, T]$, $k \in \mathbb{N}$.
        Then, thanks to \eqref{proposition_1_convergence}, the sequence $y_k(\cdot)$, $k \in \mathbb{N}$, is uniformly bounded, i.e., there exists $R > 0$ such that $\|y_k(\cdot)\|_\infty \leq R$ for all $k \in \mathbb{N}$.

        Further, consider the auxiliary functions
        \begin{equation*}
            w(\tau)
            \coloneq (1 + R) \int_{0}^{\tau} c(\xi) \, \rd \xi,
            \quad w_k(\tau)
            \coloneq (1 + R) \int_{0}^{\tau} c_k(\xi) \, \rd \xi,
        \end{equation*}
        where $\tau \in [0, T]$, $k \in \mathbb{N}$.
        By \eqref{proposition_1_convergence}, $\max_{\tau \in [0, T]} |w_k(\tau) - w(\tau)| \to 0$ as $k \to \infty$, and, therefore, the sequence $w_k(\cdot)$, $k \in \mathbb{N}$, is uniformly equicontinuous.
        For any $k \in \mathbb{N}$, $\tau_1$, $\tau_2 \in [t_k, T]$ with $\tau_2 > \tau_1$, we derive
        \begin{equation} \label{proposition_1_proof_1}
            \begin{aligned}
                \|y_k(\tau_2) - y_k(\tau_1)\|
                & \leq \int_{\tau_1}^{\tau_2} \|\dot{y}_k(\xi)\| \, \rd \xi \\
                & \leq \int_{\tau_1}^{\tau_2} c_k(\xi) (1 + \|y_k(\cdot \wedge \xi)\|_\infty) \, \rd \xi \\
                & \leq (1 + R) \int_{\tau_1}^{\tau_2} c_k(\xi) \, \rd \xi \\
                & = w_k(\tau_2) - w_k(\tau_1).
            \end{aligned}
        \end{equation}
        Hence, the uniform equicontinuity of the sequences $x_k(\cdot)$, $w_k(\cdot)$, $k \in \mathbb{N}$, implies the uniform equicontinuity of the sequence $y_k(\cdot)$, $k \in \mathbb{N}$.

        As a result, according to the Arzel\`{a}--Ascoli theorem, there exists a subsequence $y_{k_i}(\cdot)$, $i \in \mathbb{N}$, and a function $y(\cdot) \in C([- h, T], \mathbb{R}^n)$ such that $\|y_{k_i}(\cdot) - y(\cdot)\|_\infty \to 0$ as $i \to \infty$, and it remains to verify that $y(\cdot) \in Y(t, x(\cdot); c(\cdot))$.

        Note that
        \begin{align*}
            0
            & = \lim_{i \to \infty} \|y(\cdot \wedge t) - y_{k_i}(\cdot \wedge t_{k_i})\|_\infty \\
            & = \lim_{i \to \infty} \|y(\cdot \wedge t) - x_{k_i}(\cdot \wedge t_{k_i})\|_\infty \\
            & = \|y(\cdot \wedge t) - x(\cdot \wedge t)\|_\infty,
        \end{align*}
        which means that $y(\cdot \wedge t) = x(\cdot \wedge t)$.
        Further, based on \eqref{proposition_1_convergence} and \eqref{proposition_1_proof_1}, we obtain
        \begin{equation*}
            \|y(\tau_2) - y(\tau_1)\|
            \leq \int_{\tau_1}^{\tau_2} c(\xi) (1 + \|y(\cdot \wedge \xi)\|_\infty) \, \rd \xi
        \end{equation*}
        for all $\tau_1$, $\tau_2 \in [t, T]$ with $\tau_2 > \tau_1$.
        Therefore, $y|_{[t, T]}(\cdot)$ is absolutely continuous and $\|\dot{y}(\tau)\| \leq c(\tau) (1 + \|y(\cdot \wedge \tau)\|_\infty)$ for every $\tau \in (t, T)$ such that the derivative $\dot{y}(\tau)$ exists and $\tau$ is a Lebesgue point of the function $c(\cdot)$, i.e., for a.e. $\tau \in [t, T]$.
        Consequently, $y(\cdot) \in Y(t, x(\cdot); c(\cdot))$, and the proof is complete.
    \end{proof}

    In particular, Proposition \ref{proposition_1} implies that the set $Y(t, x(\cdot); c(\cdot))$ is compact in $C([- h, T], \mathbb{R}^n)$ for every $(t, x(\cdot)) \in [0, T] \times C([- h, T], \mathbb{R}^n)$ and every non-neg\-ative function $c(\cdot) \in \mathcal{L}_1([0, T], \mathbb{R})$.

    \begin{definition} \label{definition_U}
        A functional $\varphi \colon [0, T] \times C([- h, T], \mathbb{R}^n) \to \mathbb{R}$ is called an {\it upper solution} of the Cauchy problem \eqref{HJ}, \eqref{boundary_condition} if it is non-anticipative, lower semicontinuous, satisfies the boundary condition
        \begin{equation} \label{boundary_condition_upper}
            \varphi(T, x(\cdot))
            \geq \sigma(x(\cdot))
        \end{equation}
        for all $x(\cdot) \in C([- h, T], \mathbb{R}^n)$, and possesses the following property.

        (U)
            For any $(t, x(\cdot)) \in [0, T) \times C([- h, T], \mathbb{R}^n)$, $s \in \mathbb{R}^n$, $\tau \in (t, T]$, there exists a function $y(\cdot) \in Y(t, x(\cdot); c_H(\cdot))$ such that
            \begin{equation} \label{U.1_inequality}
                \varphi(\tau, y(\cdot))
                - \int_{t}^{\tau} \bigl( \langle s, \dot{y}(\xi) \rangle
                - H(\xi, y(\cdot), s) \bigr) \, \rd \xi
                \leq \varphi(t, x(\cdot)).
            \end{equation}
    \end{definition}

    \begin{definition} \label{definition_L}
        A functional $\varphi \colon [0, T] \times C([- h, T], \mathbb{R}^n) \to \mathbb{R}$ is called a {\it lower solution} of the Cauchy problem \eqref{HJ}, \eqref{boundary_condition} if it is non-anticipative, upper semicontinuous, satisfies the boundary condition
        \begin{equation} \label{boundary_condition_lower}
            \varphi(T, x(\cdot))
            \leq \sigma(x(\cdot))
        \end{equation}
        for all $x(\cdot) \in C([- h, T], \mathbb{R}^n)$, and possesses the following property.

        (L)
            For any $(t, x(\cdot)) \in [0, T) \times C([- h, T], \mathbb{R}^n)$, $s \in \mathbb{R}^n$, $\tau \in (t, T]$, there exists a function $y(\cdot) \in Y(t, x(\cdot); c_H(\cdot))$ such that
            \begin{equation} \label{L.1_inequality}
                \varphi(\tau, y(\cdot))
                - \int_{t}^{\tau} \bigl( \langle s, \dot{y}(\xi) \rangle
                - H(\xi, y(\cdot), s) \bigr) \, \rd \xi
                \geq \varphi(t, x(\cdot)).
            \end{equation}
    \end{definition}

    In Definitions \ref{definition_U} and \ref{definition_L}, the set $Y(t, x(\cdot); c_H(\cdot))$ is defined according to \eqref{Y} with $c(\cdot) = c_H(\cdot)$, where $c_H(\cdot)$ is the function from condition (A.4) of Assumption \ref{assumption_H_sigma}.
    Note also that, thanks to conditions (A.1), (A.2), (A.4), and (A.5), for every $x(\cdot) \in C([- h, T], \mathbb{R}^n)$ and every continuous function $s \colon [0, T] \to \mathbb{R}^n$, the function $t \mapsto H(t, x(\cdot), s(t))$, $[0, T] \to \mathbb{R}$, belongs to $\mathcal{L}_1([0, T], \mathbb{R})$.
    In particular, this implies that the integrals in \eqref{U.1_inequality} and \eqref{L.1_inequality} are well-defined.

    \begin{definition} \label{definition_minimax}
        A functional $\varphi \colon [0, T] \times C([- h, T], \mathbb{R}^n) \to \mathbb{R}$ is called a {\it minimax solution} of the Cauchy problem \eqref{HJ}, \eqref{boundary_condition} if it is an upper solution as well as a lower solution of this problem.
    \end{definition}

    In particular, a minimax solution $\varphi$ is non-anticipative, continuous, and satisfies the boundary condition \eqref{boundary_condition} for all $x(\cdot) \in C([- h, T], \mathbb{R}^n)$.

    \begin{remark} \label{remark_H_null_set}
        Let Hamiltonians $H_1$ and $H_2$ satisfy conditions (A.1)--(A.5) and a boundary functional $\sigma$ satisfy condition (A.6).
        In addition, let condition (A.4) for $H_1$ and for $H_2$ be fulfilled with $c_{H_1}(\cdot) = c_{H_2}(\cdot) = c(\cdot)$ for some non-negative function $c(\cdot) \in \mathcal{L}_1([0, T], \mathbb{R})$.
        Suppose that, for every $x(\cdot) \in C([- h, T], \mathbb{R}^n)$, there exists a set $E \coloneq E(x(\cdot)) \subset [0, T]$ with $\mu(E) = T$ and such that $H_1(t, x(\cdot), s) = H_2(t, x(\cdot), s)$ for all $t \in E$, $s \in \mathbb{R}^n$.
        Then, it follows directly from the above definitions that a functional $\varphi \colon [0, T] \times C([- h, T], \mathbb{R}^n) \to \mathbb{R}$ is an upper (respectively, lower, minimax) solution of the Cauchy problem \eqref{HJ}, \eqref{boundary_condition} with $H = H_1$ if and only if it is an upper (respectively, lower, minimax) solution of the Cauchy problem \eqref{HJ}, \eqref{boundary_condition} with $H = H_2$.
        In this regard, we note that conditions (A.2) and (A.4) can be somewhat strengthened when dealing with upper, lower, or minimax solutions.
        Namely, we can assume that mapping \eqref{A.2_mapping} is continuous for all $t \in [0, T]$ and that inequality \eqref{A.4_inequality} holds for all $t \in [0, T]$, $x(\cdot) \in C([- h, T], \mathbb{R}^n)$, $s_1$, $s_2 \in \mathbb{R}^n$.
    \end{remark}

    Concluding this section, we give an auxiliary result on the Lebesgue points of the Hamiltonian $H$ with respect to the first variable $t$.
    The proof follows the same lines as the proof of \cite[Lemma 12.2]{Tran_Mikio_Nguyen_2000}.
    \begin{lemma} \label{lemma_Lebesgue_point}
        Suppose that conditions {\rm (A.1)}--{\rm (A.5)} of Assumption {\rm \ref{assumption_H_sigma}} hold.
        Then, for every compact set $D \subset C([- h, T], \mathbb{R}^n)$, there exists a set $E_H^1 \coloneq E_H^1(D) \subset [0, T)$ with $\mu(E_H^1) = T$ and such that, for any $t \in E_H^1$, $x(\cdot) \in D$, $s \in \mathbb{R}^n$,
        \begin{equation} \label{lemma_Lebesgue_point_main}
            \lim_{\tau \to t^+} \frac{1}{\tau - t} \int_{t}^{\tau} H(\xi, x(\cdot), s) \, \rd \xi
            = H(t, x(\cdot), s).
        \end{equation}
    \end{lemma}
    \begin{proof}
        Consider the set $E^{\rm (A.2)}$ and the functions $\lambda_H(\cdot) \coloneq \lambda_H(\cdot; D)$ and $c_H(\cdot)$ from conditions (A.2)--(A.4).
        Denote by $E_1$ the set of all $t \in [0, T)$ such that $t$ is a Lebesgue point of both functions $\lambda_H(\cdot)$ and $c_H(\cdot)$.
        Take an at most countable set $D_\ast \subset D$ that is dense in $D$ and let $E_2$ be the set of all $t \in [0, T)$ such that $t$ is a Lebesgue point of each of the functions $\tau \mapsto H(\tau, x(\cdot), s)$, $[0, T] \to \mathbb{R}$, where $x(\cdot) \in D_\ast$, $s \in \mathbb{Q}^n$.
        Put $E_H^1 \coloneq E^{\rm (A.2)} \cap E_1 \cap E_2$ and observe that $\mu(E_H^1) = T$.

        Fix $t \in E_H^1$, $x(\cdot) \in D$, $s \in \mathbb{R}^n$ and let $\varepsilon > 0$.
        Since $t \in E_1$, there exists $M > 0$ such that, for every $\tau \in (t, T]$,
        \begin{equation} \label{M_choice}
            \frac{1}{\tau - t} \int_{t}^{\tau} c_H(\xi) \, \rd \xi
            \leq M,
            \quad \frac{1}{\tau - t} \int_{t}^{\tau} \lambda_H(\xi) \, \rd \xi
            \leq M.
        \end{equation}
        Due to the inclusion $t \in E^{\rm (A.2)}$, there exists $\delta_1 > 0$ such that
        \begin{equation} \label{lemma_Lebesgue_point_proof_1}
            |H(t, y(\cdot), r) - H(t, x(\cdot), s)|
            \leq \varepsilon / 3
        \end{equation}
        for all $y(\cdot) \in D$ and $r \in \mathbb{R}^n$ with $\|y(\cdot) - x(\cdot)\|_\infty \leq \delta_1$ and $\|r - s\| \leq \delta_1$.
        Define $\delta_2 \coloneq \min\{\delta_1, \varepsilon / (6 M (1 + \|s\|)), \varepsilon / (6 M (2 + \|x(\cdot)\|_\infty)), 1\}$.
        Take $y(\cdot) \in D_\ast$ and $r \in \mathbb{Q}^n$ satisfying the inequalities $\|y(\cdot) - x(\cdot)\|_\infty \leq \delta_2$ and $\|r - s\| \leq \delta_2$.
        Owing to the inclusion $t \in E_2$, there exists $\delta_3 \in (0, T - t]$ such that, for every $\tau \in (t, t + \delta_3]$,
        \begin{equation} \label{lemma_Lebesgue_point_proof_2}
           \biggl| \frac{1}{\tau - t} \int_{t}^{\tau} H(\xi, y(\cdot), r) \, \rd \xi
           - H(t, y(\cdot), r) \biggr|
           \leq \frac{\varepsilon}{3}.
        \end{equation}
        Fix $\tau \in (t, t + \delta_3]$.
        We have
        \begin{align*}
            & \frac{1}{\tau - t} \int_{t}^{\tau} | H(\xi, x(\cdot), s)
            - H(\xi, y(\cdot), r) | \, \rd \xi \\
            & \leq \frac{1}{\tau - t} \int_{t}^{\tau}
            \lambda_H(\xi) (1 + \|s\|) \|x(\cdot \wedge \xi) - y(\cdot \wedge \xi) \|_\infty \, \rd \xi \\
            & \quad + \frac{1}{\tau - t} \int_{t}^{\tau}
            c_H(\xi) (1 + \|y(\cdot \wedge \xi)\|_\infty) \|r - s\| \, \rd \xi \\
            & \leq M (1 + \|s\|) \|x(\cdot) - y(\cdot)\|_\infty
            + M (2 + \|x(\cdot)\|_\infty) \|r - s\| \\
            & \leq \varepsilon / 3.
        \end{align*}
        Combining this estimate with \eqref{lemma_Lebesgue_point_proof_1} and \eqref{lemma_Lebesgue_point_proof_2}, we obtain
        \begin{equation*}
            \biggl| \frac{1}{\tau - t} \int_{t}^{\tau} H(\xi, x(\cdot), s) \, \rd \xi
           - H(t, x(\cdot), s) \biggr|
           \leq \varepsilon.
        \end{equation*}
        The proof is complete.
    \end{proof}

    In addition, we note that Lemma \ref{lemma_Lebesgue_point} allows us to obtain stronger versions of conditions (A.3) and (A.5) and, in particular, to justify equality \eqref{H_non-anticipative}.
    Namely, the following two corollaries are valid under conditions (A.1)--(A.5) on $H$.
    \begin{corollary} \label{corollary_assumptions_strong}
        Let $D \subset C([- h, T], \mathbb{R}^n)$ be a compact set and $\lambda_H(\cdot) \coloneq \lambda_H(\cdot; D)$ and $m_H(\cdot) \coloneq m_H(\cdot; D)$ be the corresponding functions from conditions {\rm (A.3)} and {\rm (A.5)}.
        Then, there exists a set $E_H^2 \coloneq E_H^2(D) \subset [0, T]$ with $\mu(E_H^2) = T$ and such that inequalities \eqref{A.3_inequality} and \eqref{A.5_inequality} hold for all $t \in E_H^2$, $x(\cdot)$, $x_1(\cdot)$, $x_2(\cdot) \in D$, and $s \in \mathbb{R}^n$.
    \end{corollary}
    \begin{proof}
        Consider the set $E_H^1 \coloneq E_H^1(D)$ from Lemma \ref{lemma_Lebesgue_point} and let $E$ be the set of all $t \in [0, T)$ such that $t$ is a Lebesgue point of both functions $\lambda_H(\cdot)$ and $m_H(\cdot)$.
        Define $E_H^2 \coloneq E_H^1 \cap E$ and note that $\mu(E_H^2) = T$.

        For any $t \in E_H^2$, $x_1(\cdot)$, $x_2(\cdot) \in D$, $s \in \mathbb{R}^n$, we have
        \begin{align*}
            & |H(t, x_1(\cdot), s) - H(t, x_2(\cdot), s)| \\
            & \leq \limsup_{\tau \to t^+} \frac{1}{\tau - t}
            \int_{t}^{\tau} | H(\xi, x_1(\cdot), s)
            - H(\xi, x_2(\cdot), s) | \, \rd \xi \\
            & \leq \limsup_{\tau \to t^+} \frac{1}{\tau - t}
            \int_{t}^{\tau} \lambda_H(\xi) (1 + \|s\|)
            \|x_1(\cdot \wedge \xi) - x_2(\cdot \wedge \xi)\|_\infty \, \rd \xi \\
            & \leq (1 + \|s\|)
            \lim_{\tau \to t^+} \|x_1(\cdot \wedge \tau) - x_2(\cdot \wedge \tau)\|_\infty
            \lim_{\tau \to t^+} \frac{1}{\tau - t}
            \int_{t}^{\tau} \lambda_H(\xi) \, \rd \xi \\
            & = \lambda_H(t) (1 + \|s\|) \|x_1(\cdot \wedge t) - x_2(\cdot \wedge t)\|_\infty.
        \end{align*}
        Similarly, for any $t \in E_H^2$, $x(\cdot) \in D$, we obtain
        \begin{align*}
            |H(t, x(\cdot), 0)|
            & \leq \limsup_{\tau \to t^+} \frac{1}{\tau - t} \int_{t}^{\tau} |H(\xi, x(\cdot), 0)| \, \rd \xi \\
            & \leq \lim_{\tau \to t^+} \frac{1}{\tau - t} \int_{t}^{\tau} m_H(\xi) \, \rd \xi \\
            & = m_H(t).
        \end{align*}
        The proof is complete.
    \end{proof}

    The difference of the properties from Corollary \ref{corollary_assumptions_strong} compared to conditions (A.3) and (A.5) is that the set of $t \in [0, T]$ for which inequalities \eqref{A.3_inequality} and \eqref{A.5_inequality} hold depends only on the compact set $D$ and does not depend on the specific choice of $x_1(\cdot)$, $x_2(\cdot)$, $s$ in \eqref{A.3_inequality} and of $x(\cdot)$ in \eqref{A.5_inequality}.

    \begin{corollary} \label{corollary_H_non-anticipative}
        For every compact set $D \subset C([- h, T], \mathbb{R}^n)$, there exists a set $E_H \coloneq E_H(D) \subset [0, T]$ with $\mu(E_H) = T$ and such that equality \eqref{H_non-anticipative} holds for all $t \in E_H$, $x(\cdot) \in D$, $s \in \mathbb{R}^n$.
    \end{corollary}
    \begin{proof}
        Consider the set $D_\ast \coloneq \{ x(\cdot \wedge t) \colon x(\cdot) \in D, \, t \in [0, T] \}$, which is compact in $C([- h, T], \mathbb{R}^n)$ by compactness of $D$.
        Let $\lambda_H(\cdot) \coloneq \lambda_H(\cdot; D_\ast)$ be the function from condition (A.3) and let $E_H^2 \coloneq E_H^2(D_\ast)$ be the set from Corollary \ref{corollary_assumptions_strong}.
        Put $E_H \coloneq E_H^2$.
        For any $t \in E_H$, $x(\cdot) \in D$, $s \in \mathbb{R}^n$, noting that $x(\cdot)$, $x(\cdot \wedge t) \in D_\ast$, we obtain
        \begin{equation*}
            |H(t, x(\cdot), s) - H(t, x(\cdot \wedge t), s)|
            \leq \lambda_H(t) (1 + \|s\|) \|x(\cdot \wedge t) - x(\cdot \wedge t)\|_\infty
            = 0,
        \end{equation*}
        which completes the proof.
    \end{proof}

\section{Semiclassical solutions and consistency}
\label{section_consistency}

    A first result of this section is that a semiclassical solution of the Cauchy problem \eqref{HJ}, \eqref{boundary_condition} (if it exists) is a minimax solution of this problem.
    In order to give a definition of a semiclassical solution, which follows \cite[Chapter 11]{Tran_Mikio_Nguyen_2000}, we need to introduce an appropriate set of functionals $\varphi \colon [0, T] \times C([- h, T], \mathbb{R}^n) \to \mathbb{R}$.

    Denote by $\Phi$ the set of all functionals $\varphi \colon [0, T] \times C([- h, T], \mathbb{R}^n) \to \mathbb{R}$ satisfying the following conditions.

    ($\Phi.1$)
        The functional $\varphi$ is continuous and the function $\tau \mapsto \varphi(\tau, y(\cdot))$, $[t, \vartheta] \to \mathbb{R}$, is absolutely continuous for all $(t, x(\cdot)) \in [0, T) \times C([- h, T], \mathbb{R}^n)$, $y(\cdot) \in AC(t, x(\cdot))$, $\vartheta \in (t, T)$.

    ($\Phi.2$)
        There exists a set $E_\varphi \subset [0, T)$ with $\mu(E_\varphi) = T$ and such that the functional $\varphi$ is $ci$-differentiable at every point $(t, x(\cdot)) \in E_\varphi \times C([- h, T], \mathbb{R}^n)$.

    ($\Phi.3$)
        The function $t \mapsto (\partial_t \varphi(t, x(\cdot)), \nabla \varphi(t, x(\cdot)))$, $E_\varphi \to \mathbb{R} \times \mathbb{R}^n$, is measurable for every $x(\cdot) \in C([- h, T], \mathbb{R}^n)$, and the mapping $x(\cdot) \mapsto (\partial_t \varphi(t, x(\cdot)), \nabla \varphi(t, x(\cdot)))$, $C([- h, T], \mathbb{R}^n) \to \mathbb{R} \times \mathbb{R}^n$, is continuous for every $t \in E_\varphi$.

    ($\Phi.4$)
        For every compact set $D \subset C([- h, T], \mathbb{R}^n)$ and every $\vartheta \in [0, T)$, there exist a non-negative function $m_\varphi(\cdot) \coloneq m_\varphi(\cdot; D, \vartheta) \in \mathcal{L}_1([0, \vartheta], \mathbb{R})$ and a number $\ell_\varphi \coloneq \ell_\varphi(D, \vartheta) > 0$ such that, for any $t \in E_\varphi \cap [0, \vartheta]$, $x(\cdot) \in D$,
        \begin{equation*}
            |\partial_t \varphi(t, x(\cdot))|
            \leq m_\varphi(t),
            \quad \|\nabla \varphi(t, x(\cdot)) \|
            \leq \ell_\varphi.
        \end{equation*}

    Note that (see, e.g., \cite[Section 3.1]{Gomoyunov_Lukoyanov_RMS_2024}) it follows from ($\Phi.2$) that
    \begin{equation} \label{varphi_derivatives_non-anticipative}
        \begin{aligned}
            \varphi(t, x(\cdot))
            & = \varphi(t, x(\cdot \wedge t)), \\
            \partial_t \varphi(t, x(\cdot))
            & = \partial_t \varphi(t, x(\cdot \wedge t)), \\
            \nabla \varphi(t, x(\cdot))
            & = \nabla \varphi(t, x(\cdot \wedge t))
        \end{aligned}
    \end{equation}
    for all $(t, x(\cdot)) \in E_\varphi \times C([- h, T], \mathbb{R}^n)$.
    In particular, by continuity of $\varphi$, we conclude that $\varphi$ is automatically non-anticipative.

    For every functional $\varphi \in \Phi$, the following so-called {\it functional chain rule} holds.
    \begin{proposition} \label{proposition_Phi}
        Let $\varphi \in \Phi$, $(t, x(\cdot)) \in [0, T) \times C([- h, T], \mathbb{R}^n)$, $y(\cdot) \in AC(t, x(\cdot))$ be given.
        Then, for every $\tau \in [t, T)$,
        \begin{equation} \label{proposition_Phi_main}
            \varphi(\tau, y(\cdot))
            = \varphi(t, x(\cdot)) + \int_{t}^{\tau} \bigl( \partial_t \varphi (\xi, y(\cdot))
            + \langle \nabla \varphi(\xi, y(\cdot)), \dot{y}(\xi) \rangle \bigr) \, \rd \xi.
        \end{equation}
    \end{proposition}
    \begin{proof}
        For every $k \in \mathbb{N}$, by the Lusin theorem, there exists a closed set $F_k \subset [t, T]$ with $\mu(F_k) \geq T - t - 1 / k$ and such that $\dot{y}|_{F_k}(\cdot)$ is continuous.
        Consider a function $y_k(\cdot) \in AC(t, x(\cdot))$ such that $\dot{y}_k(\tau) = f_k(\tau)$ for a.e. $\tau \in [t, T]$, where $f_k(\tau) \coloneq \dot{y}(\tau)$ if $\tau \in F_k$ and $f_k(\tau) \coloneq 0$ otherwise.
        By construction, $y_k|_{[t, T]}(\cdot)$ is Lipschitz continuous and $\|\dot{y}_k(\tau)\| \leq \|\dot{y}(\tau)\|$ for a.e. $\tau \in [t, T]$.
        In addition, we have $\|y_k(\cdot) - y(\cdot)\|_\infty \to 0$ as $k \to \infty$ and, by passing to a subsequence if necessary, $\dot{y}_k(\tau) \to \dot{y}(\tau)$ as $k \to \infty$ for a.e. $\tau \in [t, T]$.

        We fix $\vartheta \in (t, T)$ and prove equality \eqref{proposition_Phi_main} for all $\tau \in [t, \vartheta]$.

        For every $k \in \mathbb{N}$, consider the function $\omega_k(\tau) \coloneq \varphi(\tau, y_k(\cdot))$, $\tau \in [t, \vartheta]$, which is absolutely continuous by ($\Phi.1$).
        Let $E_k$ denote the set of all $\tau \in (t, \vartheta)$ such that $\tau \in E_\varphi$ and both derivatives $\dot{\omega}_k(\tau)$ and $\dot{y}_k(\tau)$ exist.
        Note that $\mu(E_k) = \vartheta - t$ and fix $\tau \in E_k$.
        Recalling that the functional $\varphi$ is $ci$-differentiable at the point $(\tau, y_k(\cdot))$ by ($\Phi.2$) and using the inclusion $y_k(\cdot) \in AC(\tau, y_k(\cdot))$, we obtain that there exists a function $o_k \colon (0, + \infty) \to \mathbb{R}$ such that
        \begin{equation} \label{proposition_Phi_proof_1}
            \begin{aligned}
                \omega_k(\xi) - \omega_k(\tau)
                & = \varphi(\xi, y_k(\cdot)) - \varphi(\tau, y_k(\cdot)) \\
                & = \partial_t \varphi (\tau, y_k(\cdot)) (\xi - \tau)
                + \langle \nabla \varphi (\tau, y_k(\cdot)),
                y_k(\xi) - y_k(\tau) \rangle \\
                & \quad + o_k(\xi - \tau)
            \end{aligned}
        \end{equation}
        for all $\xi \in (\tau, \vartheta]$ and
        \begin{equation*}
            \lim_{\xi \to \tau^+}
            \frac{o_k(\xi - \tau)}
            {\xi - \tau + \int_{\tau}^{\xi} \|\dot{y}_k(\eta)\| \, \rd \eta}
            = 0.
        \end{equation*}
        Since $y_k|_{[t, T]}(\cdot)$ is Lipschitz continuous, $\lim_{\xi \to \tau^+} o_k(\xi - \tau) / (\xi - \tau) = 0$.
        Hence, if we divide \eqref{proposition_Phi_proof_1} by $\xi - \tau$ and pass to the limit as $\xi \to \tau^+$, we get
        \begin{equation*}
            \dot{\omega}_k(\tau)
            = \partial_t \varphi (\tau, y_k(\cdot))
            + \langle \nabla \varphi (\tau, y_k(\cdot)),
            \dot{y}_k(\tau) \rangle.
        \end{equation*}
        As a result, we conclude that, for any $k \in \mathbb{N}$, $\tau \in [t, \vartheta]$,
        \begin{equation} \label{proposition_Phi_proof_2}
            \varphi(\tau, y_k(\cdot)) - \varphi(t, y_k(\cdot))
            = \int_{t}^{\tau} \bigl( \partial_t \varphi (\xi, y_k(\cdot))
            + \langle \nabla \varphi(\xi, y_k(\cdot)), \dot{y}_k(\xi) \rangle \bigr) \, \rd \xi.
        \end{equation}

        Fix $\tau \in [t, \vartheta]$.
        Thanks to ($\Phi.1$), ($\Phi.3$), and \eqref{varphi_derivatives_non-anticipative}, we derive $\varphi(\tau, y_k(\cdot)) \to \varphi(\tau, y(\cdot))$ as $k \to \infty$, $\varphi(t, y_k(\cdot)) = \varphi(t, x(\cdot))$ for all $k \in \mathbb{N}$, and, for a.e. $\xi \in [t, \tau]$,
        \begin{equation*}
            \lim_{k \to \infty} \bigl( \partial_t \varphi (\xi, y_k(\cdot))
            + \langle \nabla \varphi(\xi, y_k(\cdot)), \dot{y}_k(\xi) \rangle \bigr)
            = \partial_t \varphi (\xi, y(\cdot))
            + \langle \nabla \varphi(\xi, y(\cdot)), \dot{y}(\xi) \rangle.
        \end{equation*}
        Define the compact set $D \coloneq \{y_k(\cdot) \colon k \in \mathbb{N} \} \cup \{y(\cdot)\}$ and take $m_\varphi(\cdot) \coloneq m_\varphi(\cdot; D, \vartheta)$, $\ell_\varphi \coloneq \ell_\varphi(D, \vartheta)$ from ($\Phi.4$).
        For every $k \in \mathbb{N}$, we have
        \begin{equation*}
            \bigl| \partial_t \varphi (\xi, y_k(\cdot))
            + \langle \nabla \varphi(\xi, y_k(\cdot)), \dot{y}_k(\xi) \rangle \bigr|
            \leq m_\varphi(\xi) + \ell_\varphi \|\dot{y}_k(\xi)\|
            \leq m_\varphi(\xi) + \ell_\varphi \|\dot{y}(\xi)\|
        \end{equation*}
        for a.e. $\xi \in [t, \tau]$.
        Then, by the Lebesgue dominated convergence theorem,
        \begin{equation*}
            \begin{aligned}
                & \lim_{k \to \infty} \int_{t}^{\tau} \bigl( \partial_t \varphi (\xi, y_k(\cdot))
                + \langle \nabla \varphi(\xi, y_k(\cdot)), \dot{y}_k(\xi) \rangle \bigr) \, \rd \xi \\
                & = \int_{t}^{\tau} \bigl( \partial_t \varphi (\xi, y(\cdot))
                + \langle \nabla \varphi(\xi, y(\cdot)), \dot{y}(\xi) \rangle \bigr) \, \rd \xi.
            \end{aligned}
        \end{equation*}
        Consequently, by passing to the limit as $k \to \infty$ in \eqref{proposition_Phi_proof_2}, we come to equality \eqref{proposition_Phi_main} and complete the proof.
    \end{proof}

    \begin{remark} \label{remark_ci-smooth}
        Suppose that a functional $\varphi \colon [0, T] \times C([- h, T], \mathbb{R}^n) \to \mathbb{R}$ has the following properties (see also Remark \ref{remark_ci-derivatives}).

        (i)
            For every $(t, x(\cdot)) \in [0, T) \times C([- h, T], \mathbb{R}^n)$, there exist $\partial_t^\ast \varphi (t, x(\cdot)) \in \mathbb{R}$ and $\nabla^\ast \varphi(t, x(\cdot)) \in \mathbb{R}^n$ such that, for every function $y(\cdot) \in AC(t, x(\cdot))$ such that $y|_{[t, T]}(\cdot)$ is Lipschitz continuous, there exists a function $o \colon (0, + \infty) \to \mathbb{R}$ such that equality \eqref{ci-differentiability} holds for all $\tau \in (t, T]$ and $\lim_{\tau \to t^+} o(\tau - t) / (\tau - t) = 0$.

        (ii)
            The functional $\varphi$ as well as the mappings $\partial_t^\ast \varphi \colon [0, T) \times C([- h, T], \mathbb{R}^n) \to \mathbb{R}$, $\nabla^\ast \varphi \colon [0, T) \times C([- h, T], \mathbb{R}^n) \to \mathbb{R}^n$ are continuous.

        Such a functional $\varphi$ is often called {\it co-invariantly smooth} ({\it $ci$-smooth}, for short) in the literature.
        According to, e.g., \cite[Proposition 1]{Gomoyunov_Lukoyanov_RMS_2024}, for every point $(t, x(\cdot)) \in [0, T) \times C([- h, T], \mathbb{R}^n)$ and every function $y(\cdot) \in AC(t, x(\cdot))$ such that $y|_{[t, T]}(\cdot)$ is Lipschitz continuous, the function $\tau \mapsto \varphi(\tau, y(\cdot))$, $[t, \vartheta] \to \mathbb{R}$, is Lipschitz continuous for every $\vartheta \in (t, T)$ and the equality
        \begin{equation} \label{proposition_Phi_main^ast}
            \varphi(\tau, y(\cdot))
            = \varphi(t, x(\cdot)) + \int_{t}^{\tau} \bigl( \partial_t^\ast \varphi (\xi, y(\cdot))
            + \langle \nabla^\ast \varphi(\xi, y(\cdot)), \dot{y}(\xi) \rangle \bigr) \, \rd \xi
        \end{equation}
        holds for all $\tau \in [t, T)$.
        Then, arguing similarly to the proof of Proposition \ref{proposition_Phi}, we obtain that equality \eqref{proposition_Phi_main^ast} actually holds for all $(t, x(\cdot)) \in [0, T) \times C([- h, T], \mathbb{R}^n)$, $y(\cdot) \in AC(t, x(\cdot))$, $\tau \in [t, T)$.
        Together with continuity of $\partial_t^\ast \varphi$ and $\nabla^\ast \varphi$, this fact implies that $\varphi$ is $ci$-differentiable at every point $(t, x(\cdot)) \in [0, T) \times C([- h, T], \mathbb{R}^n)$ with $\partial_t \varphi(t, x(\cdot)) = \partial_t^\ast \varphi(t, x(\cdot))$, $\nabla \varphi(t, x(\cdot)) = \nabla^\ast \varphi(t, x(\cdot))$.
        As a result, for every functional $\varphi$ satisfying (i) and (ii), we derive $\varphi \in \Phi$ with $E_\varphi = [0, T)$.
    \end{remark}

    Now, we are in a position to give a definition of a semiclassical solution.
    \begin{definition}
        A functional $\varphi \in \Phi$ is called a {\it semiclassical solution} of the Cauchy problem \eqref{HJ}, \eqref{boundary_condition} if it satisfies the Hamilton--Jacobi equation \eqref{HJ} for all $(t, x(\cdot)) \in E_\varphi \times C([- h, T], \mathbb{R}^n)$ and the boundary condition \eqref{boundary_condition} for all $x(\cdot) \in C([- h, T], \mathbb{R}^n)$.
    \end{definition}

    The first result of this section is the following.
    \begin{theorem} \label{theorem_semiclassical}
        Under Assumption {\rm \ref{assumption_H_sigma}}, every semiclassical solution $\varphi$ of the Cauchy problem \eqref{HJ}, \eqref{boundary_condition} is a minimax solution of this problem.
    \end{theorem}
    \begin{proof}
        To prove the result, it suffices to fix $(t, x(\cdot)) \in [0, T) \times C([- h, T], \mathbb{R}^n)$, $s \in \mathbb{R}^n$ and find a function $y(\cdot) \in Y(t, x(\cdot); c_H(\cdot))$ for which both inequalities \eqref{U.1_inequality} and \eqref{L.1_inequality} take place for all $\tau \in [t, T]$.
        Recall that the set $Y(t, x(\cdot); c_H(\cdot))$ is defined by \eqref{Y} with $c(\cdot) = c_H(\cdot)$, where the function $c_H(\cdot)$ is taken from condition (A.4) of Assumption \ref{assumption_H_sigma}.

        Put $E \coloneq [t, T) \cap E^{\rm (A.2)} \cap E^{\rm (A.4)} \cap E_\varphi$, where $E^{\rm (A.2)}$, $E^{\rm (A.4)}$, $E_\varphi$ are the sets from conditions (A.2), (A.4), ($\Phi.2$) respectively.
        Note that $\mu(E) = T - t$.
        Let $A$ be the set of all points $(\tau, y(\cdot)) \in E \times C([- h, T], \mathbb{R}^n)$ such that $\nabla \varphi(\tau, y(\cdot \wedge \tau)) \neq s$.
        Consider the mapping
        \begin{equation*}
            f(\tau, y(\cdot))
            \coloneq \frac{H \bigl( \tau, y(\cdot \wedge \tau), \nabla \varphi(\tau, y(\cdot \wedge \tau)) \bigr)
            - H(\tau, y(\cdot \wedge \tau), s)}{\|\nabla \varphi(\tau, y(\cdot \wedge \tau)) - s\|^2}
            \bigl( \nabla \varphi(\tau, y(\cdot \wedge \tau)) - s \bigr),
        \end{equation*}
        where $(\tau, y(\cdot)) \in A$.
        It can be directly verified that, for every $\tau \in E$, the set $A_\tau \coloneq \{y(\cdot) \in C([- h, T], \mathbb{R}^n) \colon (\tau, y(\cdot)) \in A \}$ is open, the mapping $y(\cdot) \mapsto f(\tau, y(\cdot))$, $A_\tau \to \mathbb{R}^n$, is continuous, and the inequality below holds for all $y(\cdot) \in A_\tau$:
        \begin{equation*}
            \|f(\tau, y(\cdot))\|
            \leq c_H(\tau) (1 + \|y(\cdot \wedge \tau)\|_\infty);
        \end{equation*}
        in addition, for every $y(\cdot) \in C([- h, T], \mathbb{R}^n)$, the set $A^{y(\cdot)} \coloneq \{\tau \in E \colon (\tau, y(\cdot)) \in A\}$ and the function $\tau \mapsto f(\tau, y(\cdot))$, $A^{y(\cdot)} \to \mathbb{R}^n$, are measurable.

        Now, consider a multivalued mapping $F \colon [t, T] \times C([- h, T], \mathbb{R}^n) \multimap \mathbb{R}^n$ defined for every point $(\tau, y(\cdot)) \in [t, T] \times C([- h, T], \mathbb{R}^n)$ by
        \begin{equation*}
            F(\tau, y(\cdot))
            \coloneq \begin{cases}
                \{f(\tau, y(\cdot))\} & \mbox{if } (\tau, y(\cdot)) \in A, \\
                B(c_H(\tau)(1 + \|y(\cdot \wedge \tau)\|_\infty)) & \mbox{otherwise},
            \end{cases}
        \end{equation*}
        where $B(R)$ denotes the closed ball in $\mathbb{R}^n$ with center at the origin and radius $R \geq 0$.
        By construction, $F$ is non-anticipative and has non-empty, compact, and convex values.
        From the properties of the set $A$ and the mapping $f$ listed above, it follows that, for every $\tau \in [t, T]$, the multivalued mapping $y(\cdot) \mapsto F(\tau, y(\cdot))$, $C([- h, T], \mathbb{R}^n) \multimap \mathbb{R}^n$, is upper semicontinuous and, for any $y(\cdot) \in C([- h, T], \mathbb{R}^n)$,
        \begin{equation*}
            \max \bigl\{ \|f\| \colon f \in F(\tau, y(\cdot)) \bigr\}
            \leq c_H(\tau) (1 + \|y(\cdot \wedge \tau)\|_\infty);
        \end{equation*}
        furthermore, the multivalued mapping $\tau \mapsto F(\tau, y(\cdot))$, $[t, T] \multimap \mathbb{R}^n$, is measurable for every $y(\cdot) \in C([- h, T], \mathbb{R}^n)$.
        Therefore, applying a result on the existence of solutions of functional differential inclusions of retarded type (see, e.g., \cite[Theorem 2.1]{Obukhovskii_1992}), we obtain that there exists a function $y(\cdot) \in Y(t, x(\cdot); c_H(\cdot))$ satisfying the inclusion $\dot{y}(\tau) \in F(\tau, y(\cdot))$ for a.e. $\tau \in [t, T]$.

        By construction, we have
        \begin{equation*}
            \langle \nabla \varphi(\tau, y(\cdot \wedge \tau)) - s, \dot{y}(\tau) \rangle
            = H \bigl( \tau, y(\cdot \wedge \tau), \nabla \varphi(\tau, y(\cdot \wedge \tau)) \bigr)
            - H(\tau, y(\cdot \wedge \tau), s)
        \end{equation*}
        for a.e. $\tau \in [t, T]$.
        Hence, in view of equalities \eqref{H_non-anticipative} and \eqref{varphi_derivatives_non-anticipative}, we get
        \begin{equation} \label{theorem_semiclassical_proof_1}
            \langle \nabla \varphi(\tau, y(\cdot)) - s, \dot{y}(\tau) \rangle
            = H \bigl( \tau, y(\cdot), \nabla \varphi(\tau, y(\cdot)) \bigr)
            - H(\tau, y(\cdot), s)
        \end{equation}
        for a.e. $\tau \in [t, T]$.
        Consequently, for every $\tau \in [t, T]$ such that equality \eqref{theorem_semiclassical_proof_1} holds and $\varphi$ satisfies the Hamilton--Jacobi equation \eqref{HJ} at the point $(\tau, y(\cdot))$, i.e., for a.e. $\tau \in [t, T]$, we derive
        \begin{equation*}
            \partial_t \varphi(\tau, y(\cdot))
            = - H \bigl( \tau, y(\cdot), \nabla \varphi(\tau, y(\cdot)) \bigr)
            = \langle s - \nabla \varphi(\tau, y(\cdot)), \dot{y}(\tau) \rangle
            - H(\tau, y(\cdot), s).
        \end{equation*}

        Thus, applying the functional chain rule (see Proposition \ref{proposition_Phi}), we get
        \begin{equation} \label{theorem_semiclassical_proof_2}
            \varphi(\tau, y(\cdot)) - \varphi(t, x(\cdot))
            = \int_{t}^{\tau} \bigl( \langle s, \dot{y}(\xi) \rangle
            - H(\xi, y(\cdot), s) \bigr) \, \rd \xi
        \end{equation}
        for all $\tau \in [t, T)$.
        Due to continuity of $\varphi$ and conditions (A.4) and (A.5), by passing to the limit as $\tau \to T^-$ in \eqref{theorem_semiclassical_proof_2}, we conclude that \eqref{theorem_semiclassical_proof_2} holds for $\tau = T$ as well.
        So, inequalities \eqref{U.1_inequality} and \eqref{L.1_inequality} are valid for all $\tau \in [t, T]$, which completes the proof.
    \end{proof}

    Our next goal is to prove that a minimax solution $\varphi$ of the Cauchy problem \eqref{HJ}, \eqref{boundary_condition} satisfies the Hamilton--Jacobi equation \eqref{HJ} at points of $ci$-differentiability.
    To this end, we need to have stronger properties of $\varphi$ than (U) and (L) (see Definitions \ref{definition_U} and \ref{definition_L}).
    \begin{lemma} \label{lemma_stability_strong}
        Let conditions {\rm(A.1)}--{\rm(A.5)} of Assumption {\rm \ref{assumption_H_sigma}} hold and let a functional $\varphi \colon [0, T] \times C([- h, T], \mathbb{R}^n) \to \mathbb{R}$ be non-anticipative.
        Then, if $\varphi$ is lower semicontinuous and possesses property {\rm (U)}, then, for any $(t, x(\cdot)) \in [0, T) \times C([- h, T], \mathbb{R}^n)$, $s \in \mathbb{R}^n$, there exists $y(\cdot) \in Y(t, x(\cdot); c_H(\cdot))$ such that inequality \eqref{U.1_inequality} is valid for all $\tau \in (t, T]$.
        Analogously, if $\varphi$ is upper semicontinuous and possesses property {\rm (L)}, then, for any $(t, x(\cdot)) \in [0, T) \times C([- h, T], \mathbb{R}^n)$, $s \in \mathbb{R}^n$, there exists $y(\cdot) \in Y(t, x(\cdot); c_H(\cdot))$ such that inequality \eqref{L.1_inequality} is valid for all $\tau \in (t, T]$.
    \end{lemma}
    \begin{proof}
        We prove the first statement only, since the proof for the second one is similar.
        Fix $(t, x(\cdot)) \in [0, T) \times C([- h, T], \mathbb{R}^n)$, $s \in \mathbb{R}^n$.

        Let $k \in \mathbb{N}$ and $t_{k, i} \coloneq t + (T - t) i / k$ for all $i \in \overline{0, k}$.
        By property (U), there are functions $y_{k, i}(\cdot)$, $i \in \overline{0, k}$, such that $y_{k, 0}(\cdot) = x(\cdot)$ and, for every $i \in \overline{1, k}$, the inclusion $y_{k, i}(\cdot) \in Y(t_{k, i - 1}, y_{k, i - 1}(\cdot); c_H(\cdot))$ and the inequality below take place:
        \begin{equation*}
            \varphi(t_{k, i}, y_{k, i}(\cdot))
            - \int_{t_{k, i - 1}}^{t_{k, i}}
            \bigl( \langle s, \dot{y}_{k, i}(\xi) \rangle
            - H(\xi, y_{k, i}(\cdot), s) \bigr) \, \rd \xi
            \leq \varphi(t_{k, i - 1}, y_{k, i - 1}(\cdot)).
        \end{equation*}
        Denote $y_k (\cdot) \coloneq y_{k, k}(\cdot)$.
        For every $i \in \overline{0, k}$, we have $y_k(\cdot \wedge t_{k, i}) = y_{k, i}(\cdot \wedge t_{k, i})$, and, hence, $\varphi(t_{k, i}, y_k(\cdot)) = \varphi(t_{k, i}, y_{k, i}(\cdot))$ since the functional $\varphi$ is non-anticipative and, moreover, $H(\xi, y_{k, i}(\cdot), s) = H (\xi, y_k(\cdot), s)$ for a.e. $\xi \in [t, t_{k, i}]$ according to  \eqref{H_non-anticipative}.
        Thus, $y_k(\cdot) \in Y(t, x(\cdot); c_H(\cdot))$ and, for every $i \in \overline{0, k}$,
        \begin{equation*}
            \varphi(t_{k, i}, y_k(\cdot))
            - \int_{t}^{t_{k, i}}
            \bigl( \langle s, \dot{y}_k(\xi) \rangle
            - H(\xi, y_k(\cdot), s) \bigr) \, \rd \xi
            \leq \varphi(t, x(\cdot)).
        \end{equation*}

        Thanks to compactness of the set $Y(t, x(\cdot); c_H(\cdot))$, by passing to a subsequence if necessary, we can assume that $\|y_k(\cdot) - y(\cdot)\|_\infty \to 0$ as $k \to \infty$ for some function $y(\cdot) \in Y(t, x(\cdot); c_H(\cdot))$.
        Now, let $\tau \in (t, T]$ be fixed.
        Then, for every $k \in \mathbb{N}$, denoting $t_k \coloneq \max\{ t_{k, i} \colon t_{k, i} \leq \tau, \, i \in \overline{0, k}\}$, we get
        \begin{equation} \label{lemma_stability_strong_proof}
            \varphi(t_k, y_k(\cdot))
            - \int_{t}^{t_k}
            \bigl( \langle s, \dot{y}_k(\xi) \rangle
            - H(\xi, y_k(\cdot), s) \bigr) \, \rd \xi
            \leq \varphi(t, x(\cdot)).
        \end{equation}
        We have $t_k \to \tau$ as $k \to \infty$, and, hence,
        \begin{equation*}
            \liminf_{k \to \infty} \varphi(t_k, y_k(\cdot))
            \geq \varphi(\tau, y(\cdot))
        \end{equation*}
        by lower semicontinuity of $\varphi$ and
        \begin{align*}
            \lim_{k \to \infty} \int_{t}^{t_k} \langle s, \dot{y}_k(\xi) \rangle \, \rd \xi
            & = \lim_{k \to \infty} \langle s, y_k(t_k) - y_k(t) \rangle \\
            & = \langle s, y(\tau) - y(t) \rangle \\
            & = \int_{t}^{\tau} \langle s, \dot{y}(\xi) \rangle \, \rd \xi.
        \end{align*}
        Moreover, taking the corresponding functions $\lambda_H(\cdot) \coloneq \lambda_H (\cdot; Y(t, x(\cdot); c_H(\cdot)))$ and $m_H(\cdot) \coloneq m_H(\cdot; Y(t, x(\cdot); c_H(\cdot)))$ from conditions (A.3) and (A.5) respectively and using condition (A.4) as well, we derive
        \begin{align*}
            & \biggl| \int_{t}^{t_k} H(\xi, y_k(\cdot), s) \, \rd \xi
            - \int_{t}^{\tau} H(\xi, y(\cdot), s) \, \rd \xi \biggr| \\
            & \leq \int_{t}^{t_k} |H(\xi, y_k(\cdot), s) - H(\xi, y(\cdot), s)| \, \rd \xi
            + \int_{t_k}^{\tau} |H(\xi, y(\cdot), s) - H(\xi, y(\cdot), 0)| \, \rd \xi \\
            & \quad + \int_{t_k}^{\tau} |H(\xi, y(\cdot), 0)| \, \rd \xi \\
            & \leq (1 + \|s\|) \|y_k(\cdot) - y(\cdot)\|_\infty \int_{t}^{t_k} \lambda_H(\xi) \, \rd \xi
            + (1 + \|y(\cdot)\|_\infty) \|s\| \int_{t_k}^{\tau} c_H(\xi) \, \rd \xi \\
            & \quad + \int_{t_k}^{\tau} m_H(\xi) \, \rd \xi
        \end{align*}
        for all $k \in \mathbb{N}$, which yields
        \begin{equation*}
            \lim_{k \to \infty} \int_{t}^{t_k} H(\xi, y_k(\cdot), s) \, \rd \xi
            = \int_{t}^{\tau} H(\xi, y(\cdot), s) \, \rd \xi.
        \end{equation*}
        Consequently, by passing to the inferior limit in \eqref{lemma_stability_strong_proof} as $k \to \infty$, we arrive at \eqref{U.1_inequality} and complete the proof.
    \end{proof}

    The second result of this section is the following.
    \begin{theorem} \label{theorem_consistency_compacts}
        Let Assumption {\rm \ref{assumption_H_sigma}} hold and let a set $S$ be the union of compact sets $D_k \subset C([- h, T], \mathbb{R}^n)$ over $k \in \mathbb{N}$.
        Then, there exists a set $E \coloneq E(S) \subset [0, T)$ with $\mu(E) = T$ and such that a minimax solution $\varphi$ of the Cauchy problem \eqref{HJ}, \eqref{boundary_condition} satisfies the Hamilton--Jacobi equation \eqref{HJ} at every point $(t, x(\cdot)) \in E \times S$ where $\varphi$ is $ci$-differentiable.
    \end{theorem}
    \begin{proof}
        For every $k \in \mathbb{N}$, denote by $Y_k$ the union of the sets $Y(t, x(\cdot); c_H(\cdot))$ over $(t, x(\cdot)) \in [0, T] \times D_k$.
        Recall that $c_H(\cdot)$ is the function from condition (A.4).
        Since $D_k$ is compact, it follows from Proposition \ref{proposition_1} that the set $Y_k$ is compact.
        Let $Y$ be the union of $Y_k$ over $k \in \mathbb{N}$.
        Note that $S \subset Y$.
        By Corollary \ref{corollary_H_non-anticipative} and Lemma \ref{lemma_Lebesgue_point}, there exists a set $E_1 \subset [0, T)$ with $\mu(E_1) = T$ and such that equalities \eqref{H_non-anticipative} and \eqref{lemma_Lebesgue_point_main} hold for all $t \in E_1$, $x(\cdot) \in Y$, $s \in \mathbb{R}^n$.
        Let $E_2$ be the set of all $t \in [0, T)$ such that $t$ is a Lebesgue point of the function $c_H(\cdot)$.
        Put $E \coloneq E_1 \cap E_2$ and note that $\mu(E) = T$.

        Fix $(t, x(\cdot)) \in E \times S$ and suppose that the minimax solution $\varphi$ is $ci$-differentiable at $(t, x(\cdot))$.
        Since $\varphi$ is an upper solution, and thanks to Lemma \ref{lemma_stability_strong}, there exists a function $y(\cdot) \in Y(t, x(\cdot); c_H(\cdot))$ such that, for all $\tau \in (t, T]$,
        \begin{equation*}
            \varphi(\tau, y(\cdot)) - \int_{t}^{\tau} \Bigl( \langle \nabla \varphi(t, x(\cdot)), \dot{y}(\xi) \rangle
            - H \bigl( \xi, y(\cdot), \nabla \varphi(t, x(\cdot)) \bigr) \Bigr) \, \rd \xi
            \leq \varphi(t, x(\cdot)).
        \end{equation*}
        Due to $ci$-differentiability of $\varphi$ at $(t, x(\cdot))$, there exists a function $o \colon (0, + \infty) \to \mathbb{R}$ such that equality \eqref{ci-differentiability} holds for all $\tau \in (t, T]$ and relation \eqref{o} is valid.
        Thus, for every $\tau \in (t, T]$,
        \begin{equation} \label{theorem_consistency_compacts_proof}
            \partial_t \varphi(t, x(\cdot)) (\tau - t)
            + \int_{t}^{\tau} H \bigl( \xi, y(\cdot), \nabla \varphi(t, x(\cdot)) \bigr) \, \rd \xi
            \leq - o(\tau - t).
        \end{equation}

        Owing to the inclusions $t \in E_1$ and $x(\cdot)$, $y(\cdot) \in Y$, we have
        \begin{align*}
            \lim_{\tau \to t^+}
            \frac{1}{\tau - t}
            \int_{t}^{\tau} H \bigl( \xi, y(\cdot), \nabla \varphi(t, x(\cdot)) \bigr) \, \rd \xi
            & = H \bigl( t, y(\cdot), \nabla \varphi(t, x(\cdot)) \bigr) \\
            & = H \bigl( t, y(\cdot \wedge t), \nabla \varphi(t, x(\cdot)) \bigr) \\
            & = H \bigl( t, x(\cdot \wedge t), \nabla \varphi(t, x(\cdot)) \bigr) \\
            & = H \bigl( t, x(\cdot), \nabla \varphi(t, x(\cdot)) \bigr).
        \end{align*}
        Further, in view of the inclusion $t \in E_2$, there exists $M > 0$ such that the first inequality in \eqref{M_choice} is valid for all $\tau \in (t, T]$.
        Hence, for every $\tau \in (t, T]$, we derive
        \begin{align*}
            \int_{t}^{\tau} \|\dot{y}(\xi)\| \, \rd \xi
            & \leq \int_{t}^{\tau} c_H(\xi) (1 + \|y(\cdot \wedge \xi)\|_\infty) \, \rd \xi \\
            & \leq (1 + \|y(\cdot)\|_\infty) \int_{t}^{\tau} c_H(\xi) \, \rd \xi \\
            & \leq M (1 + \|y(\cdot)\|_\infty) (\tau - t),
        \end{align*}
        which implies that $\lim_{\tau \to t^+} o(\tau - t) / (\tau - t) = 0$.
        Thus, if we divide \eqref{theorem_consistency_compacts_proof} by $\tau - t$ and then pass to the limit as $\tau \to t^+$, we obtain
        \begin{equation*}
            \partial_t \varphi(t, x(\cdot))
            + H \bigl( t, x(\cdot), \nabla \varphi(t, x(\cdot)) \bigr)
            \leq 0.
        \end{equation*}

        In a similar way, we can deduce that
        \begin{equation*}
            \partial_t \varphi(t, x(\cdot))
            + H \bigl( t, x(\cdot), \nabla \varphi(t, x(\cdot)) \bigr)
            \geq 0.
        \end{equation*}
        Consequently, $\varphi$ satisfies the Hamilton--Jacobi equation \eqref{HJ} at the point $(t, x(\cdot))$, and the proof is complete.
    \end{proof}

    Let us recall that the space $C([- h, T], \mathbb{R}^n)$ cannot be represented as a countable union of its compact subsets.
    That is why a set $S$ appears in Theorem \ref{theorem_consistency_compacts} instead of the whole space $C([- h, T], \mathbb{R}^n)$.
    However, if we strengthen condition (A.3) of Assumption \ref{assumption_H_sigma} accordingly, we can obtain a more conventional consistency result, which corresponds to, e.g., \cite[Theorem 12.5, ii)]{Tran_Mikio_Nguyen_2000} in the non-path-dependent case.
    Namely, let us consider the following assumption.
    \begin{assumption} \label{assumption_A.3_strong}
        For every bounded set $D \subset C([- h, T], \mathbb{R}^n)$, there exists a non-negative function $\lambda_H(\cdot) \coloneq \lambda_H(\cdot; D) \in \mathcal{L}_1([0, T], \mathbb{R})$ with the following property:
        for any $x_1(\cdot)$, $x_2(\cdot) \in D$, $s \in \mathbb{R}^n$, there exists a set $E \coloneq E(D, x_1(\cdot), x_2(\cdot), s) \subset [0, T]$ with $\mu(E) = T$ and such that inequality \eqref{A.3_inequality} holds for all $t \in E$.
    \end{assumption}

    In this case, the following theorem is valid.
    \begin{theorem} \label{theorem_consistency}
        Under conditions {\rm (A.1)}, {\rm (A.2)}, {\rm (A.4)}--{\rm (A.6)} of Assumption {\rm \ref{assumption_H_sigma}} and Assumption {\rm \ref{assumption_A.3_strong}}, there exists a set $E \subset [0, T)$ with $\mu(E) = T$ and such that a minimax solution $\varphi$ of the Cauchy problem \eqref{HJ}, \eqref{boundary_condition} satisfies the Hamilton--Jacobi equation \eqref{HJ} at every point $(t, x(\cdot)) \in E \times C([- h, T], \mathbb{R}^n)$ where $\varphi$ is $ci$-differentiable.
    \end{theorem}

    The proof of Theorem \ref{theorem_consistency} is completely analogous to that of Theorem \ref{theorem_consistency_compacts} and is therefore omitted.

\section{Comparison principle and uniqueness}
\label{section_comparison}

    Suppose that Hamiltonians $H_1$ and $H_2$ satisfy conditions (A.1)--(A.5) (see Assumption \ref{assumption_H_sigma}) and let $c_{H_1}(\cdot)$ and $c_{H_2}(\cdot)$ be the corresponding functions from condition (A.4).
    Consider a non-anticipative upper semicontinuous functional $\varphi_1 \colon [0, T] \times C([- h, T], \mathbb{R}^n) \to \mathbb{R}$ possessing property (L) with $H = H_1$ (see Definition \ref{definition_L}) and a non-anticipative lower semicontinuous functional $\varphi_2 \colon [0, T] \times C([- h, T], \mathbb{R}^n) \to \mathbb{R}$ possessing property (U) with $H = H_2$ (see Definition \ref{definition_U}).
    Fix $(t, x(\cdot)) \in [0, T) \times C([- h, T], \mathbb{R}^n)$ and define the set $Y(t, x(\cdot); c_{H_1}(\cdot))$ according to \eqref{Y} with $c(\cdot) = c_{H_1}(\cdot)$.

    \begin{theorem} \label{theorem_comparison_principle}
        Let the above assumptions be fulfilled.
        Suppose that, for every $y(\cdot) \in Y(t, x(\cdot); c_{H_1}(\cdot))$, there exists a set $E \coloneq E(y(\cdot)) \subset [t, T]$ with $\mu(E) = T - t$ and such that, for any $\tau \in E$ and $s \in \mathbb{R}^n$,
        \begin{equation} \label{theorem_comparison_principle_main_H_1-H_2}
            H_1(\tau, y(\cdot), s)
            \leq H_2(\tau, y(\cdot), s).
        \end{equation}
        Suppose also that there exists $\vartheta \in (t, T]$ such that
        \begin{equation} \label{theorem_comparison_principle_main_vartheta}
            \varphi_1(\vartheta, y(\cdot))
            \leq \varphi_2(\vartheta, y(\cdot))
        \end{equation}
        for all $y(\cdot) \in Y(t, x(\cdot); c_{H_1}(\cdot))$.
        Then, the following inequality is valid:
        \begin{equation} \label{theorem_comparison_principle_main}
            \varphi_1(t, x(\cdot))
            \leq \varphi_2(t, x(\cdot)).
        \end{equation}
    \end{theorem}
    \begin{proof}
        The proof follows the scheme of the proof of \cite[Lemma 1]{Gomoyunov_Lukoyanov_RMS_2024}.
        For convenience, we split the proof into four steps.

        \smallskip

        {\it Step 1.}
        Let $c_{H_2}(\cdot)$ be the function from condition (A.4) for $H = H_2$ and consider the set $Y(t, x(\cdot); c_{H_2}(\cdot))$.
        Denote $Y_1 \coloneq Y(t, x(\cdot); c_{H_1}(\cdot))$, $Y_2 \coloneq Y(t, x(\cdot); c_{H_2}(\cdot))$, $D \coloneq Y_1 \cup Y_2$, $R \coloneq \max\{ \|y(\cdot)\|_\infty \colon y(\cdot) \in D\}$.
        Take the corresponding functions $\lambda_{H_1}(\cdot) \coloneq \lambda_{H_1}(\cdot; D)$ and $\lambda_{H_2}(\cdot) \coloneq \lambda_{H_2}(\cdot; D)$ from condition (A.3).
        Let $E_1$ be the set of all $\tau \in [0, T)$ such that $\frac{\rd}{\rd \tau} \int_{0}^{\tau} \lambda_{H_2}(\xi) \, \rd \xi = \lambda_{H_2}(\tau)$ and note that $\mu(E_1) = T$.
        Using condition (A.4) and Corollary \ref{corollary_assumptions_strong}, choose a set $E_2 \subset [t, T]$ with $\mu(E_2) = T - t$ and such that, for all $\tau \in E_2$, $y_1(\cdot)$, $y_2(\cdot) \in D$, $s_1$, $s_2 \in \mathbb{R}^n$, $i \in \{1, 2\}$,
        \begin{equation} \label{H_i_basic}
            \begin{aligned}
                & | H_i(\tau, y_1(\cdot), s_1) - H_i(\tau, y_2(\cdot), s_2) | \\
                & \leq \lambda_{H_i}(\tau) (1 + \|s_2\|) \|y_1(\cdot \wedge \tau) - y_2(\cdot \wedge \tau)\|_\infty \\
                & \quad + c_{H_i}(\tau) (1 + \|y_1(\cdot \wedge \tau)\|_\infty) \|s_1 - s_2\|.
            \end{aligned}
        \end{equation}

        Following, e.g., \cite{Zhou_2020_1} (see also \cite{Gomoyunov_Lukoyanov_Plaksin_2021}), consider the mappings
        \begin{equation*}
            \gamma (\tau, w(\cdot))
            \coloneq \begin{cases}
                \displaystyle
                \frac{\bigl( \|w(\cdot \wedge \tau)\|_\infty^2 - \|w(\tau)\|^2 \bigr)^2}{\|w(\cdot \wedge \tau)\|_\infty^2} + \|w(\tau)\|^2
                & \text{if } \|w(\cdot \wedge \tau)\|_\infty > 0, \\
                0 & \text{otherwise}
              \end{cases}
        \end{equation*}
        and
        \begin{equation*}
            q(\tau, w(\cdot))
            \coloneq \begin{cases}
                \displaystyle
                \biggl( 2 - 4 \frac{\|w(\cdot \wedge \tau)\|_\infty^2 - \|w(\tau)\|^2}{\|w(\cdot \wedge \tau)\|_\infty^2} \biggr) w(\tau)
                & \text{if } \|w(\cdot \wedge \tau)\|_\infty > 0, \\
                0 & \text{otherwise},
            \end{cases}
        \end{equation*}
        where $(\tau, w(\cdot)) \in [0, T] \times C([- h, T], \mathbb{R}^n)$.
        In accordance with \cite[Appendix B]{Gomoyunov_Lukoyanov_Plaksin_2021} and Remark \ref{remark_ci-smooth}, the mapping $q \colon [0, T] \times C([- h, T], \mathbb{R}^n) \to \mathbb{R}^n$ is continuous, the functional $\gamma$ belongs to the set $\Phi$ (see Section \ref{section_consistency}) with $E_\gamma = [0, T)$, and its $ci$-derivatives are given by
        \begin{equation*}
            \partial_t \gamma(\tau, w(\cdot))
            = 0,
            \quad \nabla \gamma(\tau, w(\cdot))
            = q(\tau, w(\cdot))
        \end{equation*}
        for all $(\tau, w(\cdot)) \in [0, T) \times C([- h, T], \mathbb{R}^n)$.
        Moreover (see, e.g., \cite[Lemma 2.3]{Zhou_2020_1} and \cite[Section 4.1]{Gomoyunov_Lukoyanov_Plaksin_2021}), $\gamma$ and $q$ satisfy the inequalities
        \begin{equation} \label{gamma_bound}
            \gamma(\tau, w(\cdot))
            \geq \varkappa \|w(\cdot \wedge \tau)\|_\infty^2,
            \quad \varkappa \coloneq (3 - \sqrt{5}) / 2,
            \quad \|q(\tau, w(\cdot))\|
            \leq 2 \|w(\tau)\|
        \end{equation}
        for all $(\tau, w(\cdot)) \in [0, T] \times C([- h, T], \mathbb{R}^n)$.

        Now, take the number $\varkappa$ from \eqref{gamma_bound} and put
        \begin{equation*}
            \varepsilon_0
            \coloneq \frac{1}{\sqrt{\varkappa}} \exp
            \biggl( - \frac{\|\lambda_{H_2}(\cdot)\|_1}{\varkappa} \biggr).
        \end{equation*}
        For every $\varepsilon \in (0, \varepsilon_0]$, consider the Lyapunov--Krasovskii functional
        \begin{equation} \label{nu_varepsilon}
            \nu_\varepsilon(\tau, w(\cdot))
            \coloneq \frac{\sqrt{\varepsilon^4 + \gamma(\tau, w(\cdot))}}{\varepsilon} \biggl( \exp \biggl(- \frac{1}{\varkappa} \int_{0}^{\tau} \lambda_{H_2}(\xi) \, \rd \xi \biggr) - \varepsilon \sqrt{\varkappa} \biggr)
        \end{equation}
        and the auxiliary mapping
        \begin{equation*}
            s_\varepsilon(\tau, w(\cdot))
            \coloneq \frac{q(\tau, w(\cdot))}{2 \varepsilon \sqrt{\varepsilon^4 + \gamma(\tau, w(\cdot))}}
            \biggl( \exp \biggl(- \frac{1}{\varkappa} \int_{0}^{\tau} \lambda_{H_2}(\xi) \, \rd \xi \biggr) - \varepsilon \sqrt{\varkappa} \biggr),
        \end{equation*}
        where $(\tau, w(\cdot)) \in [0, T] \times C([- h, T], \mathbb{R}^n)$.
        Then, $s_\varepsilon \colon [0, T] \times C([- h, T], \mathbb{R}^n) \to \mathbb{R}^n$ is continuous, $\nu_\varepsilon$ is non-negative, and
        \begin{equation} \label{nu_varepsilon_upper}
            \nu_\varepsilon(t, w(\cdot) \equiv 0) \leq \varepsilon.
        \end{equation}

        Further, we have $\nu_\varepsilon \in \Phi$ with $E_{\nu_\varepsilon} = E_1$ and
        \begin{equation} \label{nu_varepsilon_derivatives}
            \begin{aligned}
                \partial_t \nu_\varepsilon(\tau, w(\cdot))
                & = - \frac{\lambda_{H_2}(\tau) \sqrt{\varepsilon^4 + \gamma(\tau, w(\cdot))}}{\varepsilon \varkappa} \exp \biggl(- \frac{1}{\varkappa} \int_{0}^{\tau} \lambda_{H_2}(\xi) \, \rd \xi \biggr), \\
                \nabla \nu_\varepsilon(\tau, w(\cdot))
                & = s_\varepsilon(\tau, w(\cdot))
            \end{aligned}
        \end{equation}
        for all $(\tau, w(\cdot)) \in E_{\nu_\varepsilon} \times C([- h, T], \mathbb{R}^n)$.

        Finally, according to \eqref{gamma_bound}, for any $\tau \in E_2 \cap E_{\nu_\varepsilon}$, $y_1(\cdot)$, $y_2(\cdot) \in D$, we derive
        \begin{align*}
            \|y_1(\cdot \wedge \tau) - y_2(\cdot \wedge \tau)\|_\infty
            & \leq \sqrt{\varepsilon^4
            + \gamma(\tau, y_1(\cdot) - y_2(\cdot))} / \sqrt{\varkappa}, \\
            \|q(\tau, y_1(\cdot) - y_2(\cdot))\|
            & \leq 2 \|y_1(\tau) - y_2(\tau)\| \\
            & \leq 2 \sqrt{\varepsilon^4 + \gamma(\tau, y_1(\cdot) - y_2(\cdot))} / \sqrt{\varkappa},
        \end{align*}
        and, therefore (see also \eqref{H_i_basic}),
        \begin{equation} \label{nu_varepsilon_main_inequality}
            \begin{aligned}
                & H_2 \bigl( \tau, y_1(\cdot), s_\varepsilon(\tau, y_1(\cdot) - y_2(\cdot)) \bigr)
                - H_2 \bigl( \tau, y_2(\cdot), s_\varepsilon(\tau, y_1(\cdot) - y_2(\cdot)) \bigr) \\
                & \leq \lambda_{H_2}(\tau) \biggl( 1 + \frac{\|q(\tau,  y_1(\cdot) - y_2(\cdot))\|}{2 \varepsilon \sqrt{\varepsilon^4 + \gamma(\tau, y_1(\cdot) - y_2(\cdot))}}
                \biggl( \exp \biggl(- \frac{1}{\varkappa} \int_{0}^{\tau} \lambda_{H_2}(\xi) \, \rd \xi \biggr) - \varepsilon \sqrt{\varkappa} \biggr) \biggr) \\
                & \quad \times \|y_1(\cdot \wedge \tau) - y_2(\cdot \wedge \tau)\|_\infty \\
                & \leq \frac{\lambda_{H_2}(\tau)}{\sqrt{\varkappa}}
                \biggl( 1 + \frac{1}{\varepsilon \sqrt{\varkappa}}
                \biggl( \exp \biggl(- \frac{1}{\varkappa} \int_{0}^{\tau} \lambda_{H_2}(\xi) \, \rd \xi \biggr) - \varepsilon \sqrt{\varkappa} \biggr) \biggr) \\
                & \quad \times \sqrt{\varepsilon^4 + \gamma(\tau, y_1(\cdot) - y_2(\cdot))} \\
                & = - \partial_t \nu_\varepsilon(\tau, y_1(\cdot) - y_2(\cdot)).
            \end{aligned}
        \end{equation}

        \smallskip

        {\it Step 2.}
        Fix $\varepsilon \in (0, \varepsilon_0]$ and denote by $\mathcal{W}_\varepsilon$ the set of all triples $(y_1(\cdot), y_2(\cdot), z(\cdot))$ such that $y_1(\cdot) \in Y_1$, $y_2(\cdot) \in Y_2$, $z \colon [t, T] \to \mathbb{R}$ is absolutely continuous,
        \begin{equation} \label{z_initial_condition}
            z(t)
            = \varphi_2(t, x(\cdot)) - \varphi_1(t, x(\cdot)),
        \end{equation}
        and, for a.e. $\tau \in [t, T]$,
        \begin{align*}
            & \bigl| \dot{z}(\tau) + \langle s_\varepsilon(\tau, y_1(\cdot) - y_2(\cdot)),
            \dot{y}_1(\tau) - \dot{y}_2(\tau) \rangle \\
            & \quad - H_1 \bigl( \tau, y_1(\cdot), s_\varepsilon(\tau, y_1(\cdot) - y_2(\cdot)) \bigr)
            + H_2 \bigl( \tau, y_2(\cdot), s_\varepsilon(\tau, y_1(\cdot) - y_2(\cdot)) \bigr) \bigr| \\
            & \leq \varepsilon (c_{H_1}(\tau) + c_{H_2}(\tau)).
        \end{align*}
        Given any $y_1(\cdot) \in Y_1$, $y_2(\cdot) \in Y_2$, we have $(y_1(\cdot), y_2(\cdot), z(\cdot)) \in \mathcal{W}_\varepsilon$ with
        \begin{align*}
            z(\tau)
            & \coloneq \varphi_2(t, x(\cdot)) - \varphi_1(t, x(\cdot))
            - \int_{t}^{\tau}
            \Bigl( \langle s_\varepsilon(\xi, y_1(\cdot) - y_2(\cdot)),
            \dot{y}_1(\xi) - \dot{y}_2(\xi) \rangle \\
            & \quad - H_1 \bigl( \xi, y_1(\cdot), s_\varepsilon(\xi, y_1(\cdot) - y_2(\cdot)) \bigr)
            + H_2 \bigl( \xi, y_2(\cdot), s_\varepsilon(\xi, y_1(\cdot) - y_2(\cdot)) \bigr) \Bigr) \, \rd \xi
        \end{align*}
        for all $\tau \in [t, T]$.
        In particular, $\mathcal{W}_\varepsilon \neq \varnothing$.
        In the remainder of this step, we prove that the set $\mathcal{W}_\varepsilon$ is compact in $C([- h, T], \mathbb{R}^n) \times C([- h, T], \mathbb{R}^n) \times C([t, T], \mathbb{R})$.

        Consider a sequence $(y_1^{[k]}(\cdot), y_2^{[k]}(\cdot), z^{[k]}(\cdot)) \in \mathcal{W}_\varepsilon$, $k \in \mathbb{N}$.
        Since the sets $Y_1$ and $Y_2$ are compact, by passing to a subsequence if necessary, we can assume that $\|y_1^{[k]}(\cdot) - y_1^{[0]}(\cdot)\|_\infty \to 0$ and $\|y_2^{[k]}(\cdot) - y_2^{[0]}(\cdot)\|_\infty \to 0$ as $k \to \infty$ for some functions $y_1^{[0]}(\cdot) \in Y_1$ and $y_2^{[0]}(\cdot) \in Y_2$.
        Denote
        \begin{align*}
            s^{[k]}(\tau)
            & \coloneq s_\varepsilon(\tau, y_1^{[k]}(\cdot) - y_2^{[k]}(\cdot)), \\
            H^{[k]}(\tau)
            & \coloneq H_1(\tau, y_1^{[k]}(\cdot), s^{[k]}(\tau))
            - H_2(\tau, y_2^{[k]}(\cdot), s^{[k]}(\tau)),
        \end{align*}
        where $\tau \in [t, T]$, $k \in \mathbb{N} \cup \{0\}$.
        By continuity of $s_\varepsilon$, we have
        \begin{equation} \label{s^k_convergence}
            \lim_{k \to \infty} \max_{\tau \in [t, T]} \|s^{[k]}(\tau) - s^{[0]}(\tau)\|
            = 0.
        \end{equation}
        In particular, there exists $M > 0$ such that $\|s^{[k]}(\tau)\| \leq M$ for all $\tau \in [t, T]$ and $k \in \mathbb{N} \cup \{0\}$.
        Furthermore, for every $\tau \in E_2$, using \eqref{H_i_basic}, we derive
        \begin{equation} \label{H^k_convergence}
            \begin{aligned}
                & |H^{[k]}(\tau) - H^{[0]}(\tau)| \\
                & \leq (\lambda_{H_1}(\tau) + \lambda_{H_2}(\tau)) (1 + M)
                (\|y_1^{[k]}(\cdot) - y_1^{[0]}(\cdot)\|_\infty
                + \|y_2^{[k]}(\cdot) - y_2^{[0]}(\cdot)\|_\infty) \\
                & \quad + (c_{H_1}(\tau) + c_{H_2}(\tau)) (1 + R) \|s^{[k]}(\tau) - s^{[0]}(\tau)\|.
            \end{aligned}
        \end{equation}

        Let $\delta \geq 0$.
        For every $\tau \in [t, T]$, denote by $F_\delta(\tau)$ the set consisting of all triples $(f_1, f_2, h) \in \mathbb{R}^n \times \mathbb{R}^n \times \mathbb{R}$ such that
        \begin{equation} \label{W_varepsilon_f}
            \|f_1\|
            \leq c_{H_1}(\tau) (1 + R),
            \quad \|f_2\|
            \leq c_{H_2}(\tau) (1 + R),
        \end{equation}
        and
        \begin{equation*}
            \begin{aligned}
                |h + \langle s^{[0]}(\tau), f_1 - f_2 \rangle - H^{[0]}(\tau)|
                & \leq \varepsilon (c_{H_1}(\tau) + c_{H_2}(\tau)) \\
                & \quad + \delta (\lambda_{H_1}(\tau) + \lambda_{H_2}(\tau) + c_{H_1}(\tau) + c_{H_2}(\tau)).
            \end{aligned}
        \end{equation*}
        Note that, for every $\tau \in [t, T]$, the set $F_\delta(\tau)$ is non-empty, convex, and compact, and the inequality
        \begin{align*}
            \|f_1\| + \|f_2\| + |h|
            & \leq (1 + R)(1 + M)(c_{H_1}(\tau) + c_{H_2}(\tau))
            + |H^{[0]}(\tau)| \\
            & \quad + \varepsilon (c_{H_1}(\tau) + c_{H_2}(\tau))
            + \delta (\lambda_{H_1}(\tau) + \lambda_{H_2}(\tau) + c_{H_1}(\tau) + c_{H_2}(\tau))
        \end{align*}
        holds for all $(f_1, f_2, h) \in F_\delta(\tau)$.
        Furthermore, it can be verified directly that the multivalued mapping $\tau \mapsto F_\delta(\tau)$, $[t, T] \multimap \mathbb{R}^n \times \mathbb{R}^n \times \mathbb{R}$, is measurable.

        Fix $\delta > 0$.
        Based on \eqref{s^k_convergence} and \eqref{H^k_convergence}, take $k_\ast \in \mathbb{N}$ such that, for every $k \geq k_\ast$, every $\tau \in E_2$, and any $f_1$, $f_2 \in \mathbb{R}^n$ satisfying \eqref{W_varepsilon_f}, the inequality below is valid:
        \begin{align*}
            & |\langle s^{[0]}(\tau) - s^{[k]}(\tau), f_1 - f_2 \rangle
            + H^{[k]}(\tau) - H^{[0]}(\tau)| \\*
            & \leq \delta (\lambda_{H_1}(\tau) + \lambda_{H_2}(\tau) + c_{H_1}(\tau) + c_{H_2}(\tau)).
        \end{align*}
        Then, for every $k \geq k_\ast$, we obtain $(\dot{y}_1^{[k]}(\tau), \dot{y}_2^{[k]}(\tau), \dot{z}^{[k]}(\tau)) \in F_\delta(\tau)$ for a.e. $\tau \in [t, T]$.

        Consequently, applying a result on the compactness of solution sets of (ordinary) differential inclusions (see, e.g., \cite[Theorem 6.3.3]{Bettiol_Vinter_2024}), we derive that, by passing to a subsequence if necessary, $\max_{\tau \in [t, T]} |z^{[k]}(\tau) - z^{[0]}(\tau)| \to 0$ as $k \to \infty$ for some absolutely continuous function $z^{[0]} \colon [t, T] \to \mathbb{R}$ such that the initial condition \eqref{z_initial_condition} is satisfied and the inclusion $(\dot{y}_1^{[0]}(\tau), \dot{y}_2^{[0]}(\tau), \dot{z}^{[0]}(\tau)) \in F_0(\tau)$ is valid for a.e. $\tau \in [t, T]$.
        So, $(y_1^{[0]}(\cdot), y_2^{[0]}(\cdot), z^{[0]}(\cdot)) \in \mathcal{W}_\varepsilon$, and the proof of compactness of $\mathcal{W}_\varepsilon$ is complete.

        \smallskip

        {\it Step 3.}
        Let us show that a triple $(y_1^{(\varepsilon)}(\cdot), y_2^{(\varepsilon)}(\cdot), z^{(\varepsilon)}(\cdot)) \in \mathcal{W}_\varepsilon$ exists such that
        \begin{equation} \label{proof_lemma_choice}
            z^{(\varepsilon)}(\vartheta)
            \geq \varphi_2(\vartheta, y_2^{(\varepsilon)}(\cdot)) - \varphi_1(\vartheta, y_1^{(\varepsilon)}(\cdot)).
        \end{equation}
        For every $\tau \in [t, \vartheta]$, consider the set
        \begin{equation*}
            \mathcal{M}_\varepsilon(\tau)
            \coloneq \bigl\{ (y_1(\cdot), y_2(\cdot), z(\cdot)) \in \mathcal{W}_\varepsilon \colon
            z(\tau) \geq \varphi_2(\tau, y_2(\cdot)) - \varphi_1(\tau, y_1(\cdot)) \bigr\}.
        \end{equation*}
        Due to the initial condition \eqref{z_initial_condition} and since $\varphi_1$ and $\varphi_2$ are non-anticipative, we have $\mathcal{M}_\varepsilon(t) = \mathcal{W}_\varepsilon \neq \varnothing$.
        Put
        \begin{equation} \label{tau_varepsilon}
            \tau_\varepsilon
            \coloneq \max \bigl\{ \tau \in [t, \vartheta]
            \colon \mathcal{M}_\varepsilon(\tau)
            \neq \varnothing \bigr\}.
        \end{equation}
        The maximum in \eqref{tau_varepsilon} is attained by compactness of $\mathcal{W}_\varepsilon$, upper semicontinuity of $\varphi_1$, and lower semicontinuity of $\varphi_2$.
        In order to prove the statement, it suffices to verify that $\tau_\varepsilon = \vartheta$.
        Arguing by contradiction, assume that $\tau_\varepsilon < \vartheta$.

        Take arbitrarily $(\hat{y}_1(\cdot), \hat{y}_2(\cdot), \hat{z}(\cdot)) \in \mathcal{M}_\varepsilon(\tau_\varepsilon)$ and denote $\hat{s} \coloneq s_\varepsilon (\tau_\varepsilon, \hat{y}_1(\cdot) - \hat{y}_2(\cdot))$.
        In accordance with Lemma \ref{lemma_stability_strong}, there are functions $y_1(\cdot) \in Y(\tau_\varepsilon, \hat{y}_1(\cdot); c_{H_1}(\cdot)) \subset Y_1$, $y_2(\cdot) \in Y(\tau_\varepsilon, \hat{y}_2(\cdot); c_{H_2}(\cdot)) \subset Y_2$ such that
        \begin{equation*}
            \varphi_2(\tau, y_2(\cdot)) - \varphi_1(\tau, y_1(\cdot))
            \leq \varphi_2(\tau_\varepsilon, \hat{y}_2(\cdot))
            - \varphi_1(\tau_\varepsilon, \hat{y}_1(\cdot)) + z(\tau),
        \end{equation*}
        where $\tau \in [\tau_\varepsilon, T]$ and
        \begin{equation*}
            z(\tau)
            \coloneq \int_{\tau_\varepsilon}^{\tau}
            \bigl( - \langle \hat{s}, \dot{y}_1(\xi) - \dot{y}_2(\xi) \rangle
            + H_1(\xi, y_1(\cdot), \hat{s}) - H_2(\xi, y_2(\cdot), \hat{s}) \bigr) \, \rd \xi.
        \end{equation*}
        Since $s_\varepsilon$ is continuous and non-anticipative, we obtain $s_\varepsilon(\tau, y_1(\cdot) - y_2(\cdot)) \to \hat{s}$ as $\tau \to \tau_\varepsilon^{+}$.
        Hence, and using \eqref{H_i_basic}, we conclude that there exists $\delta \in (0, \vartheta - \tau_\varepsilon]$ such that, for a.e. $\tau \in [\tau_\varepsilon, \tau_\varepsilon + \delta]$,
        \begin{align*}
            & \bigl| \dot{z}(\tau) + \langle s_\varepsilon(\tau, y_1(\cdot) - y_2(\cdot)),
            \dot{y}_1(\tau) - \dot{y}_2(\tau) \rangle \\
            & \quad - H_1 \bigl( \tau, y_1(\cdot), s_\varepsilon(\tau, y_1(\cdot) - y_2(\cdot)) \bigr)
            + H_2 \bigl( \tau, y_2(\cdot), s_\varepsilon(\tau, y_1(\cdot) - y_2(\cdot)) \bigr) \bigr| \\
            & \leq \varepsilon (c_{H_1}(\tau) + c_{H_2}(\tau)).
        \end{align*}
        Now, take $y_1^\ast(\cdot) \in Y(\tau_\varepsilon + \delta, y_1(\cdot); c_{H_1}(\cdot)) \subset Y_1$, $y_2^\ast(\cdot) \in Y(\tau_\varepsilon + \delta, y_2(\cdot); c_{H_2}(\cdot)) \subset Y_2$ and consider a function $z^\ast \colon [t, T] \to \mathbb{R}$ that is defined by $z^\ast(\tau) \coloneq \hat{z}(\tau)$ if $\tau \in [t, \tau_\varepsilon]$, $z^\ast(\tau) \coloneq \hat{z}(\tau_\varepsilon) + z(\tau)$ if $\tau \in [\tau_\varepsilon, \tau_\varepsilon + \delta]$, and
        \begin{align*}
            z^\ast(\tau)
            & \coloneq \hat{z}(\tau_\varepsilon) + z(\tau_\varepsilon + \delta)
            - \int_{\tau_\varepsilon + \delta}^{\tau}
            \Bigl( \langle s_\varepsilon(\xi, y_1^\ast(\cdot) - y_2^\ast(\cdot)),
            \dot{y}_1^\ast(\xi) - \dot{y}_2^\ast(\xi) \rangle \\
            & \quad - H_1 \bigl( \xi, y_1^\ast(\cdot), s_\varepsilon(\xi, y_1^\ast(\cdot) - y_2^\ast(\cdot)) \bigr)
            + H_2 \bigl( \xi, y_2^\ast(\cdot), s_\varepsilon(\xi, y_1^\ast(\cdot) - y_2^\ast(\cdot)) \bigr) \Bigr) \, \rd \xi
        \end{align*}
        if $\tau \in (\tau_\varepsilon + \delta, T]$.
        Then, taking equality \eqref{H_non-anticipative} into account and recalling that $s_\varepsilon$ is non-anticipative, we derive $(y_1^\ast(\cdot), y_2^\ast(\cdot), z^\ast(\cdot)) \in \mathcal{W}_\varepsilon$.
        Furthermore, since $\varphi_1$ and $\varphi_2$ are non-anticipative, we have
        \begin{align*}
            \varphi_2(\tau_\varepsilon + \delta, y_2^\ast(\cdot))
            - \varphi_1(\tau_\varepsilon + \delta, y_1^\ast(\cdot))
            & = \varphi_2(\tau_\varepsilon + \delta, y_2(\cdot))
            - \varphi_1(\tau_\varepsilon + \delta, y_1(\cdot)) \\
            & \leq \varphi_2(\tau_\varepsilon, \hat{y}_2(\cdot))
            - \varphi_1(\tau_\varepsilon, \hat{y}_1(\cdot)) + z(\tau_\varepsilon + \delta) \\
            & \leq \hat{z}(\tau_\varepsilon) + z(\tau_\varepsilon + \delta) \\
            & = z^\ast(\tau_\varepsilon + \delta).
        \end{align*}
        Therefore, the inclusion $(y_1^\ast(\cdot), y_2^\ast(\cdot), z^\ast(\cdot)) \in \mathcal{M}_\varepsilon(\tau_\varepsilon + \delta)$ is valid, which contradicts the definition of $\tau_\varepsilon$ (see \eqref{tau_varepsilon}).

        \smallskip

        {\it Step 4.}
        Consider the auxiliary function $\omega (\tau) \coloneq \nu_\varepsilon(\tau, y_1^{(\varepsilon)}(\cdot) - y_2^{(\varepsilon)}(\cdot)) + z^{(\varepsilon)}(\tau)$, where $\tau \in [t, \vartheta]$.
        Fix $\vartheta^\prime \in (t, \vartheta)$.
        Since $\nu_\varepsilon \in \Phi$, the restriction $\omega|_{[t, \vartheta^\prime]}(\cdot)$ of the function $\omega(\cdot)$ to the interval $[t, \vartheta^\prime]$ is absolutely continuous and, by Proposition \ref{proposition_Phi},
        \begin{align*}
            \dot{\omega}(\tau)
            & = \partial_t \nu_\varepsilon(\tau, y_1^{(\varepsilon)}(\cdot) - y_2^{(\varepsilon)}(\cdot)) \\
            & \quad + \langle \nabla \nu_\varepsilon(\tau, y_1^{(\varepsilon)}(\cdot) - y_2^{(\varepsilon)}(\cdot)), \dot{y}_1^{(\varepsilon)}(\tau) - \dot{y}_2^{(\varepsilon)}(\tau) \rangle
            + \dot{z}^{(\varepsilon)}(\tau)
        \end{align*}
        for a.e. $\tau \in [t, \vartheta^\prime]$.
        Due to the inclusion $(y_1^{(\varepsilon)}(\cdot), y_2^{(\varepsilon)}(\cdot), z^{(\varepsilon)}(\cdot)) \in \mathcal{W}_\varepsilon$ and relations \eqref{theorem_comparison_principle_main_H_1-H_2} and \eqref{nu_varepsilon_main_inequality}, we have
        \begin{align*}
            \dot{z}^{(\varepsilon)}(\tau)
            & \leq - \langle s_\varepsilon(\tau, y_1^{(\varepsilon)}(\cdot) - y_2^{(\varepsilon)}(\cdot)),
            \dot{y}_1^{(\varepsilon)}(\tau) - \dot{y}_2^{(\varepsilon)}(\tau) \rangle \\
            & \quad + H_1 \bigl( \tau, y_1^{(\varepsilon)}(\cdot), s_\varepsilon(\tau, y_1^{(\varepsilon)}(\cdot) - y_2^{(\varepsilon)}(\cdot)) \bigr) \\
            & \quad - H_2 \bigl( \tau, y_2^{(\varepsilon)}(\cdot), s_\varepsilon(\tau, y_1^{(\varepsilon)}(\cdot) - y_2^{(\varepsilon)}(\cdot)) \bigr) \\
            & \quad + \varepsilon(c_{H_1}(\tau) + c_{H_2}(\tau)) \\
            & \leq - \langle s_\varepsilon(\tau, y_1^{(\varepsilon)}(\cdot) - y_2^{(\varepsilon)}(\cdot)),
            \dot{y}_1^{(\varepsilon)}(\tau) - \dot{y}_2^{(\varepsilon)}(\tau) \rangle
            - \partial_t \nu_\varepsilon(\tau, y_1^{(\varepsilon)}(\cdot) - y_2^{(\varepsilon)}(\cdot))\\
            & \quad + \varepsilon(c_{H_1}(\tau) + c_{H_2}(\tau))
        \end{align*}
        for a.e. $\tau \in [t, \vartheta^\prime]$.
        Hence, and  thanks to \eqref{nu_varepsilon_derivatives}, we obtain $\dot{\omega}(\tau) \leq \varepsilon(c_{H_1}(\tau) + c_{H_2}(\tau))$ for a.e. $\tau \in [t, \vartheta^\prime]$, which yields
        \begin{equation*}
            \omega(\vartheta^\prime)
            \leq \omega(t) + \varepsilon \int_{t}^{\vartheta^\prime} (c_{H_1}(\tau) + c_{H_2}(\tau)) \, \rd \tau.
        \end{equation*}
        Since this inequality holds for all $\vartheta^\prime \in (t, \vartheta)$ and $\omega(\cdot)$ is continuous by continuity of $\nu_\varepsilon$, we eventually get
        \begin{equation*}
            \omega(\vartheta)
            \leq \omega(t) + \varepsilon \|c_{H_1}(\cdot) + c_{H_2}(\cdot)\|_1.
        \end{equation*}

        In view of the definition of $\omega(\cdot)$, the initial condition \eqref{z_initial_condition}, and inequality \eqref{nu_varepsilon_upper}, recalling that $\nu_\varepsilon$ is non-anticipative, we derive
        \begin{align*}
            & \nu_\varepsilon(\vartheta, y_1^{(\varepsilon)}(\cdot) - y_2^{(\varepsilon)}(\cdot)) + z^{(\varepsilon)}(\vartheta) \\
            & \leq \nu_\varepsilon(t, y_1^{(\varepsilon)}(\cdot) - y_2^{(\varepsilon)}(\cdot)) + z^{(\varepsilon)}(t) + \varepsilon \|c_{H_1}(\cdot) + c_{H_2}(\cdot)\|_1 \\
            & \leq \varphi_2(t, x(\cdot)) - \varphi_1(t, x(\cdot))
            + \varepsilon (1 + \|c_{H_1}(\cdot) + c_{H_2}(\cdot)\|_1).
        \end{align*}
        As a result, according to \eqref{proof_lemma_choice}, we have
        \begin{equation} \label{proof_lemma_estimate}
            \begin{aligned}
                & \nu_\varepsilon(\vartheta, y_1^{(\varepsilon)}(\cdot) - y_2^{(\varepsilon)}(\cdot))
                + \varphi_2(\vartheta, y_2^{(\varepsilon)}(\cdot))
                - \varphi_1(\vartheta, y_1^{(\varepsilon)}(\cdot)) \\
                & \leq \varphi_2(t, x(\cdot)) - \varphi_1(t, x(\cdot))
                + \varepsilon (1 + \|c_{H_1}(\cdot) + c_{H_2}(\cdot)\|_1).
            \end{aligned}
        \end{equation}
        Note that inequality \eqref{proof_lemma_estimate} holds for every $\varepsilon \in (0, \varepsilon_0]$.

        Since $Y_1$ and $Y_2$ are compact, there exist a sequence $\varepsilon_k \in (0, \varepsilon_0]$, $k \in \mathbb{N}$, and functions $y_1^\ast(\cdot) \in Y_1$ and $y_2^\ast(\cdot) \in Y_2$ such that
        \begin{equation*}
            \lim_{k \to \infty} \bigl( \varepsilon_k
            + \|y_1^{(\varepsilon_k)}(\cdot) - y_1^\ast(\cdot)\|_\infty
            + \|y_2^{(\varepsilon_k)}(\cdot) - y_2^\ast(\cdot)\|_\infty \bigr)
            = 0.
        \end{equation*}
        Then, by upper semicontinuity of $\varphi_1$ and lower semicontinuity of $\varphi_2$, it follows from \eqref{proof_lemma_estimate} that
        \begin{equation} \label{proof_lemma_estimate_limsup}
            \begin{aligned}
                & \limsup_{k \to \infty} \nu_{\varepsilon_k} (\vartheta, y_1^{(\varepsilon_k)}(\cdot) - y_2^{(\varepsilon_k)}(\cdot))
                + \varphi_2(\vartheta, y_2^\ast(\cdot))
                - \varphi_1(\vartheta, y_1^\ast(\cdot)) \\
                & \leq \varphi_2(t, x(\cdot)) - \varphi_1(t, x(\cdot)).
            \end{aligned}
        \end{equation}
        In particular, there exists $K > 0$ such that $\nu_{\varepsilon_k} (\vartheta, y_1^{(\varepsilon_k)}(\cdot) - y_2^{(\varepsilon_k)}(\cdot)) \leq K$ for all $k \in \mathbb{N}$.
        Therefore, in view of the first inequality in \eqref{gamma_bound} and \eqref{nu_varepsilon}, we obtain
        \begin{align*}
            & \|y_1^{(\varepsilon_k)}(\cdot \wedge \vartheta) - y_2^{(\varepsilon_k)}(\cdot \wedge \vartheta)\|_\infty \\
            & \leq \biggl( \frac{\varepsilon_k}{\sqrt{\varkappa}} \nu_{\varepsilon_k} (\vartheta, y_1^{(\varepsilon_k)}(\cdot) - y_2^{(\varepsilon_k)}(\cdot))
            + \varepsilon_k \sqrt{\varepsilon_k^4 + \gamma(\vartheta, y_1^{(\varepsilon_k)}(\cdot) - y_2^{(\varepsilon_k)}(\cdot))} \biggr) \exp
            \biggl( \frac{\|\lambda_{H_2}(\cdot)\|_1}{\varkappa} \biggr) \\
            & \leq \biggl( \frac{K \varepsilon_k}{\sqrt{\varkappa}}
            + \varepsilon_k \sqrt{\varepsilon_k^4 + \gamma(\vartheta, y_1^{(\varepsilon_k)}(\cdot) - y_2^{(\varepsilon_k)}(\cdot))} \biggr)
            \exp
            \biggl( \frac{\|\lambda_{H_2}(\cdot)\|_1}{\varkappa} \biggr)
        \end{align*}
        for all $k \in \mathbb{N}$, which implies that $\|y_1^{(\varepsilon_k)}(\cdot \wedge \vartheta) - y_2^{(\varepsilon_k)}(\cdot \wedge \vartheta)\|_\infty \to 0$ as $k \to \infty$ (recall that $\gamma$ is continuous), and, consequently, $y_1^\ast (\cdot \wedge \vartheta) = y_2^\ast (\cdot \wedge \vartheta)$.
        So, using \eqref{proof_lemma_estimate_limsup} and taking into account that $\nu_\varepsilon$ is non-negative and $\varphi_2$ is non-anticipative, we get
        \begin{equation*}
            \varphi_2(\vartheta, y_1^\ast(\cdot)) - \varphi_1(\vartheta, y_1^\ast(\cdot))
            \leq \varphi_2(t, x(\cdot)) - \varphi_1(t, x(\cdot)).
        \end{equation*}
        Since the left-hand side of this inequality is non-negative by \eqref{theorem_comparison_principle_main_vartheta}, we arrive at the desired inequality \eqref{theorem_comparison_principle_main}, which completes the proof.
    \end{proof}

    Note that if $\varphi_1$ is a lower solution of the Cauchy problem \eqref{HJ}, \eqref{boundary_condition} with $H = H_1$ and $\varphi_2$ is an upper solution of the Cauchy problem \eqref{HJ}, \eqref{boundary_condition} with $H = H_2$, then inequality \eqref{theorem_comparison_principle_main_vartheta} with $\vartheta = T$ follows from the boundary conditions \eqref{boundary_condition_lower} for $\varphi_1$ and \eqref{boundary_condition_upper} for $\varphi_2$.
    
    \begin{remark} \label{remark_condition_A.3}
        The proof of Theorem \ref{theorem_comparison_principle} explicitly uses condition (A.3) of Assumption \ref{assumption_H_sigma} only for the compact set $D \coloneq Y(t, x(\cdot); c_{H_1}(\cdot)) \cup Y(t, x(\cdot); c_{H_2}(\cdot))$.
        In this regard, observe that, in particular, for the comparison principle, a certain weakening of condition (A.3) seems possible:
        instead of requiring it for every compact set $D \subset C([- h, T], \mathbb{R}^n)$, it suffices to require it only for the compact sets $D \coloneq Y(t, x(\cdot); c(\cdot))$, where $(t, x(\cdot)) \in [0, T] \times C([- h, T], \mathbb{R}^n)$ and $c(\cdot) \in \mathcal{L}_1([0, T], \mathbb{R})$ is non-negative.
        Note that some related conditions appear in, e.g., \cite{Bandini_Keller_2025,Bayraktar_Gomoyunov_Keller_2025,Bayraktar_Keller_2018}.
    \end{remark}

    The following result, symmetric to Theorem \ref{theorem_comparison_principle} in some sense, is also valid.
    \begin{theorem} \label{theorem_comparison_principle_2}
        Let Hamiltonians $H_1$ and $H_2$ satisfy conditions {\rm(A.1)}--{\rm(A.5)}.
        Let $\varphi_1 \colon [0, T] \times C([- h, T], \mathbb{R}^n) \to \mathbb{R}$ be a non-anticipative lower semicontinuous functional possessing property {\rm(U)} with $H = H_1$, $\varphi_2 \colon [0, T] \times C([- h, T], \mathbb{R}^n) \to \mathbb{R}$ be a non-anticipative upper semicontinuous functional possessing property {\rm(L)} with $H = H_2$, and let a point $(t, x(\cdot)) \in [0, T) \times C([- h, T], \mathbb{R}^n)$ be fixed.
        Suppose that, for every $y(\cdot) \in Y(t, x(\cdot); c_{H_1}(\cdot))$, there is a set $E \coloneq E(y(\cdot)) \subset [t, T]$ with $\mu(E) = T - t$ and such that $H_1(\tau, y(\cdot), s) \geq H_2(\tau, y(\cdot), s)$ for all $\tau \in E$ and $s \in \mathbb{R}^n$.
        Suppose also that there exists $\vartheta \in (t, T]$ such that $\varphi_1(\vartheta, y(\cdot)) \geq \varphi_2(\vartheta, y(\cdot))$ for all $y(\cdot) \in Y(t, x(\cdot); c_{H_1}(\cdot))$.
        Then, the inequality $\varphi_1(t, x(\cdot)) \geq \varphi_2(t, x(\cdot))$ is valid.
    \end{theorem}

    The proof of Theorem \ref{theorem_comparison_principle_2} is completely analogous to that of Theorem \ref{theorem_comparison_principle} and is therefore omitted.

    Recalling Definition \ref{definition_minimax} of a minimax solution of the Cauchy problem \eqref{HJ}, \eqref{boundary_condition}, we derive from Theorems \ref{theorem_comparison_principle} and \ref{theorem_comparison_principle_2} the result below.
    \begin{corollary} \label{corollary_uniqueness_at_point}
        Let Hamiltonians $H_1$ and $H_2$ satisfy conditions {\rm(A.1)}--{\rm(A.5)} and a boundary functional $\sigma$ satisfy condition {\rm (A.6)}.
        Let $\varphi_1$ be a minimax solution of the Cauchy problem \eqref{HJ}, \eqref{boundary_condition} with $H = H_1$, $\varphi_2$ be a minimax solution of the Cauchy problem \eqref{HJ}, \eqref{boundary_condition} with $H = H_2$, and let a point $(t, x(\cdot)) \in [0, T) \times C([- h, T], \mathbb{R}^n)$ be fixed.
        Suppose that, for every $y(\cdot) \in Y(t, x(\cdot); c_{H_1}(\cdot))$, there exists a set $E \coloneq E(y(\cdot)) \subset [t, T]$ with $\mu(E) = T - t$ and such that $H_1(\tau, y(\cdot), s) = H_2(\tau, y(\cdot), s)$ for all $\tau \in E$ and $s \in \mathbb{R}^n$.
        Then, the equality $\varphi_1(t, x(\cdot)) = \varphi_2(t, x(\cdot))$ is valid.
    \end{corollary}

    In particular, we get the following uniqueness theorem for a minimax solution of the Cauchy problem \eqref{HJ}, \eqref{boundary_condition}.
    \begin{theorem} \label{theorem_uniqueness}
        Suppose that a Hamiltonian $H$ and a boundary functional $\sigma$ satisfy Assumption {\rm \ref{assumption_H_sigma}}.
        Then, there exists at most one minimax solution $\varphi$ of the Cauchy problem \eqref{HJ}, \eqref{boundary_condition}.
    \end{theorem}

    Note that properties (U) and (L) from Definitions \ref{definition_U} and \ref{definition_L} involve a function $c_H(\cdot)$, which can be chosen from condition (A.4) of Assumption \ref{assumption_H_sigma} in an ambiguous way.
    Nevertheless, the theorem below shows that, if conditions (A.1)--(A.5) of Assumption \ref{assumption_H_sigma} are satisfied, then properties (U) and (L) do not depend on this choice (in this connection, see also, e.g., \cite[Proposition 3]{Gomoyunov_Lukoyanov_RMS_2024}).

    \begin{theorem} \label{thm:cH_independence}
        Let a Hamiltonian $H$ satisfy conditions {\rm(A.1)}--{\rm(A.5)} of Assumption {\rm \ref{assumption_H_sigma}}, let $c_1(\cdot)$, $c_2(\cdot) \in \mathcal{L}_1([0, T], \mathbb{R})$ be non-negative functions such that condition {\rm(A.4)} is fulfilled both with $c_H(\cdot) = c_1(\cdot)$ and with $c_H(\cdot) = c_2(\cdot)$, and let a functional $\varphi \colon [0, T] \times C([- h, T], \mathbb{R}^n) \to \mathbb{R}$ be non-anticipative.
        Then, if $\varphi$ is lower {\rm(}respectively, upper{\rm)} semicontinuous and possesses property {\rm(U)} {\rm(}respectively, property {\rm(L)}{\rm)} with $c_H(\cdot) = c_1(\cdot)$, then $\varphi$ possesses the same property with $c_H(\cdot) = c_2(\cdot)$.
    \end{theorem}
    \begin{proof}
        For convenience, we split the proof into two steps.

        \smallskip

        {\it Step 1.}
        Fix $r \in \mathbb{R}^n$ and consider the functional
        \begin{equation} \label{eq:psi_r}
            \psi_r(t, x(\cdot))
            \coloneq \langle r, x(t) \rangle - \int_0^t H(\xi, x(\cdot), r) \, \rd \xi,
        \end{equation}
        where $(t, x(\cdot)) \in [0, T] \times C([- h, T], \mathbb{R}^n)$.
        For every $(t, x(\cdot)) \in [0, T] \times C([- h, T], \mathbb{R}^n)$, in accordance with \eqref{H_non-anticipative}, we have $H(\xi, x(\cdot), r) = H(\xi, x(\cdot \wedge \xi), r) = H(\xi, x(\cdot \wedge t), r)$ for a.e. $\xi \in [0, t]$, and, hence, $\psi_r (t, x(\cdot)) = \psi_r(t, x(\cdot \wedge t))$, which means that $\psi_r$ is non-anticipative.

        Let us show that $\psi_r$ is continuous.
        Let $(t, x(\cdot))$, $(t_k, x_k(\cdot)) \in [0, T] \times C([- h, T], \mathbb{R}^n)$, $k \in \mathbb{N}$, be such that $|t_k - t| + \|x_k(\cdot) - x(\cdot)\|_\infty \to 0$ as $k \to \infty$.
        Consider the compact set $D \coloneq \{x_k(\cdot) \colon k \in \mathbb{N}\} \cup \{x(\cdot)\}$ and take the functions $\lambda_H(\cdot) \coloneq \lambda_H(\cdot; D)$ and $m_H(\cdot) \coloneq m_H(\cdot; D)$ from conditions (A.3) and (A.5) respectively.
        Let $k \in \mathbb{N}$ and assume that $t_k \leq t$ for definiteness.
        Then, we derive
        \begin{align*}
            & |\psi_r(t_k, x_k(\cdot)) - \psi_r(t, x(\cdot))| \\
            & \leq \|r\| \|x_k(t_k) - x(t)\|
            + \int_{0}^{t_k} |H(\xi, x_k(\cdot), r) - H(\xi, x(\cdot), r)| \, \rd \xi \\
            & \quad + \int_{t_k}^{t} |H(\xi, x(\cdot), r) - H(\xi, x(\cdot), 0)| \, \rd \xi
            + \int_{t_k}^{t} |H(\xi, x(\cdot), 0)| \, \rd \xi \\
            & \leq \|r\| \|x_k(t_k) - x(t)\|
            + (1 + \|r\|) \|\lambda_H(\cdot)\|_1 \|x_k(\cdot) - x(\cdot)\|_\infty \\
            & \quad + (1 + \|x(\cdot)\|_\infty) \|r\| \int_{t_k}^t c_2(\xi) \, \rd \xi
            + \int_{t_k}^t m_H(\xi) \, \rd \xi.
        \end{align*}
        Therefore, we have $\psi_r(t_k, x_k(\cdot)) \to \psi_r(t, x(\cdot))$ as $k \to \infty$.

        Now, arguing similarly to the proof of Theorem \ref{theorem_semiclassical}, let us verify that $\psi_r$ possesses properties (U) and (L) with $c_H(\cdot) = c_2(\cdot)$.
        Fix $(t, x(\cdot)) \in [0, T) \times C([- h, T], \mathbb{R}^n)$ and $s \in \mathbb{R}^n$.
        Suppose that $s \neq r$.
        Put $E \coloneq [t, T] \cap E^{\rm (A.2)} \cap E^{\rm (A.4)}$, where $E^{\rm (A.2)}$ and $E^{\rm (A.4)}$ are the sets from conditions (A.2) and (A.4) with $c_H(\cdot) = c_2(\cdot)$ respectively.
        Note that $\mu(E) = T - t$.
        Consider the mapping
        \begin{equation*}
            f(\tau, y(\cdot))
            \coloneq \begin{cases}
                \displaystyle
                \frac{H(\tau, y(\cdot \wedge \tau), r)
                - H(\tau, y(\cdot \wedge \tau), s)}{\|r - s\|^2}
                (r - s) & \mbox{if } \tau \in E, \\
                0 & \mbox{otherwise},
            \end{cases}
        \end{equation*}
        where $(\tau, y(\cdot)) \in [t, T] \times C([- h, T], \mathbb{R}^n)$.
        It can be directly verified that the mapping $f$ is non-anticipative;
        for every $y(\cdot) \in C([- h, T], \mathbb{R}^n)$, the function $\tau \mapsto f(\tau, y(\cdot))$, $[t, T] \to \mathbb{R}^n$, is measurable;
        and, for every $\tau \in [t, T]$, the mapping $y(\cdot) \mapsto f(\tau, y(\cdot))$, $C([- h, T], \mathbb{R}^n) \to \mathbb{R}^n$, is continuous, and $\|f(\tau, y(\cdot))\| \leq c_2(\tau) (1 + \|y(\cdot \wedge \tau)\|_\infty)$ for all $y(\cdot) \in C([- h, T], \mathbb{R}^n)$.
        Therefore (see, e.g., \cite[Theorem 2.1]{Obukhovskii_1992}), there exists a function $y(\cdot) \in Y(t, x(\cdot); c_2(\cdot))$ such that $\dot{y}(\tau) = f(\tau, y(\cdot))$ for a.e. $\tau \in [t, T]$.
        By construction, and in view of \eqref{H_non-anticipative}, we have
        \begin{equation} \label{eq:psi_r_char}
            \langle r - s, \dot{y}(\tau) \rangle
            = H(\tau, y(\cdot), r) - H(\tau, y(\cdot), s)
        \end{equation}
        for a.e. $\tau \in [t, T]$.
        In the case where $s = r$, we take arbitrarily $y(\cdot) \in Y(t, x(\cdot); c_2(\cdot))$ and observe that \eqref{eq:psi_r_char} holds for all $\tau \in [t, T]$.
        Thus, according to \eqref{eq:psi_r}, \eqref{eq:psi_r_char} and since $\psi_r$ is non-anticipative, we obtain
        \begin{equation*}
            \psi_r(\tau, y(\cdot))
            - \int_t^\tau \bigl( \langle s, \dot{y}(\xi) \rangle
            - H(\xi, y(\cdot), s) \bigr) \, \rd \xi
            = \psi_r(t, y(\cdot))
            = \psi_r(t, x(\cdot))
        \end{equation*}
        for all $\tau \in [t, T]$.

        \smallskip

        {\it Step 2.}
        Suppose that $\varphi$ is lower semicontinuous and possesses property (U) with $c_H(\cdot) = c_1(\cdot)$.
        Fix $(t, x(\cdot)) \in [0, T) \times C([- h, T], \mathbb{R}^n)$, $s \in \mathbb{R}^n$, $\tau \in (t, T]$ and denote $Y \coloneq Y(t, x(\cdot); c_2(\cdot))$.
        Since $\varphi$ is lower semicontinuous, $\psi_s$ is continuous, and $Y$ is compact, there exists $y^\ast(\cdot) \in Y$ such that
        \begin{equation} \label{choice_M}
            M
            \coloneq \min_{y(\cdot) \in Y}
            \bigl( \varphi(\tau, y(\cdot)) - \psi_s(\tau, y(\cdot)) \bigr)
            = \varphi(\tau, y^\ast(\cdot)) - \psi_s(\tau, y^\ast(\cdot)).
        \end{equation}
        Consider the functional
        \begin{equation*}
            \varphi^\ast(t^\prime, x^\prime(\cdot))
            \coloneq \psi_s(t^\prime, x^\prime(\cdot)) + M
            = \psi_s(t^\prime, x^\prime(\cdot))
            + \varphi(\tau, y^\ast(\cdot)) - \psi_s(\tau, y^\ast(\cdot)),
        \end{equation*}
        where $(t^\prime, x^\prime(\cdot)) \in [0, T] \times C([- h, T], \mathbb{R}^n)$.
        From the properties of the functional $\psi_s$ derived in Step 1, it follows that $\varphi^\ast$ is non-anticipative, continuous, and possesses property (L) with $c_H(\cdot) = c_2(\cdot)$.
        In addition, we have $\varphi^\ast(\tau, y(\cdot)) \leq \varphi(\tau, y(\cdot))$ for all $y(\cdot) \in Y$ thanks to \eqref{choice_M}.
        Hence, by Theorem \ref{theorem_comparison_principle}, we get $\varphi^\ast(t, x(\cdot)) \leq \varphi(t, x(\cdot))$.
        Consequently, recalling the definition of $\psi_s$ (see \eqref{eq:psi_r}) and taking into account that $\psi_s$ is non-anticipative, we obtain
        \begin{align*}
            \varphi(\tau, y^\ast(\cdot))
            & \leq \varphi(t, x(\cdot)) + \psi_s(\tau, y^\ast(\cdot)) - \psi_s(t, y^\ast(\cdot)) \\
            & = \varphi(t, x(\cdot))
            + \int_t^\tau \bigl( \langle s, \dot{y}^\ast(\xi) \rangle
            - H(\xi, y^\ast(\cdot), s) \bigr) \, \rd \xi.
        \end{align*}
        Thus, $\varphi$ possesses property (U) with $c_H(\cdot) = c_2(\cdot)$.

        The case where $\varphi$ is upper semicontinuous and possesses property (L) with $c_H(\cdot) = c_1(\cdot)$ is handled in a similar manner.
        The proof is complete.
    \end{proof}

    Theorem \ref{thm:cH_independence} immediately implies the following result.
    \begin{corollary} \label{cor:cH_independence}
        Suppose that a Hamiltonian $H$ and a boundary functional $\sigma$ satisfy Assumption {\rm\ref{assumption_H_sigma}}.
        Then, the sets of all upper, lower, and minimax solutions of the Cauchy problem \eqref{HJ}, \eqref{boundary_condition} do not depend on the choice of the function $c_H(\cdot)$ from condition {\rm(A.4)}.
    \end{corollary}

\section{Stability}

    Let Hamiltonians $H$, $H_k$, $k \in \mathbb{N}$, and boundary functionals $\sigma$, $\sigma_k$, $k \in \mathbb{N}$, satisfy Assumption \ref{assumption_H_sigma}.
    Suppose that
    \begin{equation} \label{c_H_convergence}
        \lim_{k \to \infty} \|c_{H_k}(\cdot) - c_H(\cdot)\|_1
        = 0,
    \end{equation}
    where $c_H(\cdot)$, $c_{H_k}(\cdot)$, $k \in \mathbb{N}$, are the corresponding functions from condition (A.4) (in this connection, see also Corollary \ref{cor:cH_independence}), and
    \begin{equation} \label{sigma_H_convergence}
        \begin{aligned}
            \lim_{k \to \infty} \max_{y(\cdot) \in D}
            \int_{0}^{T} |H_k(t, y(\cdot), s) - H(t, y(\cdot), s)| \, \rd t
            & = 0, \\
            \lim_{k \to \infty} \max_{y(\cdot) \in D} |\sigma_k(y(\cdot)) - \sigma(y(\cdot))|
            & = 0
        \end{aligned}
    \end{equation}
    for every compact set $D \subset C([- h, T], \mathbb{R}^n)$ and every $s \in \mathbb{R}^n$.
    Suppose also that, for every $k \in \mathbb{N}$, a minimax solution $\varphi_k$ of the Cauchy problem \eqref{HJ}, \eqref{boundary_condition} with $H = H_k$ and $\sigma = \sigma_k$ exists.

    \begin{theorem} \label{theorem_continuous_dependence}
        Let the above assumptions be fulfilled.
        Then, the Cauchy problem \eqref{HJ}, \eqref{boundary_condition} has at least one minimax solution $\varphi$.
        In addition, it holds that
        \begin{equation} \label{theorem_continuous_dependence_main}
            \lim_{k \to \infty} \max_{(t, x(\cdot)) \in [0, T] \times D}
            |\varphi_k(t, x(\cdot)) - \varphi(t, x(\cdot))|
            = 0
        \end{equation}
        for every compact set $D \subset C([- h, T], \mathbb{R}^n)$.
    \end{theorem}
    \begin{proof}
        We follow the scheme of the proof of \cite[Theorem 9.1]{Lukoyanov_2011_Eng} (see also, e.g., \cite[Theorem 6.1]{Bayraktar_Gomoyunov_Keller_2025}).
        We split the proof into five steps for convenience.

        \smallskip

        {\it Step 1.}
            For every point $(t, x(\cdot)) \in [0, T] \times C([- h, T], \mathbb{R}^n)$, define
            \begin{equation} \label{varphi_-}
                \begin{aligned}
                    \varphi_- (t, x(\cdot))
                    & \coloneq \lim_{\delta \to 0^+} \inf
                    \bigl\{ \varphi_k (t^\prime, x^\prime(\cdot)) \colon \\
                    & \quad k \in \mathbb{N},
                    \, k \geq 1 / \delta,
                    \, (t^\prime, x^\prime(\cdot)) \in O_\delta(t, x(\cdot \wedge t)) \bigr\},
                \end{aligned}
            \end{equation}
            where
            \begin{align*}
                O_\delta(t, x(\cdot \wedge t))
                & \coloneq \bigl\{ (t^\prime, x^\prime(\cdot)) \in [0, T] \times C([- h, T], \mathbb{R}^n) \colon \\
                & \quad |t^\prime - t| + \|x^\prime(\cdot) - x(\cdot \wedge t)\|_\infty
                \leq \delta \bigr\}.
            \end{align*}
            Note that the limit in \eqref{varphi_-} exists by monotonicity (but may take infinite values, in general) and $\varphi_-$ is non-anticipative.

            According to \eqref{varphi_-}, for every point $(t, x(\cdot)) \in [0, T] \times C([- h, T], \mathbb{R}^n)$, there exist a subsequence $\varphi_{k_i}$, $i \in \mathbb{N}$, and a sequence $(t_i, x_i(\cdot)) \in [0, T] \times C([- h, T], \mathbb{R}^n)$, $i \in \mathbb{N}$, such that
            \begin{equation} \label{varphi_-_cor}
                \lim_{i \to \infty} \varphi_{k_i}(t_i, x_i(\cdot))
                = \varphi_-(t, x(\cdot)),
                \quad \lim_{i \to \infty} \bigl( |t_i - t| + \|x_i(\cdot) - x(\cdot \wedge t)\|_\infty \bigr)
                = 0.
            \end{equation}

        \smallskip

        {\it Step 2.}
            Let us fix $x(\cdot) \in C([- h, T], \mathbb{R}^n)$ and prove the inequality
            \begin{equation} \label{varphi_-_boundary_condition}
                \varphi_-(T, x(\cdot))
                \geq \sigma(x(\cdot)).
            \end{equation}
            Let $\varphi_{k_i}$, $(t_i, x_i(\cdot)) \in [0, T] \times C([- h, T], \mathbb{R}^n)$, $i \in \mathbb{N}$, be such that
            \begin{equation} \label{step_2_convergence}
                \lim_{i \to \infty} \varphi_{k_i}(t_i, x_i(\cdot))
                = \varphi_-(T, x(\cdot)),
                \quad \lim_{i \to \infty} \bigl( |t_i - T| + \|x_i(\cdot) - x(\cdot)\|_\infty
                \bigr)
                = 0.
            \end{equation}
            For every $i \in \mathbb{N}$, since $\varphi_{k_i}$ is an upper solution of the Cauchy problem \eqref{HJ}, \eqref{boundary_condition} with $H = H_{k_i}$ and $\sigma = \sigma_{k_i}$ (see Definition \ref{definition_U}), there exists $y_i(\cdot) \in Y(t_i, x_i(\cdot); c_{H_{k_i}}(\cdot))$ such that
            \begin{equation} \label{step_2_basic}
                \varphi_{k_i}(t_i, x_i(\cdot))
                \geq \sigma_{k_i}(y_i(\cdot))
                + \int_{t_i}^{T} H_{k_i}(\xi, y_i(\cdot), 0) \, \rd \xi.
            \end{equation}
            In view of \eqref{c_H_convergence}, \eqref{step_2_convergence}, Proposition \ref{proposition_1}, and the equality $Y(T, x(\cdot); c_H(\cdot)) = \{x(\cdot)\}$, by passing to a subsequence if necessary, we can assume that $\|y_i(\cdot) - x(\cdot)\|_\infty \to 0$ as $i \to \infty$.
            Consider the compact set $D_\ast \coloneq \{y_i(\cdot) \colon i \in \mathbb{N}\} \cup \{x(\cdot)\}$ and take the corresponding function $m_H(\cdot) \coloneq m_H(\cdot; D_\ast)$ from condition (A.5).
            Then, owing to the second equality in \eqref{sigma_H_convergence} and continuity of $\sigma$, we have $\sigma_{k_i}(y_i(\cdot)) \to \sigma(x(\cdot))$ as $i \to \infty$.
            Furthermore, for every $i \in \mathbb{N}$, we derive
            \begin{align*}
                & \int_{t_i}^{T} | H_{k_i}(\xi, y_i(\cdot), 0)| \, \rd \xi \\
                & \leq \max_{z(\cdot) \in D_\ast} \int_{0}^{T} | H_{k_i}(\xi, z(\cdot), 0) - H(\xi, z(\cdot), 0)| \, \rd \xi
                + \int_{t_i}^{T} m_H(\xi) \, \rd \xi.
            \end{align*}
            Thus, by passing to the limit in \eqref{step_2_basic} as $i \to \infty$ and using the first equality in \eqref{sigma_H_convergence}, we arrive at the desired inequality \eqref{varphi_-_boundary_condition}.

        \smallskip

        {\it Step 3.}
            Let us verify that $\varphi_-$ has property (U) from Definition \ref{definition_U}.
            Namely, let us fix $(t, x(\cdot)) \in [0, T) \times C([- h, T], \mathbb{R}^n)$, $s \in \mathbb{R}^n$, $\tau \in (t, T]$ and show that there exists $y(\cdot) \in Y(t, x(\cdot); c_H(\cdot))$ for which inequality \eqref{U.1_inequality} with $\varphi = \varphi_-$ holds.

            Consider $\varphi_{k_i}$, $(t_i, x_i(\cdot)) \in [0, T] \times C([- h, T], \mathbb{R}^n)$, $i \in \mathbb{N}$, satisfying \eqref{varphi_-_cor}.
            Since $t_i \to t$ as $i \to \infty$ and $t < \tau$, by passing to a subsequence if necessary, we can assume that $t_i < \tau$ for all $i \in \mathbb{N}$.
            For every $i \in \mathbb{N}$, taking into account that $\varphi_{k_i}$ possesses property (U) with $H = H_{k_i}$, choose $y_i(\cdot) \in Y(t_i, x_i(\cdot); c_{H_{k_i}}(\cdot))$ such that
            \begin{equation} \label{step_3_basic}
                \varphi_{k_i}(\tau, y_i(\cdot))
                - \int_{t_i}^{\tau} \bigl( \langle s, \dot{y}_i(\xi) \rangle
                - H_{k_i}(\xi, y_i(\cdot), s) \bigr) \, \rd \xi
                \leq \varphi_{k_i}(t_i, x_i(\cdot)).
            \end{equation}
            Due to Proposition \ref{proposition_1}, by passing to a subsequence if necessary, we can assume that $\|y_i(\cdot) - y(\cdot)\|_\infty \to 0$ as $i \to \infty$ for some function $y(\cdot) \in Y(t, x(\cdot); c_H(\cdot))$.
            In particular, we have
            \begin{equation} \label{step_3_convergence_1}
                \begin{aligned}
                    \lim_{i \to \infty} \int_{t_i}^{\tau} \langle s, \dot{y}_i(\xi) \rangle \, \rd \xi
                    & = \lim_{i \to \infty} \langle s, y_i(\tau) - y_i(t_i) \rangle \\
                    & = \langle s, y(\tau) - y(t) \rangle \\
                    & = \int_{t}^{\tau} \langle s, \dot{y}(\xi) \rangle \, \rd \xi.
                \end{aligned}
            \end{equation}
            In addition,
            \begin{equation} \label{step_3_convergence_2}
                \liminf_{i \to \infty} \varphi_{k_i}(\tau, y_i(\cdot))
                = \liminf_{i \to \infty} \varphi_{k_i}(\tau, y_i(\cdot \wedge \tau))
                \geq \varphi_-(\tau, y(\cdot))
            \end{equation}
            since $\varphi_{k_i}$, $i \in \mathbb{N}$, are non-anticipative and thanks to the definition of $\varphi_-(\tau, y(\cdot))$.

            Further, consider the compact set $D^\ast \coloneq \{y_i(\cdot) \colon i \in \mathbb{N}\} \cup \{y(\cdot)\}$ and take the corresponding functions $\lambda_H^\ast(\cdot) \coloneq \lambda_H(\cdot; D^\ast)$ and $m_H^\ast(\cdot) \coloneq m_H(\cdot; D^\ast)$ from conditions (A.3) and (A.5) respectively.
            Let $i \in \mathbb{N}$ and assume that $t_i \leq t$ for definiteness.
            Then, we derive
            \begin{align*}
                & \biggl| \int_{t_i}^{\tau} H_{k_i}(\xi, y_i(\cdot), s) \, \rd \xi
                - \int_{t}^{\tau} H(\xi, y(\cdot), s) \, \rd \xi \biggr| \\
                & \leq \int_{t_i}^{\tau}
                | H_{k_i}(\xi, y_i(\cdot), s) - H(\xi, y_i(\cdot), s)| \, \rd \xi \\
                & \quad + \int_{t_i}^{\tau}
                | H(\xi, y_i(\cdot), s) - H(\xi, y(\cdot), s)| \, \rd \xi \\
                & \quad + \int_{t_i}^{t}
                |H(\xi, y(\cdot), s) - H(\xi, y(\cdot), 0) | \, \rd \xi
                + \int_{t_i}^{t}
                |H(\xi, y(\cdot), 0) | \, \rd \xi \\
                & \leq \max_{z(\cdot) \in D^\ast} \int_{0}^{T}
                | H_{k_i}(\xi, z(\cdot), s) - H(\xi, z(\cdot), s)| \, \rd \xi \\
                & \quad + (1 + \|s\|) \|\lambda_H^\ast(\cdot)\|_1 \|y_i(\cdot) - y(\cdot)\|_\infty \\
                & \quad + (1 + \|y(\cdot)\|_\infty) \|s\| \int_{t_i}^{t} c_H(\xi) \, \rd \xi
                + \int_{t_i}^{t} m_H^\ast(\xi) \, \rd \xi.
            \end{align*}
            Therefore,
            \begin{equation} \label{step_3_convergence_3}
                \lim_{i \to \infty}
                \int_{t_i}^{\tau} H_{k_i}(\xi, y_i(\cdot), s) \, \rd \xi
                = \int_{t}^{\tau} H(\xi, y(\cdot), s) \, \rd \xi.
            \end{equation}

            Thus, by passing to the inferior limit in \eqref{step_3_basic} as $i \to \infty$ and taking \eqref{step_3_convergence_1}--\eqref{step_3_convergence_3} into account, we conclude the validity of \eqref{U.1_inequality} with $\varphi = \varphi_-$.

        \smallskip

        {\it Step 4.}
            For any $(t, x(\cdot)) \in [0, T] \times C([- h, T], \mathbb{R}^n)$, we obtain $\varphi_-(t, x(\cdot)) > - \infty$.
            Indeed, if $t = T$, we have $\varphi_-(T, x(\cdot)) \geq \sigma(x(\cdot))$ (see \eqref{varphi_-_boundary_condition}), and if $t \in [0, T)$, owing to property (U) and using \eqref{varphi_-_boundary_condition}, we derive
            \begin{equation*}
                \varphi_-(t, x(\cdot))
                \geq \varphi_-(T, y(\cdot)) + \int_{t}^{T} H(\xi, y(\cdot), 0) \, \rd \xi
                \geq \sigma(y(\cdot)) + \int_{t}^{T} H(\xi, y(\cdot), 0) \, \rd \xi
            \end{equation*}
            for some function $y(\cdot) \in Y(t, x(\cdot); c_H(\cdot))$.

        \smallskip

        {\it Step 5.}
            Now, for every $(t, x(\cdot)) \in [0, T] \times C([- h, T], \mathbb{R}^n)$, define
            \begin{equation} \label{varphi_+}
                \begin{aligned}
                    \varphi_+(t, x(\cdot))
                    & \coloneq \lim_{\delta \to 0^+} \sup \bigl\{ \varphi_k(t^\prime, x^\prime(\cdot))
                    \colon \\
                    & \quad k \in \mathbb{N},
                    \, k \geq 1 / \delta,
                    \, (t^\prime, x^\prime(\cdot)) \in O_\delta(t, x(\cdot \wedge t)) \bigr\}.
                \end{aligned}
            \end{equation}
            Then, repeating the above reasoning with clear changes, we can conclude that $\varphi_+$ is non-anticipative, $\varphi_+(T, x(\cdot)) \leq \sigma(x(\cdot))$ for all $x(\cdot) \in C([- h, T], \mathbb{R}^n)$, $\varphi_+$ satisfies property (L) from Definition \ref{definition_L}, and the inequality $\varphi_+(t, x(\cdot)) < + \infty$ holds for all $(t, x(\cdot)) \in [0, T] \times C([- h, T], \mathbb{R}^n)$.

            According to \eqref{varphi_-} and \eqref{varphi_+}, for every $(t, x(\cdot)) \in [0, T] \times C([- h, T], \mathbb{R}^n)$,
            \begin{equation} \label{step_5}
                - \infty < \varphi_-(t, x(\cdot)) \leq \varphi_+(t, x(\cdot)) < + \infty.
            \end{equation}
            Hence, it follows directly from \eqref{varphi_-} and \eqref{varphi_+} that $\varphi_- \colon [0, T] \times C([- h, T], \mathbb{R}^n) \to \mathbb{R}$ is lower semicontinuous and $\varphi_+ \colon [0, T] \times C([- h, T], \mathbb{R}^n) \to \mathbb{R}$ is upper semicontinuous.
            Consequently, $\varphi_-$ is an upper solution of the Cauchy problem \eqref{HJ}, \eqref{boundary_condition} and $\varphi_+$ is a lower solution of this problem.
            Therefore, we have $\varphi_+(t, x(\cdot)) \leq \varphi_-(t, x(\cdot))$ for all $(t, x(\cdot)) \in [0, T] \times C([- h, T], \mathbb{R}^n)$ by Theorem \ref{theorem_comparison_principle} (see also the note after its proof).
            Thanks to \eqref{step_5}, we obtain that the functionals $\varphi_-$ and $\varphi_+$ are equal, which means that $\varphi \coloneq \varphi_- = \varphi_+$ is a minimax solution of the Cauchy problem  \eqref{HJ}, \eqref{boundary_condition}.

            It remains to observe that the validity of the equality \eqref{theorem_continuous_dependence_main} for every compact set $D \subset C([- h, T], \mathbb{R}^n)$ follows from the equalities $\varphi = \varphi_- = \varphi_+$ and definitions \eqref{varphi_-} and \eqref{varphi_+} of the functionals $\varphi_-$ and $\varphi_+$ (in this connection, see, e.g., \cite[Remark 6.4]{Crandall_Ishii_Lions_1992}).
            The proof is complete.
    \end{proof}

\section{Existence}
\label{section_existence}

    In this section, we apply Theorem \ref{theorem_continuous_dependence} in order to prove the existence of a minimax solution of the Cauchy problem \eqref{HJ}, \eqref{boundary_condition} under Assumption \ref{assumption_H_sigma} on the Hamiltonian $H$ and the boundary functional $\sigma$.

    For every $k \in \mathbb{N}$, define a Hamiltonian $H_k \colon [0, T] \times C([- h, T], \mathbb{R}^n) \times \mathbb{R}^n \to \mathbb{R}$ using the Steklov transformation of $H$ with respect to the first variable $t$: for any $t \in [0, T]$, $x(\cdot) \in C([- h, T], \mathbb{R}^n)$, $s \in \mathbb{R}^n$,
    \begin{equation} \label{H_k}
        H_k(t, x(\cdot), s)
        \coloneq \frac{k}{2} \int_{t - 1 / k}^{t + 1 / k} H(\tau, x(\cdot \wedge t), s) \, \rd \tau.
    \end{equation}
    In the case where $\tau \in [- 1, 0) \cup (T, T + 1]$, we put formally $H(\tau, x(\cdot), s) \coloneq 0$ for all $x(\cdot) \in C([- h, T], \mathbb{R}^n)$, $s \in \mathbb{R}^n$.

    \begin{lemma} \label{lemma_H_k}
        For every $k \in \mathbb{N}$, the Hamiltonian $H_k$ satisfies conditions {\rm (B.1)}--{\rm (B.3)} of Assumption {\rm \ref{assumption_H_sigma_strong}}.
    \end{lemma}
    \begin{proof}
        Let us show that $H_k$ satisfies (B.2).
        Fix a compact set $D \subset C([- h, T], \mathbb{R}^n)$.
        Consider the set $D_\ast \coloneq \{ x(\cdot \wedge t) \colon x(\cdot) \in D, \, t \in [0, T] \}$, which is compact by compactness of $D$, take the corresponding function $\lambda_H(\cdot) \coloneq \lambda_H(\cdot; D_\ast)$ from condition (A.3) of Assumption \ref{assumption_H_sigma}, and define
        \begin{equation*}
            \lambda_{H_k}
            \coloneq \max_{t \in [0, T]} \frac{k}{2} \int_{t - 1 / k}^{t + 1 / k} \lambda_H(\tau) \, \rd \tau,
        \end{equation*}
        where we put formally $\lambda_H (\tau) \coloneq 0$ for all $\tau \in [- 1, 0) \cup (T, T + 1]$.
        Then, for any $t \in [0, T]$, $x_1(\cdot)$, $x_2(\cdot) \in D$, $s \in \mathbb{R}^n$, we obtain
        \begin{align*}
            & | H_k(t, x_1(\cdot), s) - H_k(t, x_2(\cdot), s)| \\
            & \leq \frac{k}{2} \int_{t - 1 / k}^{t + 1 / k}
            | H(\tau, x_1(\cdot \wedge t), s)
            - H(\tau, x_2(\cdot \wedge t), s) | \, \rd \tau \\
            & \leq (1 + \|s\|) \|x_1(\cdot \wedge t) - x_2(\cdot \wedge t)\|_\infty
            \frac{k}{2} \int_{t - 1 / k}^{t + 1 / k} \lambda_H(\tau) \, \rd \tau \\
            & \leq \lambda_{H_k} (1 + \|s\|) \|x_1(\cdot \wedge t) - x_2(\cdot \wedge t)\|_\infty.
        \end{align*}

        Further, let us prove that $H_k$ satisfies (B.3).
        By the function $c_H(\cdot)$ from condition (A.4) of Assumption \ref{assumption_H_sigma}, define
        \begin{equation*}
            c_{H_k}
            \coloneq \max_{t \in [0, T]} \frac{k}{2} \int_{t - 1 / k}^{t + 1 / k} c_H(\tau) \, \rd \tau,
        \end{equation*}
        where we put formally $c_H (\tau) \coloneq 0$ for all $\tau \in [- 1, 0) \cup (T, T + 1]$.
        Then, for any $t \in [0, T]$, $x(\cdot) \in C([- h, T], \mathbb{R}^n)$, $s_1$, $s_2 \in \mathbb{R}^n$, we have
        \begin{align*}
            & | H_k(t, x(\cdot), s_1) - H_k(t, x(\cdot), s_2)| \\
            & \leq \frac{k}{2} \int_{t - 1 / k}^{t + 1 / k}
            | H(\tau, x(\cdot \wedge t), s_1)
            - H(\tau, x(\cdot \wedge t), s_2) | \, \rd \tau \\
            & \leq (1 + \|x(\cdot \wedge t)\|_\infty) \|s_1 - s_2\|
            \frac{k}{2} \int_{t - 1 / k}^{t + 1 / k} c_H(\tau) \, \rd \tau \\
            & \leq c_{H_k} (1 + \|x(\cdot \wedge t)\|_\infty) \|s_1 - s_2\|.
        \end{align*}

        It remains to verify that $H_k$ is continuous, i.e., satisfies (B.1).
        Thanks to (B.2) and (B.3), it suffices to fix $x(\cdot) \in C([- h, T], \mathbb{R}^n)$ and $s \in \mathbb{R}^n$ and prove that the function $t \mapsto H_k(t, x(\cdot), s)$, $[0, T] \to \mathbb{R}$, is continuous.
        Consider the compact set $D^\ast \coloneq \{ x(\cdot \wedge t) \colon t \in [0, T] \}$, take the corresponding function $\lambda_H^\ast(\cdot) \coloneq \lambda_H(\cdot; D^\ast)$ from condition (A.3), and put formally $\lambda_H^\ast (\tau) \coloneq 0$ for all $\tau \in [- 1, 0) \cup (T, T + 1]$.
        For any $t$, $t^\prime \in [0, T]$, we derive
        \begin{align*}
            & | H_k(t^\prime, x(\cdot), s) - H_k(t, x(\cdot), s)| \\*
            & \leq \frac{k}{2} \int_{t^\prime - 1 / k}^{t^\prime + 1 / k}
            |H(\tau, x(\cdot \wedge t^\prime), s) - H(\tau, x(\cdot \wedge t), s)| \, \rd \tau \\
            & \quad + \frac{k}{2} \biggl| \int_{t^\prime - 1 / k}^{t^\prime + 1 / k}
            H(\tau, x(\cdot \wedge t), s) \, \rd \tau
            - \int_{t - 1 / k}^{t + 1 / k} H(\tau, x(\cdot \wedge t), s) \, \rd \tau \biggr| \\
            & \leq (1 + \|s\|) \|x(\cdot \wedge t) - x(\cdot \wedge t^\prime)\|_\infty
            \frac{k}{2} \int_{t^\prime - 1 / k}^{t^\prime + 1 / k} \lambda_H^\ast(\tau) \, \rd \tau \\*
            & \quad + \frac{k}{2} \biggl| \int_{t^\prime - 1 / k}^{t^\prime + 1 / k}
            H(\tau, x(\cdot \wedge t), s) \, \rd \tau
            - \int_{t - 1 / k}^{t + 1 / k} H(\tau, x(\cdot \wedge t), s) \, \rd \tau \biggr|,
        \end{align*}
        which yields $H_k(t^\prime, x(\cdot), s) \to H_k(t, x(\cdot), s)$ as $t^\prime \to t$.
        The proof is complete.
    \end{proof}

    In particular, we obtain that, for every $k \in \mathbb{N}$, the Hamiltonian $H_k$ satisfies conditions (A.1)--(A.5) of Assumption \ref{assumption_H_sigma} and, moreover, condition (A.4) is fulfilled with the function
    \begin{equation*}
        \tilde{c}_{H_k}(t)
        \coloneq \frac{k}{2} \int_{t - 1 / k}^{t + 1 / k} c_H(\tau) \, \rd \tau
    \end{equation*}
    for all $t \in [0, T]$.
    Recall that $c_H(\cdot)$ is the function from condition (A.4) for the Hamiltonian $H$ and $c_H (\tau) \coloneq 0$ for all $\tau \in [- 1, 0) \cup (T, T + 1]$.
    Note that, according to, e.g., \cite[Section XVIII.3, Lemma 4]{Natanson_2_1960}, we have
    \begin{equation} \label{convergence_c_H_k}
        \lim_{k \to \infty}
        \|\tilde{c}_{H_k}(\cdot) - c_H(\cdot)\|_1
        = 0.
    \end{equation}

    \begin{lemma} \label{lemma_H_k_convergence}
        For every compact set $D \subset C([- h, T], \mathbb{R}^n)$ and every $s \in \mathbb{R}^n$, the first relation in \eqref{sigma_H_convergence} is fulfilled for $H_k$, $k \in \mathbb{N}$, given by \eqref{H_k}.
    \end{lemma}
    \begin{proof}
        Let us consider a mapping $\mathcal{H} \colon C([- h, T], \mathbb{R}^n) \to L_1([0, T], \mathbb{R})$ that to every function $y(\cdot) \in C([- h, T], \mathbb{R}^n)$ assigns the function $t \mapsto H(t, y(\cdot), s)$, $[0, T] \to \mathbb{R}$.
        Here, $L_1([0, T], \mathbb{R})$ is the Banach space of all (equivalence classes of) functions from $\mathcal{L}_1([0, T], \mathbb{R})$ with the standard norm $\|\cdot\|_1$.

        Let us verify that $\mathcal{H}$ is continuous.
        Let a function $y(\cdot) \in C([- h, T], \mathbb{R}^n)$ and a sequence $y_i(\cdot) \in C([- h, T], \mathbb{R}^n)$, $i \in \mathbb{N}$, be such that $\|y_i(\cdot) - y(\cdot)\|_\infty \to 0$ as $i \to \infty$.
        Define the compact set $D_\ast \coloneq \{ y_i(\cdot) \colon i \in \mathbb{N}\} \cup \{y(\cdot)\}$ and choose the corresponding function $\lambda_H(\cdot) \coloneq \lambda_H(\cdot; D_\ast)$ according to condition (A.3) of Assumption \ref{assumption_H_sigma}.
        Then, for every $i \in \mathbb{N}$, we have
        \begin{equation*}
            \int_{0}^{T} |H(t, y_i(\cdot), s) - H(t, y(\cdot), s)| \, \rd t
            \leq (1 + \|s\|) \|y_i(\cdot) - y(\cdot)\|_\infty
            \int_{0}^{T} \lambda_H(t) \, \rd t,
        \end{equation*}
        which implies that
        \begin{equation*}
            \lim_{i \to \infty} \int_{0}^{T} |H(t, y_i(\cdot), s) - H(t, y(\cdot), s)| \, \rd t
            = 0.
        \end{equation*}

        Now, fix a compact set $D \subset C([- h, T], \mathbb{R}^n)$ and $s \in \mathbb{R}^n$.
        Owing to continuity of $\mathcal{H}$, the set $\mathcal{H}(D)$ of all functions $t \mapsto H(t, y(\cdot), s)$, $[0, T] \to \mathbb{R}$, where $y(\cdot) \in D$, is compact in the space $L_1([0, T], \mathbb{R})$.
        Therefore, according to the Kolmogorov criterion for compactness in this space (see, e.g., the proof of the necessity part of \cite[Section XVIII.3, Theorem 6]{Natanson_2_1960}), we get
        \begin{equation} \label{lemma_H_k_convergence_proof_1}
            \lim_{k \to \infty} \max_{y(\cdot) \in D}
            \int_{0}^{T} \biggl| \frac{k}{2} \int_{t - 1 / k}^{t + 1 / k} H(\tau, y(\cdot), s) \, \rd \tau
            - H(t, y(\cdot), s) \biggr| \, \rd t
            = 0.
        \end{equation}

        Further, consider the compact set $D^\ast \coloneq \{ y(\cdot \wedge t) \colon y(\cdot) \in D, \, t \in [0, T] \}$, take the corresponding function $\lambda_H^\ast(\cdot) \coloneq \lambda_H(\cdot; D^\ast)$ from condition (A.3), and put formally $\lambda_H^\ast (\tau) \coloneq 0$ for all $\tau \in [- 1, 0) \cup (T, T + 1]$.
        For any $k \in \mathbb{N}$, $y(\cdot) \in D$, we derive
        \begin{align*}
            & \biggl| H_k(t, y(\cdot), s)
            - \frac{k}{2} \int_{t - 1 / k}^{t + 1 / k} H(\tau, y(\cdot), s) \, \rd \tau \biggr| \\
            & \leq \frac{k}{2} \int_{t - 1 / k}^{t + 1 / k}
            |H(\tau, y(\cdot \wedge t), s) - H(\tau, y(\cdot), s)| \, \rd \tau \\
            & \leq (1 + \|s\|) \max_{\tau \in [t, \min\{t + 1 / k, T\}]} \|y(t) - y(\tau)\|
            \frac{k}{2} \int_{t - 1 / k}^{t + 1 / k} \lambda_H^\ast(\tau) \, \rd \tau
        \end{align*}
        for all $t \in [0, T]$ and, therefore, using \cite[Section XVIII.3, Lemma 2]{Natanson_2_1960}, we obtain
        \begin{align*}
            & \int_{0}^{T} \biggl| H_k(t, y(\cdot), s)
            - \frac{k}{2} \int_{t - 1 / k}^{t + 1 / k} H(\tau, y(\cdot), s) \, \rd \tau \biggr| \, \rd t \\
            & \leq (1 + \|s\|) \max_{t \in [0, T]}
            \max_{\tau \in [t, \min\{t + 1 / k, T\}]} \|y(t) - y(\tau)\|
             \int_{0}^{T} \frac{k}{2} \int_{t - 1 / k}^{t + 1 / k} \lambda_H^\ast(\tau) \, \rd \tau \, \rd t \\
            & \leq (1 + \|s\|) \|\lambda_H^\ast(\cdot)\|_1
            \max_{t \in [0, T]} \max_{\tau \in [t, \min\{t + 1 / k, T\}]} \|y(t) - y(\tau)\|.
        \end{align*}
        Consequently, thanks to the Arzel\`{a}--Ascoli theorem,
        \begin{equation} \label{lemma_H_k_convergence_proof_2}
            \lim_{k \to \infty} \max_{y(\cdot) \in D}
            \int_{0}^{T} \biggl| H_k(t, y(\cdot), s)
            - \frac{k}{2} \int_{t - 1 / k}^{t + 1 / k} H(\tau, y(\cdot), s) \, \rd \tau \biggr| \, \rd t
            = 0.
        \end{equation}

        The first relation in \eqref{sigma_H_convergence} follows from \eqref{lemma_H_k_convergence_proof_1} and \eqref{lemma_H_k_convergence_proof_2}.
    \end{proof}

    For every $k \in \mathbb{N}$, by Lemma \ref{lemma_H_k} and \cite[Theorem 1]{Gomoyunov_Lukoyanov_Plaksin_2021}, there exists a (unique) minimax solution $\varphi_k$ of the Cauchy problem \eqref{HJ}, \eqref{boundary_condition} with $H = H_k$.
    According to Corollary \ref{cor:cH_independence}, $\varphi_k$ is a minimax solution of this problem also in the case where condition (A.4) for $H = H_k$ is used with the function $\tilde{c}_{H_k}(\cdot)$.
    Hence, owing to equality \eqref{convergence_c_H_k} and Lemma \ref{lemma_H_k_convergence}, we derive from Theorem \ref{theorem_continuous_dependence} that there exists a minimax solution of the original Cauchy problem \eqref{HJ}, \eqref{boundary_condition}.

    Thus, we get the following result.
    \begin{theorem} \label{theorem_existence}
        Suppose that a Hamiltonian $H$ and a boundary functional $\sigma$ satisfy Assumption {\rm \ref{assumption_H_sigma}}.
        Then, there exists at least one minimax solution of the Cauchy problem \eqref{HJ}, \eqref{boundary_condition}.
    \end{theorem}

\section{Applications to zero-sum differential games for time-delay systems}
\label{section_differential_game}

    We begin by giving a formalization of the differential game \eqref{system}, \eqref{cost_functional}.

    Let an initial point $(t, x(\cdot)) \in [0, T] \times C([- h, T], \mathbb{R}^n)$ be fixed.
    The sets $\mathcal{U}[t, T]$ and $\mathcal{V}[t, T]$ of {\it players' admissible controls} on the time interval $[t, T]$ consist of all measurable functions $u \colon [t, T] \to P$ and $v \colon [t, T] \to Q$ respectively.
    A {\it motion} of system \eqref{system} corresponding to $(t, x(\cdot))$ and $u(\cdot) \in \mathcal{U}[t, T]$, $v(\cdot) \in \mathcal{V}[t, T]$ is defined as a function $y(\cdot) \in AC(t, x(\cdot))$ (see Section \ref{section_minimax_solution}) that satisfies the dynamic equation \eqref{system} for a.e. $\tau \in [t, T]$.
    Thanks to conditions (C.1)--(C.4) (see Assumption \ref{assunption_DG}), such a motion $y(\cdot) \coloneq y(\cdot; t, x(\cdot), u(\cdot), v(\cdot))$ exists, is unique, and belongs to the compact set $Y(t, x(\cdot); c_f(\cdot))$ defined by \eqref{Y} with $c(\cdot) = c_f(\cdot)$ (note that the existence can be derived from, e.g., \cite[Theorem 2.1]{Obukhovskii_1992}, while the uniqueness follows directly from condition (C.3); in this connection, see also, e.g., \cite[Section 2.6]{Hale_Lunel_1993} and \cite[Section II.4, Theorem 3.1]{Bensoussan_Da_Prato_Delfour_Mitter_2006}).
    The corresponding value $J(t, x(\cdot), u(\cdot), v(\cdot))$ of the cost functional \eqref{cost_functional} is well-defined owing to conditions (C.1), (C.2), and (C.5).

    Let us define the lower and upper values of the game.
    For the first player, a {\it non-anticipative strategy} is a mapping $\alpha \colon \mathcal{V}[t, T] \to \mathcal{U}[t, T]$ with the following property: for any $\tau \in [t, T]$ and any controls $v(\cdot)$, $v^\prime(\cdot) \in \mathcal{V}[t, T]$ of the second player, if the equality $v(\xi) = v^\prime(\xi)$ holds for a.e. $\xi \in [t, \tau]$, then the corresponding controls $u(\cdot) \coloneq \alpha[v(\cdot)](\cdot)$ and $u^\prime(\cdot) \coloneq \alpha[v^\prime(\cdot)](\cdot)$ of the first player satisfy the equality $u(\xi) = u^\prime(\xi)$ for a.e. $\xi \in [t, \tau]$.
    Then, the {\it lower value} of the game is given by
    \begin{equation} \label{lower_value}
        \rho^-(t, x(\cdot))
        \coloneq \inf_{\alpha \in \mathcal{A}[t, T]} \sup_{v(\cdot) \in \mathcal{V}[t, T]}
        J \bigl( t, x(\cdot), \alpha[v(\cdot)](\cdot), v(\cdot) \bigr),
    \end{equation}
    where $\mathcal{A}[t, T]$ is the set of all first player's non-an\-tic\-i\-pa\-tive strategies $\alpha$.
    In a similar way, a second player's {\it non-anticipative strategy} is a mapping $\beta \colon \mathcal{U}[t, T] \to \mathcal{V}[t, T]$ such that, for any $\tau \in [t, T]$ and $u(\cdot)$, $u^\prime(\cdot) \in \mathcal{U}[t, T]$, if $u(\xi) = u^\prime(\xi)$ for a.e. $\xi \in [t, \tau]$, then $v(\xi) = v^\prime(\xi)$ for a.e. $\xi \in [t, \tau]$, where $v(\cdot) \coloneq \beta[u(\cdot)](\cdot)$ and $v^\prime(\cdot) \coloneq \beta[u^\prime(\cdot)](\cdot)$.
    So, the {\it upper value} of the game is
    \begin{equation*}
        \rho^+(t, x(\cdot))
        \coloneq \sup_{\beta \in \mathcal{B}[t, T]} \inf_{u(\cdot) \in \mathcal{U}[t, T]}
        J \bigl( t, x(\cdot), u(\cdot), \beta[u(\cdot)](\cdot) \bigr),
    \end{equation*}
    where $\mathcal{B}[t, T]$ denotes the set of all second player's non-anticipative strategies $\beta$.

    In addition, the corresponding functionals $\rho^- \colon [0, T] \times C([- h, T], \mathbb{R}^n) \to \mathbb{R}$ and $\rho^+ \colon [0, T] \times C([- h, T], \mathbb{R}^n) \to \mathbb{R}$ are called the {\it lower} and {\it upper value functionals} respectively.
    If $\rho^-$ and $\rho^+$ coincide, it is said that the game \eqref{system}, \eqref{cost_functional} has the {\it value}
    \begin{equation*}
        \rho(t, x(\cdot))
        \coloneq \rho^-(t, x(\cdot))
        = \rho^+(t, x(\cdot)),
    \end{equation*}
    where $(t, x(\cdot)) \in [0, T] \times C([- h, T], \mathbb{R}^n)$, and $\rho \colon [0, T] \times C([- h, T], \mathbb{R}^n) \to \mathbb{R}$ is called the {\it value functional}.

    \begin{remark} \label{remark_f_chi_null_set}
        Let mappings $(f_1, \chi_1)$ and $(f_2, \chi_2)$ satisfy conditions (C.1)--(C.5) and a functional $\sigma$ satisfy condition (C.6) (see Assumption \ref{assunption_DG}).
        Suppose that, for every $y(\cdot) \in C([- h, T], \mathbb{R}^n)$, there exists a set $E \coloneq E(y(\cdot)) \subset [0, T]$ with $\mu(E) = T$ and such that, for any $\tau \in E$, $u \in P$, $v \in Q$,
        \begin{equation*}
            f_1(\tau, y(\cdot), u, v)
           = f_2(\tau, y(\cdot), u, v),
            \quad \chi_1(\tau, y(\cdot), u, v)
            = \chi_2(\tau, y(\cdot), u, v).
        \end{equation*}
        Consider the lower $\rho^-_1$ (respectively, $\rho^-_2$) and upper $\rho^+_1$ (respectively, $\rho^+_2$) value functionals of the differential game \eqref{system}, \eqref{cost_functional} with $(f, \chi) = (f_1, \chi_1)$ (respectively, with $(f, \chi) = (f_2, \chi_2)$).
        Then, it follows directly from the above definitions that $\rho^-_1 = \rho^-_2$ and $\rho^+_1 = \rho^+_2$.
        In this regard, we note that conditions (C.2) and (C.4) can be somewhat strengthened when dealing with lower or upper value functionals.
        Namely, we can assume that mapping \eqref{C.2_f_chi} is continuous for all $\tau \in [0, T]$ and that inequality \eqref{C.4_f} holds for all $\tau \in [0, T]$, $y(\cdot) \in C([- h, T], \mathbb{R}^n)$, $u \in P$, $v \in Q$.
    \end{remark}

    Consider the lower $H^-$ and upper $H^+$ Hamiltonians that correspond to the differential game \eqref{system}, \eqref{cost_functional} and are defined by \eqref{H_-_H_+_1} and \eqref{H_-_H_+_2}.
    It can be directly verified that $H^-$ and $H^+$ satisfy conditions (A.1)--(A.5) of Assumption \ref{assumption_H_sigma} and condition (A.4) is fulfilled with $c_{H^-}(\cdot) = c_f(\cdot)$ and $c_{H^+}(\cdot) = c_f(\cdot)$.

    The central result of this section is the following.
    \begin{theorem} \label{theorem_DG_lower_upper}
        Let Assumption {\rm \ref{assunption_DG}} be fulfilled in the differential game \eqref{system}, \eqref{cost_functional}.
        Then, the lower value functional $\rho^-$ {\rm(}respectively, upper value functional $\rho^+${\rm)} coincides with the minimax solution $\varphi^-$ {\rm(}respectively, minimax solution $\varphi^+${\rm)} of the Cauchy problem \eqref{HJ}, \eqref{boundary_condition} with $H = H^-$ {\rm(}respectively, with $H = H^+${\rm)}.
    \end{theorem}

    Recall that the statement of Theorem \ref{theorem_DG_lower_upper} takes place under Assumption \ref{assunption_DG_strong} by virtue of \cite[Theorem 7.1]{Bayraktar_Gomoyunov_Keller_2025}, where a more general case of infinite dimensional path-dependent Hamilton--Jacobi equations is considered (in this connection, see also, e.g., \cite{Lukoyanov_2010_IMM_Eng_1,Lukoyanov_2009_IMM}).
    Using this fact, we first establish a weaker version of Theorem \ref{theorem_DG_lower_upper}.

    \begin{lemma} \label{lemma_DG}
        Suppose that conditions {\rm (C.1)} and {\rm (C.6)} of Assumption {\rm \ref{assunption_DG}} are fulfilled, mapping \eqref{C.2_f_chi} is continuous for all $\tau \in [0, T]$, and there exists a number $\Lambda > 0$ such that, for any $\tau \in [0, T]$, $y(\cdot)$, $y_1(\cdot)$, $y_2(\cdot) \in C([- h, T], \mathbb{R}^n)$, $u \in P$, $v \in Q$,
        \begin{equation} \label{lemma_DG_Lip}
            \begin{aligned}
                & \|f(\tau, y_1(\cdot), u, v) - f(\tau, y_2(\cdot), u, v)\|
                + |\chi(\tau, y_1(\cdot), u, v) - \chi(\tau, y_2(\cdot), u, v)| \\
                & \leq \Lambda \|y_1(\cdot \wedge \tau) - y_2(\cdot \wedge \tau)\|_\infty
            \end{aligned}
        \end{equation}
        and
        \begin{equation} \label{lemma_DG_growth}
            \|f(\tau, y(\cdot), u, v)\| + |\chi(\tau, y(\cdot), u, v)|
            \leq \Lambda (1 + \|y(\cdot \wedge \tau)\|_\infty).
        \end{equation}
        Then, the equalities $\rho^- = \varphi^-$ and $\rho^+ = \varphi^+$ are valid.
    \end{lemma}
    \begin{proof}
        Let $(t, x(\cdot)) \in [0, T] \times C([- h, T], \mathbb{R}^n)$ be fixed.
        Below, we prove the equality $\rho^-(t, x(\cdot)) = \varphi^-(t, x(\cdot))$ only, since the proof of the equality $\rho^+(t, x(\cdot)) = \varphi^+(t, x(\cdot))$ is similar.
        We split the proof into three steps for convenience.

        \smallskip

        {\it Step 1.}
            Fix $k \in \mathbb{N}$.
            By the Scorza Dragoni theorem, there exists a closed set $A_k \subset [0, T]$ with $\mu([0, T] \setminus A_k) \leq T / k$ and such that the restriction of $(f, \chi)$ to the set $A_k \times C([- h, T], \mathbb{R}^n) \times P \times Q$ is continuous.
            Put $F_k \coloneq A_k \cup \{0, T\}$ and note that $F_k$ is closed with $\mu([0, T] \setminus F_k) \leq T / k$ and the restriction of $(f, \chi)$ to the set $F_k \times C([- h, T], \mathbb{R}^n) \times P \times Q$ is continuous.
            We can assume that $F_k \neq [0, T]$ since, otherwise, we can replace $F_k$ by $F_k \setminus (0, T / k)$ without any changes.
            Denote $G_k \coloneq [0, T] \setminus F_k \neq \varnothing$.
            Since $F_k$ is closed and $0$, $T \in F_k$, the set $G_k$ is open (in $\mathbb{R}$).
            Hence (see, e.g., \cite[Section II.5, Theorem 3]{Natanson_1_1964}), $G_k$ is a union of a family of pairwise disjoint intervals $(a_k^i, b_k^i)$, where $a_k^i$, $b_k^i \in F_k$, $a_k^i < b_k^i$, $i \in I_k$, and $I_k \subset \mathbb{N}$ is a set of indices.

            Consider a mapping $(f_k, \chi_k) \colon [0, T] \times C([- h, T], \mathbb{R}^n) \times P \times Q \to \mathbb{R}^n \times \mathbb{R}$ defined for any $y(\cdot) \in C([- h, T], \mathbb{R}^n)$, $u \in P$, $v \in Q$ by
            \begin{equation} \label{f_k_chi_k_definition_1}
                f_k(\tau, y(\cdot), u, v)
                \coloneq f(\tau, y(\cdot), u, v),
                \quad \chi_k(\tau, y(\cdot), u, v)
                \coloneq \chi(\tau, y(\cdot), u, v)
            \end{equation}
            if $\tau \in F_k$ and by
            \begin{equation} \label{f_k_chi_k_definition_2}
                \begin{aligned}
                    f_k(\tau, y(\cdot), u, v)
                    & \coloneq \frac{b_k^i - \tau}{b_k^i - a_k^i}
                    f(a_k^i, y(\cdot), u, v)
                    + \frac{\tau - a_k^i}{b_k^i - a_k^i}
                    f(b_k^i, y(\cdot \wedge \tau), u, v), \\
                    \chi_k(\tau, y(\cdot), u, v)
                    & \coloneq \frac{b_k^i - \tau}{b_k^i - a_k^i}
                    \chi(a_k^i, y(\cdot), u, v)
                    + \frac{\tau - a_k^i}{b_k^i - a_k^i}
                    \chi(b_k^i, y(\cdot \wedge \tau), u, v)
                \end{aligned}
            \end{equation}
            if $\tau \in (a_k^i, b_k^i)$ and $i \in I_k$.
            Thanks to \eqref{lemma_DG_Lip} and \eqref{lemma_DG_growth}, we have
            \begin{equation} \label{lemma_DG_Lip_k}
                \begin{aligned}
                    & \|f_k(\tau, y_1(\cdot), u, v) - f_k(\tau, y_2(\cdot), u, v)\|
                    + |\chi_k(\tau, y_1(\cdot), u, v) - \chi_k(\tau, y_2(\cdot), u, v)| \\
                    & \leq \Lambda \|y_1(\cdot \wedge \tau) - y_2(\cdot \wedge \tau)\|_\infty
                \end{aligned}
            \end{equation}
            and
            \begin{equation} \label{lemma_DG_growth_k}
                \|f_k(\tau, y(\cdot), u, v)\| + |\chi_k(\tau, y(\cdot), u, v)|
                \leq \Lambda (1 + \|y(\cdot \wedge \tau)\|_\infty)
            \end{equation}
            for all $\tau \in [0, T]$, $y(\cdot)$, $y_1(\cdot)$, $y_2(\cdot) \in C([- h, T], \mathbb{R}^n)$,  $u \in P$, $v \in Q$.

            Now, let us verify that $(f_k, \chi_k)$ is continuous.
            We deal with the mapping $f_k$ only, since the proof for $\chi_k$ is similar.
            Note that, in view of \eqref{lemma_DG_Lip_k}, it suffices to fix $y(\cdot) \in C([- h, T], \mathbb{R}^n)$, take $(\tau, u, v)$, $(\tau_j, u_j, v_j) \in [0, T] \times P \times Q$, $j \in \mathbb{N}$, such that
            \begin{equation*}
                \lim_{j \to \infty} \bigl( |\tau_j - \tau|
                + \|u_j - u\|
                + \|v_j - v\| \bigr)
                = 0,
            \end{equation*}
            and prove that $f_k(\tau_j, y(\cdot), u_j, v_j) \to f_k(\tau, y(\cdot), u, v)$ as $j \to \infty$.
            Directly from the definition of $f_k$ and continuity of the restriction of $f$ to $F_k \times C([- h, T], \mathbb{R}^n) \times P \times Q$, it follows that only the case where $\tau \in F_k$ and $\tau_j \in G_k$ for all $j \in \mathbb{N}$ requires special analysis.
            Moreover, we can limit ourselves to considering only two situations: either $\tau_j > \tau$ for all $j \in \mathbb{N}$ or $\tau_j < \tau$ for all $j \in \mathbb{N}$.

            Suppose that the first situation takes place.
            For every $j \in \mathbb{N}$, consider $i(j) \in I_k$ such that $\tau_j \in (a_k^{i(j)}, b_k^{i(j)})$.
            Observe that $\tau_j > a_k^{i(j)} \geq \tau$ for all $j \in \mathbb{N}$ and, therefore, $a_k^{i(j)} \to \tau$ as $j \to \infty$.
            Fix $\varepsilon > 0$.
            Denote $D_\ast \coloneq \{y(\cdot \wedge \bar{\tau}) \colon \bar{\tau} \in [0, T]\}$.
            The restriction of $f$ to the compact set $F_k \times D_\ast \times P \times Q$ is uniformly continuous and bounded.
            Hence, first, there exists $\delta > 0$ such that, for every $j \in \mathbb{N}$, if $b_k^{i(j)} - a_k^{i(j)} \leq \delta$, then
            \begin{equation*}
                \|f(b_k^{i(j)}, y(\cdot \wedge \tau_j), u_j, v_j)
                - f(a_k^{i(j)}, y(\cdot \wedge \tau_j), u_j, v_j)\|
                \leq \varepsilon / 2,
            \end{equation*}
            and, second, there exists $j_\ast \in \mathbb{N}$ such that, for every $j \geq j_\ast$,
            \begin{equation*}
                \|f(a_k^{i(j)}, y(\cdot), u_j, v_j) - f(\tau, y(\cdot), u, v)\|
                \leq \varepsilon / 2
            \end{equation*}
            and
            \begin{equation*}
                \frac{\tau_j - a_k^{i(j)}}{\delta}
                \bigl( \|f(b_k^{i(j)}, y(\cdot \wedge \tau_j), u_j, v_j)\|
                + \|f(a_k^{i(j)}, y(\cdot \wedge \tau_j), u_j, v_j)\| \bigr)
                \leq \frac{\varepsilon}{2}.
            \end{equation*}
            Let $j \geq j_\ast$.
            In accordance with \eqref{f_k_chi_k_definition_1} and \eqref{f_k_chi_k_definition_2}, noting that
            \begin{equation*}
                f(a_k^{i(j)}, y(\cdot), u_j, v_j)
                = f(a_k^{i(j)}, y(\cdot \wedge \tau_j), u_j, v_j)
            \end{equation*}
            thanks to \eqref{lemma_DG_Lip} and since $\tau_j > a_k^{i(j)}$, we derive
            \begin{align*}
                & \|f_k(\tau_j, y(\cdot), u_j, v_j)
                - f_k(\tau, y(\cdot), u, v)\| \\
                & \leq \frac{\tau_j - a_k^{i(j)}}{b_k^{i(j)} - a_k^{i(j)}}
                \|f(b_k^{i(j)}, y(\cdot \wedge \tau_j), u_j, v_j)
                - f(a_k^{i(j)}, y(\cdot), u_j, v_j)\| \\
                & \quad + \|f(a_k^{i(j)}, y(\cdot), u_j, v_j) - f(\tau, y(\cdot), u, v)\| \\
                & \leq \frac{\tau_j - a_k^{i(j)}}{b_k^{i(j)} - a_k^{i(j)}}
                \|f(b_k^{i(j)}, y(\cdot \wedge \tau_j), u_j, v_j)
                - f(a_k^{i(j)}, y(\cdot \wedge \tau_j), u_j, v_j)\|
                + \frac{\varepsilon}{2}.
            \end{align*}
            Further, in the case where $b_k^{i(j)} - a_k^{i(j)} \leq \delta$, we obtain
            \begin{align*}
                & \frac{\tau_j - a_k^{i(j)}}{b_k^{i(j)} - a_k^{i(j)}}
                \|f(b_k^{i(j)}, y(\cdot \wedge \tau_j), u_j, v_j)
                - f(a_k^{i(j)}, y(\cdot \wedge \tau_j), u_j, v_j)\| \\
                & \leq \|f(b_k^{i(j)}, y(\cdot \wedge \tau_j), u_j, v_j)
                - f(a_k^{i(j)}, y(\cdot \wedge \tau_j), u_j, v_j)\| \\
                & \leq \varepsilon / 2.
            \end{align*}
            Otherwise, we have $b_k^{i(j)} - a_k^{i(j)} > \delta$, and, therefore,
            \begin{align*}
                & \frac{\tau_j - a_k^{i(j)}}{b_k^{i(j)} - a_k^{i(j)}}
                \|f(b_k^{i(j)}, y(\cdot \wedge \tau_j), u_j, v_j)
                - f(a_k^{i(j)}, y(\cdot \wedge \tau_j), u_j, v_j)\| \\
                & \leq \frac{\tau_j - a_k^{i(j)}}{\delta}
                \bigl( \|f(b_k^{i(j)}, y(\cdot \wedge \tau_j), u_j, v_j)\|
                + \|f(a_k^{i(j)}, y(\cdot \wedge \tau_j), u_j, v_j)\| \bigr) \\
                & \leq \varepsilon / 2.
            \end{align*}
            Consequently, we arrive at the inequality $\|f_k(\tau_j, y(\cdot), u_j, v_j) - f_k(\tau, y(\cdot), u, v)\| \leq \varepsilon$ and complete the proof of the desired convergence.

            The situation where $\tau_j < \tau$ for all $j \in \mathbb{N}$ is handled in a similar manner.

            In particular, we conclude that $(f_k, \chi_k)$ satisfies conditions (D.1)--(D.3) of Assumption \ref{assunption_DG_strong}.

        \smallskip

        {\it Step 2.}
            For every $k \in \mathbb{N}$, consider the differential game \eqref{system}, \eqref{cost_functional} with $(f, \chi) = (f_k, \chi_k)$ and the lower value of this game $\rho_k^-(t, x(\cdot))$ (see \eqref{lower_value}).
            Let us prove that
            \begin{equation} \label{rho_k^-_to_rho^-}
                \lim_{k \to \infty} \rho_k^-(t, x(\cdot))
                = \rho^-(t, x(\cdot)).
            \end{equation}

            Let $\varepsilon > 0$ be fixed.
            Define the compact set $Y(t, x(\cdot); c_\Lambda(\cdot))$ according to \eqref{Y} with $c(\cdot) = c_\Lambda(\cdot)$, where $c_\Lambda(\tau) \coloneq \Lambda$ for all $\tau \in [0, T]$, and choose $R > 0$ such that $\|y(\cdot)\|_\infty \leq R$ for all $y(\cdot) \in Y(t, x(\cdot); c_\Lambda(\cdot))$.
            Since the functional $\sigma$ is continuous by condition (C.6), there exists $\zeta > 0$ such that $|\sigma(y_1(\cdot)) - \sigma(y_2(\cdot))| \leq \varepsilon / 2$ for all $y_1(\cdot)$, $y_2(\cdot) \in Y(t, x(\cdot); c_\Lambda(\cdot))$ satisfying the inequality $\|y_1(\cdot) - y_2(\cdot)\|_\infty \leq \zeta$.
            Choose $k_\ast \in \mathbb{N}$ from the conditions
            \begin{equation*}
                2 \Lambda (1 + R) T e^{\Lambda T} / k_\ast
                \leq \zeta,
                \quad 2 \Lambda (1 + R) T (\Lambda T e^{\Lambda T} + 1) / k_\ast
                \leq \varepsilon / 2.
            \end{equation*}
            Let $k \geq k_\ast$, $u(\cdot) \in \mathcal{U}[t, T]$, $v(\cdot) \in \mathcal{V}[t, T]$, and let $y(\cdot) \coloneq y(\cdot; t, x(\cdot), u(\cdot), v(\cdot))$ be the motion of system \eqref{system} and $y_k(\cdot) \coloneq y_k(\cdot; t, x(\cdot), u(\cdot), v(\cdot))$ be the motion of system \eqref{system} with $f = f_k$.
            Note that $y(\cdot)$, $y_k(\cdot) \in Y(t, x(\cdot); c_\Lambda(\cdot))$.
            Due to \eqref{f_k_chi_k_definition_1}, we have
            \begin{align*}
                & \|y(\tau) - y_k(\tau)\| \\
                & \leq \int_{t}^{\tau} \| f(\xi, y(\cdot), u(\xi), v(\xi))
                - f_k(\xi, y_k(\cdot), u(\xi), v(\xi))\| \, \rd \xi \\
                & \leq \int_{t}^{\tau} \| f(\xi, y(\cdot), u(\xi), v(\xi))
                - f(\xi, y_k(\cdot), u(\xi), v(\xi))\| \, \rd \xi \\
                & \quad + \int_{[t, \tau] \setminus F_k} \| f(\xi, y_k(\cdot), u(\xi), v(\xi))\| \, \rd \xi
                + \int_{[t, \tau] \setminus F_k} \| f_k(\xi, y_k(\cdot), u(\xi), v(\xi))\| \, \rd \xi
            \end{align*}
            for all $\tau \in [t, T]$.
            Hence, using \eqref{lemma_DG_Lip}, \eqref{lemma_DG_growth}, and \eqref{lemma_DG_growth_k} and taking the inequality $\mu([0, T] \setminus F_k) \leq T / k$ into account, we obtain
            \begin{equation*}
                \|y(\tau) - y_k(\tau)\|
                \leq \Lambda \int_{t}^{\tau}
                \|y(\cdot \wedge \xi) - y_k(\cdot \wedge \xi)\|_\infty \, \rd \xi
                + 2 \Lambda (1 + R) T / k
            \end{equation*}
            for all $\tau \in [t, T]$, which yields
            \begin{equation*}
                \|y(\cdot \wedge \tau) - y_k(\cdot \wedge \tau)\|_\infty
                \leq \Lambda \int_{t}^{\tau}
                \|y(\cdot \wedge \xi) - y_k(\cdot \wedge \xi)\|_\infty \, \rd \xi
                + 2 \Lambda (1 + R) T / k
            \end{equation*}
            for all $\tau \in [t, T]$.
            Then, applying the Gronwall inequality, we derive
            \begin{equation*}
                \|y(\cdot \wedge \tau) - y_k(\cdot \wedge \tau)\|_\infty
                \leq 2 \Lambda (1 + R) T e^{\Lambda T} / k
            \end{equation*}
            for all $\tau \in [t, T]$.
            In particular, $\|y(\cdot) - y_k(\cdot)\|_\infty \leq \zeta$, which implies the inequality $|\sigma(y(\cdot)) - \sigma(y_k(\cdot))| \leq \varepsilon / 2$.
            In addition, arguing similarly to the above, we get
            \begin{align*}
                & \biggl| \int_{t}^{T} \chi(\tau, y(\cdot), u(\tau), v(\tau)) \, \rd \tau
                - \int_{t}^{T} \chi_k(\tau, y_k(\cdot), u(\tau), v(\tau)) \, \rd \tau \biggr| \\
                & \leq \int_{t}^{T} | \chi(\tau, y(\cdot), u(\tau), v(\tau))
                - \chi(\tau, y_k(\cdot), u(\tau), v(\tau)) | \, \rd \tau \\
                & \quad + \int_{[t, T] \setminus F_k} | \chi(\tau, y_k(\cdot), u(\tau), v(\tau)) | \, \rd \tau
                + \int_{[t, T] \setminus F_k} | \chi_k(\tau, y_k(\cdot), u(\tau), v(\tau)) | \, \rd \tau \\
                & \leq \Lambda \int_{t}^{T} \|y(\cdot \wedge \tau) - y_k(\cdot \wedge \tau) \|_\infty \, \rd \tau + 2 \Lambda (1 + R) T / k
                \\
                & \leq 2 \Lambda^2 (1 + R) T^2 e^{\Lambda T} / k + 2 \Lambda (1 + R) T / k \\
                & \leq \varepsilon / 2.
            \end{align*}
            Thus, we come to the inequality
            \begin{equation} \label{J_k}
                |J(t, x(\cdot), u(\cdot), v(\cdot))
                - J_k(t, x(\cdot), u(\cdot), v(\cdot))|
                \leq \varepsilon,
            \end{equation}
            where $J_k(t, x(\cdot), u(\cdot), v(\cdot))$ is the value of the cost functional \eqref{cost_functional} with $(f, \chi) = (f_k, \chi_k)$.
            Since inequality \eqref{J_k} holds for all $u(\cdot) \in \mathcal{U}[t, T]$ and $v(\cdot) \in \mathcal{V}[t, T]$,
            \begin{equation} \label{rho_k^--rhi^-}
                |\rho^-(t, x(\cdot)) - \rho_k^-(t, x(\cdot))|
                \leq \varepsilon,
            \end{equation}
            which completes the proof of equality \eqref{rho_k^-_to_rho^-}.

        \smallskip

        {\it Step 3.}
            Let $k \in \mathbb{N}$.
            Consider the Hamiltonian
            \begin{equation} \label{H_k^-}
                H_k^-(\tau, y(\cdot), s)
                \coloneq \max_{v \in Q} \min_{u \in P}
                \bigl( \langle s, f_k(\tau, y(\cdot), u, v) \rangle
                - \chi_k(\tau, y(\cdot), u, v) \bigr),
            \end{equation}
            where $\tau \in [0, T]$, $y(\cdot) \in C([- h, T], \mathbb{R}^n)$, and $s \in \mathbb{R}^n$.
            Note that $H_k^-$ satisfies conditions (B.1)--(B.3) of Assumption \ref{assumption_H_sigma_strong} and condition (B.3) is fulfilled with $c_{H_k^-} = \Lambda$.
            Let $\varphi_k^-$ be the minimax solution of the Cauchy problem \eqref{HJ}, \eqref{boundary_condition} with $H = H_k^-$.
            Due to \cite[Theorem 7.1]{Bayraktar_Gomoyunov_Keller_2025}, we have
            \begin{equation} \label{rho_k^-=varphi_k^-}
                \rho_k^-(t, x(\cdot))
                = \varphi_k^-(t, x(\cdot)).
            \end{equation}

            Let a compact set $D \subset C([- h, T], \mathbb{R}^n)$ and $s \in \mathbb{R}^n$ be fixed.
            For every $y(\cdot) \in D$, noting that $H_k^-(\tau, y(\cdot), s) = H^-(\tau, y(\cdot), s)$ for all $\tau \in F_k$ thanks to \eqref{f_k_chi_k_definition_1}, taking into account that (see \eqref{H_-_H_+_1} and \eqref{lemma_DG_growth})
            \begin{align*}
                |H^-(\tau, y(\cdot), s)|
                & \leq |H^-(\tau, y(\cdot), s) - H^-(\tau, y(\cdot), 0)|
                + |H^-(\tau, y(\cdot), 0)| \\
                & \leq \Lambda (1 + \|s\|) (1 + \max_{z(\cdot) \in D} \|z(\cdot)\|_\infty)
            \end{align*}
            for all $\tau \in [0, T]$ and, similarly (see \eqref{lemma_DG_growth_k} and \eqref{H_k^-}),
            \begin{equation*}
                |H_k^-(\tau, y(\cdot), s)|
                \leq \Lambda (1 + \|s\|) (1 + \max_{z(\cdot) \in D} \|z(\cdot)\|_\infty)
            \end{equation*}
            for all $\tau \in [0, T]$, and recalling that $\mu([0, T] \setminus F_k) \leq T / k$, we derive
            \begin{align*}
                & \int_{0}^{T} |H_k^-(\tau, y(\cdot), s) - H^-(\tau, y(\cdot), s)| \, \rd \tau \\
                & \leq \int_{[0, T] \setminus F_k} |H_k^-(\tau, y(\cdot), s)| \, \rd \tau
                + \int_{[0, T] \setminus F_k} |H^-(\tau, y(\cdot), s)| \, \rd \tau \\
                & \leq 2 \Lambda (1 + \|s\|) (1 + \max_{z(\cdot) \in D} \|z(\cdot)\|_\infty) T / k.
            \end{align*}
            Consequently, the first equality in \eqref{sigma_H_convergence} with $H = H^-$ and $H_k = H_k^-$ is valid.

            Thus, $\varphi_k^-(t, x(\cdot)) \to \varphi^-(t, x(\cdot))$ as $k \to \infty$ by Theorem \ref{theorem_continuous_dependence}.
            Therefore, owing to \eqref{rho_k^-_to_rho^-} and \eqref{rho_k^-=varphi_k^-}, we get $\rho^-(t, x(\cdot)) = \varphi^-(t, x(\cdot))$ and complete the proof.
    \end{proof}

    \begin{remark}
        Even if the mapping $(f, \chi)$ satisfies Assumption \ref{assumption_saddle_point}, the mappings $(f_k, \chi_k)$, $k \in \mathbb{N}$, defined according to \eqref{f_k_chi_k_definition_1} and \eqref{f_k_chi_k_definition_2} may not satisfy this assumption.
        In particular, the corresponding approximating differential games may not have a value, and if we follow the same reasoning as in the proof of Lemma \ref{lemma_DG}, we still need to deal with the lower and upper value functionals separately.
    \end{remark}

    \begin{remark}
        As in the proof of Theorem \ref{theorem_existence}, we can define approximating mappings $(f_k, \chi_k)$, $k \in \mathbb{N}$, in the proof of Lemma \ref{lemma_DG} using the Steklov transformation of $(f, \chi)$ with respect to the first variable $\tau$ (for fixed $y(\cdot) \in C([- h, T], \mathbb{R}^n)$, $u \in P$, and $v \in Q$).
        However, in this case, difficulties arise in justifying equality \eqref{rho_k^-_to_rho^-}, for example, when obtaining estimates for the values like
        \begin{align*}
            & \int_{t}^{\tau} \| f_k(\xi, y(\cdot), u(\xi), v(\xi))
            - f(\xi, y(\cdot), u(\xi), v(\xi)) \| \, \rd \xi \\
            & =\int_{t}^{\tau} \biggl\| \frac{k}{2} \int_{\xi - 1 / k}^{\xi + 1 / k}
            f(\eta, y(\cdot), u(\xi), v(\xi)) \, \rd \eta
            - f(\xi, y(\cdot), u(\xi), v(\xi)) \biggr\| \, \rd \xi.
        \end{align*}
        The problem is that different times $\eta$ and $\xi$ are involved in $f(\eta, y(\cdot), u(\xi), v(\xi))$.
    \end{remark}

    Now, relying on Lemma \ref{lemma_DG}, we prove Theorem \ref{theorem_DG_lower_upper}.
    \begin{proof}[Proof of Theorem \ref{theorem_DG_lower_upper}.]
        Let us fix $(t, x(\cdot)) \in [0, T] \times C([- h, T], \mathbb{R}^n)$ and prove only that $\rho^-(t, x(\cdot)) = \varphi^-(t, x(\cdot))$ (the proof of the equality $\rho^+(t, x(\cdot)) = \varphi^+(t, x(\cdot))$ is similar).
        For convenience, we split the proof into three steps, which are in natural agreement with the corresponding steps of the proof of Lemma \ref{lemma_DG}.

        \smallskip

        {\it Step 1.}
            Take the function $c_f(\cdot)$ from condition (C.4) (see Assumption \ref{assunption_DG}) and define the set $Y \coloneq Y(t, x(\cdot); c_f(\cdot))$ according to \eqref{Y} with $c(\cdot) = c_f(\cdot)$.
            Recall that $Y$ is compact by Proposition \ref{proposition_1} and consider the functions $\lambda_{f, \chi}(\cdot) \coloneq \lambda_{f, \chi}(\cdot; Y)$ and $m_\chi(\cdot) \coloneq m_\chi(\cdot; Y)$ from conditions (C.3) and (C.5) respectively.
            Taking Remarks \ref{remark_H_null_set} and \ref{remark_f_chi_null_set} into account, we can assume in the proof that mapping \eqref{C.2_f_chi} is continuous for all $\tau \in [0, T]$ and that inequalities \eqref{C.3_f_chi}--\eqref{C.5_chi} hold for all $\tau \in [0, T]$, $y(\cdot)$, $y_1(\cdot)$, $y_2(\cdot) \in Y$, $u \in P$, $v \in Q$.
            For every $\tau \in [0, T]$, denote
            \begin{align*}
                \lambda_\ast(\tau)
                & \coloneq (1 + \sqrt{n}) \lambda_{f, \chi}(\tau), \\
                c_\ast(\tau)
                & \coloneq (c_f(\tau) + \lambda_\ast(\tau))
                (1 + \|x(\cdot)\|_\infty), \\
                m_\ast(\tau)
                & \coloneq m_\chi(\tau)
                + \lambda_\ast(\tau) (1 + \|x(\cdot)\|_\infty).
            \end{align*}

            Fix $k \in \mathbb{N}$.
            Applying the Scorza Dragoni theorem to the mapping $(f, \chi)$ and the Lusin theorem to the functions $\lambda_\ast(\cdot)$, $c_\ast(\cdot)$, $m_\ast(\cdot)$, we choose a closed set $F_k \subset [0, T]$ with $\mu([0, T] \setminus F_k) \leq 1 / k$ and such that the restriction of $(f, \chi)$ to $F_k \times Y \times P \times Q$ and the restrictions of $\lambda_\ast(\cdot)$, $c_\ast(\cdot)$, $m_\ast(\cdot)$ to $F_k$ are continuous.

            Then, in accordance with the McShane--Whitney extension theorem, we define a mapping $(\bar{f}_k, \bar{\chi}_k) \colon F_k \times C([- h, T], \mathbb{R}^n) \times P \times Q \to \mathbb{R}^n \times \mathbb{R}$ by
            \begin{equation} \label{bar_f_k_chi_k}
                \begin{aligned}
                    \bar{f}_k^i(\tau, y(\cdot), u, v)
                    & \coloneq \max_{z(\cdot) \in Y} \bigl( f^i(\tau, z(\cdot), u, v)
                    - \lambda_{f, \chi}(\tau) \|z(\cdot \wedge \tau) - y(\cdot \wedge \tau)\|_\infty \bigr), \\
                    \bar{\chi}_k(\tau, y(\cdot), u, v)
                    & \coloneq \max_{z(\cdot) \in Y} \bigl( \chi(\tau, z(\cdot), u, v)
                    - \lambda_{f, \chi}(\tau) \|z(\cdot \wedge \tau) - y(\cdot \wedge \tau)\|_\infty \bigr)
                \end{aligned}
            \end{equation}
            for all $\tau \in F_k$, $y(\cdot) \in C([- h, T], \mathbb{R}^n)$, $u \in P$, and $v \in Q$.
            Here, $\bar{f}_k^i(\tau, y(\cdot), u, v)$ (respectively, $f^i(\tau, z(\cdot), u, v)$) is the $i$-th coordinate of the vector $\bar{f}_k(\tau, y(\cdot), u, v)$ (respectively, $f(\tau, z(\cdot), u, v)$) and $i \in \overline{1, n}$.
            The restrictions of $(\bar{f}_k, \bar{\chi}_k)$ and $(f, \chi)$ to $F_k \times Y \times P \times Q$ coincide and
            \begin{equation} \label{C.3_f_chi_k}
                \begin{aligned}
                    & \|\bar{f}_k(\tau, y_1(\cdot), u, v) - \bar{f}_k(\tau, y_2(\cdot), u, v)\|
                    + |\bar{\chi}_k(\tau, y_1(\cdot), u, v) - \bar{\chi}_k(\tau, y_2(\cdot), u, v)| \\
                    & \leq \lambda_\ast(\tau)
                    \|y_1(\cdot \wedge \tau) - y_2(\cdot \wedge \tau)\|_\infty
                \end{aligned}
            \end{equation}
            for all $\tau \in F_k$, $y_1(\cdot)$, $y_2(\cdot) \in C([- h, T], \mathbb{R}^n)$, $u \in P$, $v \in Q$.
            In addition, the mapping $(\bar{f}_k, \bar{\chi}_k)$ is continuous and, for any $\tau \in F_k$, $y(\cdot) \in C([- h, T], \mathbb{R}^n)$, $u \in P$, $v \in Q$, thanks to the inclusion $x(\cdot \wedge t) \in Y$, we derive
            \begin{equation} \label{C.4_f_k}
                \begin{aligned}
                    & \|\bar{f}_k(\tau, y(\cdot), u, v)\| \\
                    & \leq \|\bar{f}_k(\tau, y(\cdot), u, v) - \bar{f}_k(\tau, x(\cdot \wedge t), u, v) \|
                    + \|f(\tau, x(\cdot \wedge t), u, v)\| \\
                    & \leq \lambda_\ast(\tau) (\|y(\cdot \wedge \tau)\|_\infty + \|x(\cdot)\|_\infty)
                    + c_f(\tau) (1 + \|x(\cdot)\|_\infty) \\
                    & \leq c_\ast(\tau) (1 + \|y(\cdot \wedge \tau)\|_\infty)
                \end{aligned}
            \end{equation}
            and, similarly,
            \begin{equation} \label{C.5_chi_k}
                |\bar{\chi}_k(\tau, y(\cdot), u, v)|
                \leq m_\ast(\tau) (1 + \|y(\cdot \wedge \tau)\|_\infty).
            \end{equation}

            Consider a mapping $(f_k, \chi_k) \colon [0, T] \times C([- h, T], \mathbb{R}^n) \times P \times Q \to \mathbb{R}^n \times \mathbb{R}$ defined for all $y(\cdot) \in C([- h, T], \mathbb{R}^n)$, $u \in P$, $v \in Q$ by
            \begin{equation} \label{f_k_bar_f_k}
                f_k(\tau, y(\cdot), u, v)
                \coloneq \bar{f}_k(\tau, y(\cdot), u, v), \quad
                \chi_k(\tau, y(\cdot), u, v)
                \coloneq \bar{\chi}_k(\tau, y(\cdot), u, v)
            \end{equation}
            if $\tau \in F_k$ and
            \begin{equation} \label{f_k_bar_f_k_2}
                f_k(\tau, y(\cdot), u, v)
                \coloneq 0, \quad
                \chi_k(\tau, y(\cdot), u, v)
                \coloneq 0
            \end{equation}
            otherwise.
            Owing to continuity of the mapping $(\bar{f}_k, \bar{\chi}_k)$, the mapping $(f_k, \chi_k)$ satisfies condition (C.1) and the mapping
            \begin{equation*}
                \begin{aligned}
                    (y(\cdot), u, v)
                    &  \mapsto (f_k(\tau, y(\cdot), u, v), \chi_k(\tau, y(\cdot), u, v)), \\
                    C([- h, T], \mathbb{R}^n) \times P \times Q
                    & \to \mathbb{R}^n \times \mathbb{R},
                \end{aligned}
            \end{equation*}
            is continuous for all $\tau \in [0, T]$.
            Thanks to \eqref{C.3_f_chi_k}--\eqref{C.5_chi_k}, we have
            \begin{equation} \label{C.3_f_chi_k^ast}
                \begin{aligned}
                    & \|f_k(\tau, y_1(\cdot), u, v) - f_k(\tau, y_2(\cdot), u, v)\|
                    + |\chi_k(\tau, y_1(\cdot), u, v) - \chi_k(\tau, y_2(\cdot), u, v)| \\
                    & \leq \lambda_\ast(\tau)
                    \|y_1(\cdot \wedge \tau) - y_2(\cdot \wedge \tau)\|_\infty
                \end{aligned}
            \end{equation}
            and
            \begin{equation} \label{C.45_f_chi_k^ast}
                \begin{aligned}
                    \|f_k(\tau, y(\cdot), u, v)\|
                    & \leq c_\ast(\tau) (1 + \|y(\cdot \wedge \tau)\|_\infty), \\
                    |\chi_k(\tau, y(\cdot), u, v)|
                    & \leq m_\ast(\tau) (1 + \|y(\cdot \wedge \tau)\|_\infty)
                \end{aligned}
            \end{equation}
            for all $\tau \in [0, T]$, $y(\cdot)$, $y_1(\cdot)$, $y_2(\cdot) \in C([- h, T], \mathbb{R}^n)$, $u \in P$, $v \in Q$.
            Recalling that the set $F_k$ is compact and the restrictions of the functions $\lambda_\ast(\cdot)$, $c_\ast(\cdot)$, $m_\ast(\cdot)$ to $F_k$ are continuous, we conclude that inequalities \eqref{lemma_DG_Lip} and \eqref{lemma_DG_growth} with $(f, \chi) = (f_k, \chi_k)$ and $\Lambda = \Lambda_k \coloneq \max_{\tau \in F_k} \max\{ \lambda_\ast(\tau), c_\ast(\tau) + m_\ast(\tau)\}$ hold for all $\tau \in [0, T]$, $y(\cdot)$, $y_1(\cdot)$, $y_2(\cdot) \in C([- h, T], \mathbb{R}^n)$, $u \in P$, $v \in Q$.
            Thus, the mapping $(f_k, \chi_k)$ satisfies the assumptions of Lemma \ref{lemma_DG}.

        \smallskip

        {\it Step 2.}
            For every $k \in \mathbb{N}$, consider the differential game \eqref{system}, \eqref{cost_functional} with $(f, \chi) = (f_k, \chi_k)$ and define the corresponding lower game value $\rho_k^-(t, x(\cdot))$.
            Let us prove that equality \eqref{rho_k^-_to_rho^-} takes place.

            Let $\varepsilon > 0$.
            Choose $R > 0$ such that $\|y(\cdot)\|_\infty \leq R$ for all $y(\cdot) \in Y$.
            Define the compact set $Y(t, x(\cdot); c_\ast(\cdot))$.
            By condition (C.6), there exists $\zeta > 0$ such that $|\sigma(y_1(\cdot)) - \sigma(y_2(\cdot))| \leq \varepsilon / 2$ for all $y_1(\cdot)$, $y_2(\cdot) \in Y(t, x(\cdot); c_\ast(\cdot))$ satisfying the inequality $\|y_1(\cdot) - y_2(\cdot)\|_\infty \leq \zeta$.
            Taking into account that $\mu([0, T] \setminus F_k) \to 0$ as $k \to \infty$, choose $k_\ast \in \mathbb{N}$ such that, for every $k \geq k_\ast$,
            \begin{align*}
                2 (1 + R) e^{\|\lambda_\ast(\cdot)\|_1}
                \int_{[0, T] \setminus F_k} c_\ast(\xi) \, \rd \xi
                & \leq \zeta, \\
                2 (1 + R) \biggl( \int_{[0, T] \setminus F_k} m_\ast(\tau) \, \rd \tau
                + \|\lambda_\ast(\cdot)\|_1 e^{\|\lambda_\ast(\cdot)\|_1}
                \int_{[0, T] \setminus F_k} c_\ast(\tau) \, \rd \tau \biggr)
                & \leq \frac{\varepsilon}{2}.
            \end{align*}
            Let $k \geq k_\ast$, $u(\cdot) \in \mathcal{U}[t, T]$, $v(\cdot) \in \mathcal{V}[t, T]$, and let $y(\cdot) \coloneq y(\cdot; t, x(\cdot), u(\cdot), v(\cdot))$ be the motion of system \eqref{system} and $y_k(\cdot) \coloneq y_k(\cdot; t, x(\cdot), u(\cdot), v(\cdot))$ be the motion of system \eqref{system} with $f = f_k$.
            Note that $y(\cdot) \in Y \subset Y(t, x(\cdot); c_\ast(\cdot))$, $y_k(\cdot) \in Y(t, x(\cdot); c_\ast(\cdot))$, and (see also \eqref{f_k_bar_f_k})
            \begin{equation*}
                f(\xi, y(\cdot), u(\xi), v(\xi))
                = \bar{f}_k(\xi, y(\cdot), u(\xi), v(\xi))
                = f_k(\xi, y(\cdot), u(\xi), v(\xi))
            \end{equation*}
            for all $\xi \in [t, T] \cap F_k$.
            Then, using \eqref{C.3_f_chi_k^ast} and \eqref{C.45_f_chi_k^ast}, we obtain
            \begin{align*}
                & \|y(\tau) - y_k(\tau)\| \\
                & \leq \int_{t}^{\tau} \| f(\xi, y(\cdot), u(\xi), v(\xi))
                - f_k(\xi, y_k(\cdot), u(\xi), v(\xi))\| \, \rd \xi \\
                & \leq \int_{t}^{\tau} \| f(\xi, y(\cdot), u(\xi), v(\xi))
                - f_k(\xi, y(\cdot), u(\xi), v(\xi))\| \, \rd \xi \\
                & \quad + \int_{t}^{\tau} \| f_k(\xi, y(\cdot), u(\xi), v(\xi))
                - f_k(\xi, y_k(\cdot), u(\xi), v(\xi))\| \, \rd \xi \\
                & \leq \int_{[t, \tau] \setminus F_k} \| f(\xi, y(\cdot), u(\xi), v(\xi))\| \, \rd \xi
                + \int_{[t, \tau] \setminus F_k} \| f_k(\xi, y(\cdot), u(\xi), v(\xi))\| \, \rd \xi \\
                & \quad + \int_{t}^{\tau} \lambda_\ast(\xi) \|y(\cdot \wedge \xi) - y_k(\cdot \wedge \xi)\|_\infty \, \rd \xi \\
                & \leq \int_{t}^{\tau} \lambda_\ast(\xi) \|y(\cdot \wedge \xi) - y_k(\cdot \wedge \xi)\|_\infty \, \rd \xi
                + 2 (1 + R) \int_{[0, T] \setminus F_k} c_\ast(\xi) \, \rd \xi
            \end{align*}
            for all $\tau \in [t, T]$.
            Therefore, based on the Gronwall inequality, we derive
            \begin{equation*}
                \|y(\cdot \wedge \tau) - y_k(\cdot \wedge \tau)\|_\infty
                \leq 2 (1 + R) e^{\|\lambda_\ast(\cdot)\|_1}
                \int_{[0, T] \setminus F_k} c_\ast(\xi) \, \rd \xi
            \end{equation*}
            for all $\tau \in [t, T]$.
            In particular, $\|y(\cdot) - y_k(\cdot)\|_\infty \leq \zeta$, which implies the inequality $|\sigma(y(\cdot)) - \sigma(y_k(\cdot))| \leq \varepsilon / 2$.
            In a similar way, we have
            \begin{align*}
                & \biggl| \int_{t}^{T} \chi(\tau, y(\cdot), u(\tau), v(\tau)) \, \rd \tau
                - \int_{t}^{T} \chi_k(\tau, y_k(\cdot), u(\tau), v(\tau)) \, \rd \tau \biggr| \\
                & \leq \int_{[t, T] \setminus F_k} | \chi(\tau, y(\cdot), u(\tau), v(\tau))| \, \rd \tau
                + \int_{[t, T] \setminus F_k} | \chi_k(\tau, y(\cdot), u(\tau), v(\tau))| \, \rd \tau \\
                & \quad + \int_{t}^{T} \lambda_\ast(\tau) \|y(\cdot \wedge \tau) - y_k(\cdot \wedge \tau)\|_\infty \, \rd \tau \\
                & \leq 2 (1 + R) \int_{[0, T] \setminus F_k} m_\ast(\tau) \, \rd \tau
                + 2 (1 + R) \|\lambda_\ast(\cdot)\|_1 e^{\|\lambda_\ast(\cdot)\|_1}
                \int_{[0, T] \setminus F_k} c_\ast(\tau) \, \rd \tau \\
                & \leq \varepsilon / 2.
            \end{align*}
            Thus, inequality \eqref{J_k} is satisfied for all $u(\cdot) \in \mathcal{U}[t, T]$ and $v(\cdot) \in \mathcal{V}[t, T]$, which yields \eqref{rho_k^--rhi^-} and completes the proof of \eqref{rho_k^-_to_rho^-}.

        \smallskip

        {\it Step 3.}
            For every $k \in \mathbb{N}$, consider the Hamiltonian $H_k^-$ given by \eqref{H_k^-}.
            Note that $H_k^-$ satisfies conditions (A.1)--(A.5) of Assumption \ref{assumption_H_sigma} and condition (A.4) is fulfilled with $c_{H_k^-}(\cdot) = c_\ast(\cdot)$ owing to the first inequality in \eqref{C.45_f_chi_k^ast}.
            Denote by $\varphi_k^-$ the minimax solution of the Cauchy problem \eqref{HJ}, \eqref{boundary_condition} with $H = H_k^-$.
            By Lemma \ref{lemma_DG} (see also Corollary \ref{cor:cH_independence}), equality \eqref{rho_k^-=varphi_k^-} takes place.
            Thus, $\varphi_k^-(t, x(\cdot)) \to \rho^-(t, x(\cdot))$ as $k \to \infty$ thanks to \eqref{rho_k^-_to_rho^-}, and it remains to show that $\varphi_k^-(t, x(\cdot)) \to \varphi^-(t, x(\cdot))$ as $k \to \infty$.

            Let $F_\ast$ be the union of the sets $F_k$ over $k \in \mathbb{N}$.
            Note that $\mu(F_\ast) = T$.
            Consider a mapping $(f_\ast, \chi_\ast) \colon [0, T] \times C([- h, T], \mathbb{R}^n) \times P \times Q \to \mathbb{R}^n \times \mathbb{R}$ defined for all $y(\cdot) \in C([- h, T], \mathbb{R}^n)$, $u \in P$, $v \in Q$ by
            \begin{equation*}
                \begin{aligned}
                    f_\ast^i(\tau, y(\cdot), u, v)
                    & \coloneq \max_{z(\cdot) \in Y} \bigl( f^i(\tau, z(\cdot), u, v)
                    - \lambda_{f, \chi}(\tau) \|z(\cdot \wedge \tau) - y(\cdot \wedge \tau)\|_\infty \bigr), \\
                    \chi_\ast(\tau, y(\cdot), u, v)
                    & \coloneq \max_{z(\cdot) \in Y} \bigl( \chi(\tau, z(\cdot), u, v)
                    - \lambda_{f, \chi}(\tau) \|z(\cdot \wedge \tau) - y(\cdot \wedge \tau)\|_\infty \bigr)
                \end{aligned}
            \end{equation*}
            if $\tau \in F_\ast$ and
            \begin{equation*}
                f_\ast(\tau, y(\cdot), u, v)
                \coloneq 0,
                \quad \chi_\ast(\tau, y(\cdot), u, v)
                \coloneq 0
            \end{equation*}
            otherwise.
            Here, $f_\ast^i(\tau, y(\cdot), u, v)$ (respectively, $f^i(\tau, z(\cdot), u, v)$) is the $i$-th coordinate of the vector $f_\ast(\tau, y(\cdot), u, v)$ (respectively, $f(\tau, z(\cdot), u, v)$) and $i \in \overline{1, n}$.
            By construction, we have
            \begin{equation} \label{f_ast_f}
                f_\ast(\tau, y(\cdot), u, v)
                = f(\tau, y(\cdot), u, v),
                \quad \chi_\ast(\tau, y(\cdot), u, v)
                = \chi(\tau, y(\cdot), u, v)
            \end{equation}
            for all $\tau \in F_\ast$, $y(\cdot) \in Y$, $u \in P$, $v \in Q$.
            In addition, for every $k \in \mathbb{N}$, taking \eqref{bar_f_k_chi_k} and \eqref{f_k_bar_f_k} into account, we obtain
            \begin{equation} \label{f_ast_f_k}
                f_\ast(\tau, y(\cdot), u, v)
                = f_k(\tau, y(\cdot), u, v),
                \quad \chi_\ast(\tau, y(\cdot), u, v)
                = \chi_k(\tau, y(\cdot), u, v)
            \end{equation}
            for all $\tau \in F_k$, $y(\cdot) \in C([- h, T], \mathbb{R}^n)$, $u \in P$, $v \in Q$.
            This implies that the mapping $(f_\ast, \chi_\ast)$ satisfies conditions (C.1) and (C.2).
            Moreover, owing to \eqref{C.3_f_chi_k^ast} and \eqref{C.45_f_chi_k^ast}, we have
            \begin{align*}
                & \|f_\ast(\tau, y_1(\cdot), u, v) - f_\ast(\tau, y_2(\cdot), u, v)\|
                + |\chi_\ast(\tau, y_1(\cdot), u, v) - \chi_\ast(\tau, y_2(\cdot), u, v)| \\
                & \leq \lambda_\ast(\tau)
                \|y_1(\cdot \wedge \tau) - y_2(\cdot \wedge \tau)\|_\infty
            \end{align*}
            and
            \begin{equation} \label{C.45_f_chi_ast}
                \begin{aligned}
                    \|f_\ast(\tau, y(\cdot), u, v)\|
                    & \leq c_\ast(\tau) (1 + \|y(\cdot \wedge \tau)\|_\infty), \\
                    |\chi_\ast(\tau, y(\cdot), u, v)|
                    & \leq m_\ast(\tau) (1 + \|y(\cdot \wedge \tau)\|_\infty)
                \end{aligned}
            \end{equation}
            for all $\tau \in F_\ast$, $y(\cdot)$, $y_1(\cdot)$, $y_2(\cdot) \in C([- h, T], \mathbb{R}^n)$, $u \in P$, $v \in Q$.
            Therefore, the mapping $(f_\ast, \chi_\ast)$ satisfies conditions (C.3)--(C.5).

            Consider a Hamiltonian $H_\ast^- \colon [0, T] \times C([- h, T], \mathbb{R}^n) \times \mathbb{R}^n \to \mathbb{R}$ defined for any $y(\cdot) \in C([- h, T], \mathbb{R}^n)$ and $s \in \mathbb{R}^n$ by
            \begin{equation*}
                H_\ast^-(\tau, y(\cdot), s)
                \coloneq \max_{v \in Q} \min_{u \in P} \bigl( \langle s, f_\ast(\tau, y(\cdot), u, v) \rangle
                - \chi_\ast(\tau, y(\cdot), u, v) \bigr)
            \end{equation*}
            if $\tau \in F_\ast$ and by
            \begin{equation*}
                H_\ast^-(\tau, y(\cdot), s)
                \coloneq 0
            \end{equation*}
            otherwise.
            Let $\varphi_\ast^-$ be the minimax solution of the Cauchy problem \eqref{HJ}, \eqref{boundary_condition} with $H = H_\ast^-$.
            Owing to \eqref{f_ast_f}, for any $\tau \in F_\ast$, $y(\cdot) \in Y$, $s \in \mathbb{R}^n$,
            \begin{equation*}
                H^-(\tau, y(\cdot), s)
                = H_\ast^-(\tau, y(\cdot), s),
            \end{equation*}
            and, therefore, $\varphi^- (t, x(\cdot)) = \varphi^-_\ast(t, x(\cdot))$ by Corollary \ref{corollary_uniqueness_at_point}.

            Let a compact set $D \subset C([- h, T], \mathbb{R}^n)$ and $s \in \mathbb{R}^n$ be fixed.
            For any $k \in \mathbb{N}$ and $y(\cdot) \in D$, taking into account that $H_k^-(\tau, y(\cdot), s) = H_\ast^-(\tau, y(\cdot), s)$ for all $\tau \in F_k$ by \eqref{f_ast_f_k}, $H_k^-(\tau, y(\cdot), s) = 0$ for all $\tau \in [0, T] \setminus F_k$ by \eqref{f_k_bar_f_k_2}, and
            \begin{align*}
                |H_\ast^-(\tau, y(\cdot), s)|
                & \leq |H_\ast^-(\tau, y(\cdot), s) - H_\ast^-(\tau, y(\cdot), 0)|
                + |H_\ast^-(\tau, y(\cdot), 0)| \\
                & \leq (c_\ast(\tau) \|s\| + m_\ast(\tau)) (1 + \max_{z(\cdot) \in D} \|z(\cdot)\|_\infty)
            \end{align*}
            for a.e. $\tau \in [0, T]$ by \eqref{C.45_f_chi_ast}, we obtain
            \begin{equation*}
                \begin{aligned}
                    & \int_{0}^{T} |H_k^-(\tau, y(\cdot), s) - H_\ast^-(\tau, y(\cdot), s)| \, \rd \tau \\
                    & \leq \int_{[0, T] \setminus F_k} |H_k^-(\tau, y(\cdot), s) | \, \rd \tau
                    + \int_{[0, T] \setminus F_k} |H_\ast^-(\tau, y(\cdot), s)| \, \rd \tau \\
                    & \leq (1 + \max_{z(\cdot) \in D} \|z(\cdot)\|_\infty)
                    \int_{[0, T] \setminus F_k} (c_\ast(\tau) \|s\| + m_\ast(\tau)) \, \rd \tau.
                \end{aligned}
            \end{equation*}
            Hence, and due to the convergence $\mu([0, T] \setminus F_k) \to 0$ as $k \to \infty$, we conclude that the first equality in \eqref{sigma_H_convergence} with $H = H_\ast^-$ and $H_k = H_k^-$ is valid.

            Consequently, applying Theorem \ref{theorem_continuous_dependence}, we get $\varphi_k^-(t, x(\cdot)) \to \varphi_\ast^-(t, x(\cdot))$ as $k \to \infty$, which completes the proof.
    \end{proof}

    Theorem \ref{theorem_DG_lower_upper} and Corollary \ref{corollary_uniqueness_at_point} immediately imply the following result.
    \begin{theorem} \label{theorem_game_value}
        Under Assumptions {\rm \ref{assunption_DG}} and {\rm \ref{assumption_saddle_point}}, the differential game \eqref{system}, \eqref{cost_functional} has a value, and the value functional $\rho$ coincides with the minimax solution $\varphi$ of the Cauchy problem \eqref{HJ}, \eqref{boundary_condition} with $H = H^-$ {\rm(}or with $H = H^+${\rm)}.
    \end{theorem}

\bmhead{Acknowledgements}

This work is supported by RSF grant 25-11-00269, \\ \href{https://rscf.ru/en/project/25-11-00269/}{https://rscf.ru/en/project/25-11-00269/}.

\end{document}